\documentclass[11pt]{article}%
\usepackage[pdftex]{graphicx}
\usepackage{epstopdf}
\usepackage{latexsym,amsmath,amsthm,amsfonts,amscd,graphicx,pstricks,pst-plot}
\usepackage{indentfirst}
\usepackage{amsmath}
\usepackage{amsfonts}
\usepackage{epsfig}
\usepackage{amssymb}
\usepackage{graphicx}
\usepackage{setspace}

\newtheorem{thm}{Theorem}[section]
\newtheorem{cor}[thm]{Corollary}
\newtheorem{lem}[thm]{Lemma}

\newtheorem{prop}[thm]{Proposition}

\newcommand{\bea}{\begin{eqnarray}}
\newcommand\numberthis{\addtocounter{equation}{1}\tag{\theequation}}
\newcommand{\eea}{\end{eqnarray}}
\newcommand{\beast}{ \begin{eqnarray*} }
\newcommand{\eeast}{ \end{eqnarray*} }

\newcommand{\abs}[1]{\left\vert#1\right\vert}

\begin{document}

\title{A collision-time estimate for diffusion and Ornstein-Uhlenbeck particles}

\author{Guolong Li}

\date{}
\maketitle

\section*{Abstract}
We consider two different models for colloidal particles. In the first model, we consider their free motions to be diffusions while in the second model we take them to be integrated Ornstein-Uhlenbeck processes. In both models, we derived collision estimates for pairs of particles. In particular, we found that these estimates would be different to the Brownian case even when the free motions of the particles are Brownian at macroscopic scales. As a consequence, the coagulation kernels and diffusivities in the coagulation-diffusion equations would also be affected accordingly. We then proved that there exists a unique solution to the coagulation-diffusion equations in these cases under physically reasonable assumptions. 
\section{Introduction}
In this paper, we investigate the dynamics for large clouds of colloidal particles whose free motions
are Brownian at macroscopic scales, but are not Brownian on the scale of the particles themselves.
The parameters of our processes will be chosen so that each particle follows a free path for a time
of order one between collisions. The macroscopic free path may therefore be considered as Brownian.
However the actual collisions, being determined by microscopic dynamics, will be strongly affected
by the departure of the free motion from the Brownian case.

We make two new contributions. The first is to prove collision estimates for pairs of particles. Having
in mind eventually a system of $N$ particles, where $N$ is large, we choose the scale of the particle radius 
so that any given pair meets in order one time with probability of order $1/N$.
The mass and radius of each particle will also affect the characteristics of its free motion, in a way
we shall take as given, based on some physical arguments.

We prove collision estimates in two cases. In the first case, the particle, in addition to its basic
molecular diffusivity, is considered as suspended in an incompressible fluid, through which it acquires a drift, 
which we shall take to be periodic.
In the second case, following the derivation of physical Brownian motion from particle dynamics, we suppose
that the free motion is an integrated Ornstein--Uhlenbeck process.
In both cases, it is well understood that, under appropriate scalings, the macroscopic motions are Brownian, with diffusivities depending on the sizes of the particles. Our analysis shows how the small-scale motions, in both
cases, lead to strong departures from the Brownian case for the collision probabilities.
More precisely, consider two particles in $\mathbb{R}^d$ with $d\geq 3$, having radius $r$, evolving under dynamics which is approximately Brownian with constant diffusivity $a$, but where a departure from Brownian behaviour is visible on a length scale of order $\lambda$. We investigate the collision event in the limit $r,\lambda\rightarrow 0$. It is intuitive to expect that for two particles starting from $x_1$, $x_2$ colliding at $X$ at time $T$,
\begin{equation}\label{eqnes}
 \mathbb{P}(T\in dt, X\in dx)\sim p(t,x_1,x)p(t,x_2,x)k(r,\lambda)dtdx,
\end{equation}
where $p$ is the transition density of a standard Brownian motions and $k(r,\lambda)$ for small $r,\lambda$ is to be determined.

In \cite{bmk}, Norris has considered the case where the dynamics of the particles are exactly Brownian and showed that $k(r,0)=c_dar^{d-2}$ for some constant $c_d$. Therefore, in the case where $\lambda\ll r$, we would expect that $k(r,\lambda)\approx c_dar^{d-2}$. However, when $\lambda\gg r$, we will give two cases where we can show that the non-Brownian microscopic dynamics leads to different rates for coagulation.

We will first look at the case of diffusion at rate $a$ enhanced by a $\lambda$-periodic drift $b^{\lambda}(x)$. In \cite{convection}, Fannjiang and Papanicolaou showed that when $b^{\lambda}(x)=\frac{b(x/\lambda)}{\lambda}$ for some 1-periodic (i.e. $b(x+x')=b(x)$ for any integer point $x'$) divergence-free zero-mean $b$ then the underlying motion converges weakly to a Brownian motion with diffusivity $\bar{a}$, which in general does not equal to $a$. We will see in Corollary \ref{period} that when $\bar{a}$ and $a$ are both scalars 
\begin{displaymath}
\lim_{r\rightarrow 0}\lim_{\lambda\rightarrow 0}k(r,\lambda)r^{2-d}=c_da
\end{displaymath} 
while
\begin{displaymath}
\lim_{\lambda\rightarrow 0}\lim_{r\rightarrow 0}k(r,\lambda)r^{2-d}=c_d\bar{a}.
\end{displaymath}

Then we will look at the case where the motions of the particles are modelled by integrated Ornstein-Uhlenbeck processes, and make appropriate scaling so that their motions converge to Brownian motions. We will see in Theorem \ref{ou} that when $r\ll\lambda$, $k(r,\lambda)\sim f(\lambda)r^{d-1}$ for some function $f$. So, we will have that $k(r,\lambda)\ll k(r,0)$ in this case. Intuitively this is because, in the Ornstein-Uhlenbeck case, when two particles come close to each other, they are likely to get far away again with almost constant speed so that their trajectories are almost straight lines, while in the Brownian case, the particles are likely to move back and forth more before they go away from each other and this results more chance for them to collide.

Our second contribution is to the theory of coagulation-diffusion equations. This was motivated by the mass-dependent diffusivities and collision probabilities emerging in the first part of the paper, to which prior work on coagulation-diffusion equations did not apply. We show in Sections 5 and 6 an existence and uniqueness result which does apply for the diffusivities and collision probabilities associated with the Ornstein--Uhlenbeck case.

Consider a large cloud of colloidal particles of $N$ particles in which when two particles collide, they coagulate and continue the random motion as a larger particle. As the $N$ becomes large, the distribution of the particles is expected to converge to the solution of the coagulation-diffusion equations
\begin{equation}\label{eqncd}
\dot{\mu}(x,dy)=\frac{1}{2}a(y)\Delta_x\mu_t(x,dy)+K^+(\mu_t)(x,dy)-K^-(\mu_t)(x,dy),
\end{equation}
where 
\begin{displaymath}
K^+(\mu)(x,A)=\frac{1}{2}\int_0^\infty\int_0^\infty \mathbf{1}_{y+y'\in A}K(y,y')\mu(x,dy)\mu(x,dy'),
\end{displaymath}
\begin{displaymath}
K^-(\mu)(x,A)=\int_{y\in A}\int_0^\infty K(y,y')\mu(x,dy)\mu(x,dy').
\end{displaymath}
Here, $y$ represents the mass of the particles and we assume their macroscopic free motion are approximately Brownian motions with diffusivity $a(y)$. Further, $\mu_t$ is a kernel on $\mathbb{R}^d\times (0,\infty)$ with $d\geq 3$. In this context, for a measurable set $A\subseteq (0,\infty)$, $\mu_t(x,A)$ represents the density of particles of masses within the set $A$ at position $x$ at time $t$. So, $\frac{1}{2}a(y)\Delta_x\mu_t(x,dy)$ represents the rate of change of $\mu$ due to diffusions. Moreover, the coagulation kernel $K:(0,\infty)\times (0,\infty)\rightarrow (0,\infty)$ is a measurable function. Intuitively, we can think $K(y,y')$ as the rate at which a particle of mass $y$ and a particle of mass $y'$ coagulate and form a particle of mass $y+y'$ when the two particles are at same position. Thus, $K^+$ represents the rate at which new particles are created due to coagulations and $K^-$ represents the rate at which particles are lost due to coagulations. We denote $K=K^+-K^-$ to represent rate of change of the particles due to coagulations. The convergence of the $N$-particles system remains an open problem in general. See \cite{HR}, \cite{HR1} and \cite{cluster} for related work. We now explain the connection in a heuristic level. (\ref{eqnes}) can be generalised in the case when two particles are different. For two particles of masses $y_1$ and $y_2$ starting at $x_1$ and $x_2$ respectively, colliding at $X$ at time $T$, we have
\begin{displaymath}
 \mathbb{P}(T\in dt, X\in dx)\sim p_1(t,x_1,X)p_2(t,x_2,X)K(y_1,y_2)dtdx,
\end{displaymath}
where $p_1$ and $p_2$ are the transition densities of the particles. We now give an interpretation of this. For a small region $dx$, but still large relative to the sizes of the particles, we take time interval $dt$ sufficiently small so that when the particles have been in $dx$ during $dt$, they are almost certain to be in $dx$ during the entire $dt$. Then the probability that the particles collide in $dx$ during $dt$ is the probability they are both in $dx$ during $dt$ multiplied by $Kdt/dx$. Or, we can say that when the two particles are in $dx$ during $dt$, they have probability $Kdt/dx$ to collide. Now, for large $N$, suppose we can scale the particles' masses to be $f^N(y)$ for some function $f^N$ such that $K(y_1,y_2)=NK(f^N(y_1),f^N(y_2))$. We also approximate the number of particles of mass $f^N(y)$ in $dx$ at time $t$ to be $N\mu_t(x, y)dx$, which is still large. Then for a particle of mass $f^N(y_1)$ in $dx$ during $dt$, the total probability it collides there with a particle of mass $f^N(y_2)$ will be $N\mu_t(x,y_2)K(f^N(y_1),f^N(y_2))dt=\mu_t(x,y_2)K(y_1,y_2)dt$. Therefore, the expected total number of coagulations between particles of masses $y_1$ and $y_2$ there will be $N\mu_t(x,y_1)\mu_t(x,y_2)K(y_1,y_2)dxdt$. As $N\rightarrow\infty$, by law of large numbers, we can approximate the number of these coagulations to be $N\mu_t(x,y_1)\mu_t(x,y_2)K(y_1,y_2)dxdt+o(Ndxdt)$ and this means that the coagulations contribute to a loss of $\mu_t(x,y_1)\mu_t(x,y_2)K(y_1,y_2)dxdt+o(dxdt)$ to $\mu(x,y_1)dx$ and $\mu(x,y_2)dx$ and a gain of the same amount to $\mu(x,y_1+y_2)dx$ during $dt$. Integrating over $y_1$ and $y_2$ explains the form $K^+$ and $K^-$ in the coagulation-diffusion equations. Note that, although we can think $N$ as the number of particles, we do not really need to require $N$ to be integer in our analysis. Also, $K(y_1,y_2)$ doesn't need to represent exactly the coagulation rate between particles of sizes $y_1$ and $y_2$, we only need $K(f^N(y_1),f^N(y_2))$ to represent the coagulation rate between particles of sizes $f^N(y_1)$ and $f^N(y_2)$ in the limit as $N\rightarrow\infty$. However, this argument is only heuristic. We know for fixed $dx$, $dt$ and $y_1,y_2$ the number of coagulations between particles of masses $y_1$ and $y_2$ in $dxdt$ can be approximated by $N\mu_t(x,y_1)\mu_t(x,y_2)K(y_1,y_2)dxdt+o(Ndxdt)$, but the $o(Ndxdt)$ term depends on $y_1,y_2$ and $\mu$. In particular, if we fix a large $N$ and look at the distribution of sufficiently large particles, then we might have  a rather large error when using law of large numbers because there are not many large particles. Moreover, although there are not many large particles, we might not ignore their influence on the system of particles because they might coagulate fast and grow quickly. For this reason, it is difficult to show rigorously that the distribution of particles actually converges to the coagulation-diffusion equations in general.

Note that (\ref{eqncd}) only makes sense if $\mu_t(x,y)$ is twice differentiable in $x$. However, the equation can be reformulated to make sense without prior assumptions on $\mu_t$. We follow Norris \cite{BM}. Define $p^{t,x',x}(y)=(2\pi a(y)t)^{-d/2}e^{\frac{-|x'-x|^2}{2a(y)t}}$ and
\begin{displaymath}
P_t\mu(x,dy)=\int_{\mathbb{R}^d}\mu(x',dy)p^{t,x',x}(y)dx'.
\end{displaymath}
If we have a Brownian particle with diffusivity $a(y)$ starting at $x'$, then $p^{t,x',x}(y)$ is the probability density that the particle is at $x$ at time $t$. Norris then reformulated the Smoluchowski coagulation equation to be
\begin{equation}\label{eqn}
\mu_t+\int_0^t P_{t-s}K^-(\mu_s)ds=P_t\mu_0+\int_0^t P_{t-s}K^+(\mu_s)ds.
\end{equation}
This equation and some variants of it have been considered in several prior works. Many of them considered function solutions in the discrete case, i.e. $\mu_t(x,y)=\sum_{m=1}^\infty f^m_t(x)\delta_{m}(dy)$, see \cite{reg,disc,wro1,wro2,wro3}. We will restrict our review on the existence to works addressing the continuous case. In \cite{coafra}, Amann proved local existence and uniqueness in a general setting, assuming uniform bounds on diffusivities and coagulation rates and uniform positivity of the diffusivities. Later, in \cite{locglob}, Amann and Walker proved global existence for small initial data under similar hypotheses. In \cite{continuous}, Lauren{\c{c}}ot and Mischler showed the global existence when $a,\frac{1}{a}$ and $K$ are all bounded on compacts and the coagulation kernel satisfies the Galkin-Tupchiev monotonicity condition 
\begin{displaymath}
K(y_1,y_2)\leq K(y_1+y_2,y_1)
\end{displaymath} 
along with the growth bound 
\begin{displaymath}
\lim_{y'\rightarrow\infty}\sup_{y\leq R}\frac{K(y,y')}{y'}=0. 
\end{displaymath}
If we assume further that masses of all particles are uniformly positive, then Mischler and Rodriguez Richard showed in \cite{mono} that the monotonicity condition can be weakened by 
\begin{displaymath}
K(y_1,y_2)\leq K(y_1+y_2,y_1)+K(y_1+y_2,y_2)
\end{displaymath}
in the context of coagulation-diffusion in a bounded domain in $\mathbb{R}^3$. 

In \cite{Carr}, Ball and Carr noted that in the spatially homogeneous setting, the questions of uniqueness and mass conservation for coagulation equations are related to the existence of moment bounds for solutions. In \cite{moment} and \cite{fra}, Rezakhanlou and Hammond obtained suitable moment bounds for solutions under assumptions including that the diffusivity $a$ is positive, uniformly bounded and non-increasing, and that the coagulation kernel $K$ satisfies
\begin{displaymath}
\sup_{y,y'}\frac{K(y,y')}{yy'}<\infty
\end{displaymath}
and
\begin{displaymath}
\lim_{y+y'\rightarrow\infty}\frac{K(y,y')}{(y+y')(a(y)+a(y'))}\rightarrow 0.
\end{displaymath}
In \cite{rez}, Rezakhanlou has shown that the non-increasing condition on the diffusivities can be relaxed to some extent. In \cite{BM}, Norris assumed that $K(y,y')\leq w(y)w(y')$ for some sublinear function $w$ and gave a proof for the existence and uniqueness in the case requiring $a^{-\frac{d}{2}}w$ to be sublinear.

If we assume the microscopic free motions of the particles are Ornstein-Uhlenbeck processes in $\mathbb{R}^3$ satisfying the Einstein-Stokes relation. Then we can show that under certain scaling limit, an appropriate choice for $K$ would be
\begin{displaymath}
K(y_1,y_2)=(y_1^{\frac{1}{3}}+y_2^{\frac{1}{3}})^2\sqrt{\frac{1}{y_1}+\frac{1}{y_2}},
\end{displaymath}
and the diffusivity $a(y)=y^{-\frac{1}{3}}$. If we use the result from \cite{mechanical} instead of Einstein-Stokes relation, we will get the same coagulation kernel but $a(y)=y^{-\frac{2}{3}}$. If we fix $y'$, we see that $K(y)\sim y^{2/3}$ for large $y$. Therefore, in either of the cases, $a^{-\frac{d}{2}}w$ cannot be sublinear and $\frac{K(y,y')}{(y+y')(a(y)+a(y'))}$ does not converge to $0$. So, we can not directly apply the results in \cite{moment, BM, fra, rez} to obtain the existence and uniqueness of the solution. In this thesis, we will give criteria for the existence and uniqueness of solution to the Smoluchowski coagulation equations which work in these two cases.
Further, as an extension, we consider the corresponding Smoluchowski coagulation equations when the particles are subject to a position and mass dependent drift in addition to their basic diffusivityies. We will also give natural criteria for the existence and uniqueness of the solution to the Smoluchowski coagulation equations in this case.
\section{Main results}
\subsection{Collision estimates}
Consider two particles in $\mathbb{R}^d$ of radii $r_1N^{-1/(d-2)}$ and $r_2N^{-1/(d-2)}$ starting at $x_1$ and $x_2$. The reason for the scaling term $N^{-1/(d-2)}$ is that if we consider a system of $N$ particles, this scaling turns out to make the number of collisions happening per unit time to be of order $N$ and thus the rate at which a particle collides is of order $1$. To see this, we consider for simplicity the case where the two particles are independent Brownian. For $i=1,2$, let $X^i$ be the position of particle $i$, then the probability that the two particles will ever collide is $\frac{|x_1-x_2|^{2-d}}{|r_1+r_2|^{2-d}}$ by applying optional stopping theorem on the martingale $|X^1_t-X^2_t|^{2-d}$ stopped at collision time T. Moreover, it is reasonable to believe that the distributions of T conditionally on $T<\infty$ will be roughly the same for different large $N$.

Now, we let $X^i$ satisfy
\begin{displaymath}
 dX^i_t=\sqrt{a_i(X^i_t)}dB^i_t+b_i(X^i_t)dt,
\end{displaymath}
with $a_i$ being bounded H\"older continuous scalar functions, $b_i$ being bounded measurable functions and $B^i$ being independent standard Brownian motions. 
Let $p_i(s,x;t,y)=p_i(t-s,x,y)$ be the transition density of  particle $i$.  Now, set $T$ to be the first time when the two particles collide and $X(T)$ be the centre of mass of the two particles at time $T$. In \cite{bmk}, Norris has proved that if $b_i=0$ and $a_i$ are constants, then for any uniformly continuous bounded function $g$ supported on $[0,R)\times\mathbb{R}^d$ with $R>0$ we have
\begin{displaymath}
N\mathbb{E}[g(T,X(T))\mathbf{1}_{T<R}]\rightarrow K\int_0^R\int_{\mathbb{R}^d} p_1(0,x_1;s,z)p_2(0,x_2;s,z)g(s,z)dzds,
\end{displaymath}
as $N\rightarrow\infty$, where $K=c_d(a_1+a_2)(r_1+r_2)^{d-2}$ and 
\begin{displaymath}
 \frac{1}{c_d}=\int_0^\infty\frac{1}{(2\pi t)^{\frac{d}{2}}}e^{-\frac{1}{2t}}dt.
\end{displaymath}
This theorem essentially means that the probability that the two particles collide at $dz$ during time $ds$ is approximately $\frac{Kp_1(0,x_1;s,z)p_2(0,x_2;s,z)dzds}{N}$. In the Introduction, we have  explained that if we can scale the particle sizes such that $K(y_1,y_2)=NK(f^N(y_1),f^N(y_2))$, then we can hope that the evolution of the system of particles converges to the coagulation-diffusion equations. If we let $y_i$ be the mass of a particle of radius $r_i$ and $f^N(y_i)$ be the mass of a particle of radius $r_iN^{-1/(d-2)}$, then in this context, we have 
\begin{displaymath}
K(f^N(y_1),f^N(y_2))=\frac{K}{N}=\frac{K(y_1,y_2)}{N}
\end{displaymath}
which confirms our choice of scaling the radius with $N^{-1/(d-2)}$. We will generalize this result into the following theorem.
\begin{thm}\label{diff}
 For all $d\geq 3$ and $R\in[1,\infty)$ there is a constant $C$ depending only on $d$ and $R$ with the following property. Let $N\in(0,\infty)$, $x_i\in\mathbb{R}^d$ and $y_i,r_i\in[R^{-1},R]$, $i=1,2$, be given. For $i=1,2$, let $a_i:\mathbb{R}^d\rightarrow [R^{-1},R]$ be H\"older continuous functions and $b_i:\mathbb{R}^d\rightarrow \mathbb{R}^d$ be measurable with $|b_i(x)|\leq R$ for all $x\in\mathbb{R}^d$. Set $a(x)=a_1(x)+a_2(x)$, $r=r_1+r_2$ and $K(x)=c_da(x)r^{d-2}$ with
\begin{displaymath}
 \frac{1}{c_d}=\int_0^\infty\frac{1}{(2\pi t)^{\frac{d}{2}}}e^{-\frac{1}{2t}}dt.
\end{displaymath}
 For $i=1,2$, let $X^i$ be a diffusion in $\mathbb{R^d}$ satisfying
\begin{displaymath}
 dX^i_t=\sqrt{a_i(X^i_t)}dB^i_t+b_i(X^i_t)dt,
\end{displaymath}
\begin{displaymath}
 X^i_0=x_i,
\end{displaymath}
 with $B^1$, $B^2$ independent standard Brownian motions and $x_1\neq x_2$. Set $r_N=rN^{-1/(d-2)}$ and let
 \begin{displaymath}
 T=\inf\{t\geq 0: |X^1_t-X^2_t|\leq r_N\},\:X(T)=(y_1X^1_T+y_2X^2_T)/(y_1+y_2).
\end{displaymath}
For $i=1,2$, for $s,t\in\mathbb{R}$ and $x,z\in \mathbb{R}^d$, let $p_i(s,x;t,z)$ be the transition density of $X^i$. Let $1\geq\epsilon\geq 2r_N$ be given and let $g$ be a bounded measurable function on $[0,\infty)\times\mathbb{R}^d$, supported on $[0,R)\times\mathbb{R}^d$. Write $\|g\|$ for the uniform norm and set 
\begin{displaymath}
 \phi_g(\epsilon)=\sup_{|s-t|\leq\epsilon^2,|x-z|\leq\epsilon}|g(s,z)-g(t,x)|.
\end{displaymath}
Then
\begin{align*}
 &\left|N\mathbb{E}(g(T,X(T))\mathbf{1}_{T<R})-\int_0^R\int_{\mathbb{R}^d} K(z)p_1(0,x_1;s,z)p_2(0,x_2;s,z)g(s,z)dzds\right|\\
&\leq C\big[\epsilon^{2-d}\|g\|/N+\epsilon^2+\phi_g(\epsilon)\big](|x_1-x_2|)^{2-d}.
\end{align*}
In particular, when $g$ is uniformly continuous, by choosing $\epsilon=max(2rN^{-1/(d-2)},N^{-\frac{1}{2(d-2)}})$, say, we obtain
\begin{displaymath}
\left|(r_N)^{2-d}\mathbb{E}(g(T,X(T))\mathbf{1}_{T<R})-\int_0^R\int_{\mathbb{R}^d} c_da(z)p_1(0,x_1;s,z)p_2(0,x_2;s,z)g(s,z)dzds\right|\rightarrow 0.
\end{displaymath}
\end{thm}
So, now the probability that the two particles collide at $dz$ during time $ds$ is approximately $K(z)p_1(0,x_1;s,z)p_2(0,x_2;s,z)g(s,z)dzds$. A key difficulty in proving this theorem  with respect to the works of Norris is that we can not express $p$ explicitly. We will need to make estimations and bounds on $p$ and avoid the need of its explicit form to solve this difficulty. 
 
 As an application, we will investigate how Brownian particles coagulate under a periodic drift. We let the motion of the particles satisfy
\begin{displaymath}
 dX^\lambda_i(t)=\sqrt{a_i}dB_i(t)+b^{\lambda}_i(X^\lambda_i(t))dt,
\end{displaymath}
where $b^{\lambda}_i(x)=\frac{b_i(x/\lambda)}{\lambda}$ for some periodic divergence-free zero-mean $b$. In \cite{convection}, \cite{homogenization} and \cite{Norris}, they have shown that the underlying motion converges weakly to a Brownian motion with diffusivity $\bar{a}_i$ as $\lambda\rightarrow 0$ for some $\bar{a}_i$. We will assume that $a_i$ and $b_i$ are chosen such that both $a_i$ and $\bar{a}_i$ are scalars. A concrete example would be when $d=4$, for $j=1,2,3,4$, denote $b_i^j(x)$ the $j$th component of $b_i(x)$ in Cartesian coordinates and let $b_i^1(x)=\sin(x_1)\cos(x_2)$, $b_i^2(x)=-\sin(x_2)\cos(x_1)$, $b_i^3(x)=\sin(x_3)\cos(x_4)$ and $b_i^4(x)=-\sin(x_4)\cos(x_3)$. It has been shown that when $a_i$ is small, $\bar{a}_i$ will be approximately $c\sqrt{a_i}$ for some constant $c$.
\begin{cor}\label{period}
We will use same notation as in Theorem \ref{diff}. Let $\bar{a}=\bar{a}_1+\bar{a}_2$, then for any bounded continuous measurable function $g$ on $[0,\infty)\times\mathbb{R}^d$, supported on $[0,R)\times\mathbb{R}^d$, we have
\begin{align*}
 &\lim_{N\rightarrow\infty}\lim_{\lambda\rightarrow 0}\left|N\mathbb{E}(g(T,X^\lambda_T)\mathbf{1}_{T<R})-\int_0^R\int_{\mathbb{R}^d} \bar{K}p_1(0,x_1;s,z)p_2(0,x_2;s,z)g(s,z)dzds\right|\rightarrow 0
\end{align*}
and
\begin{align*}
 &\lim_{\lambda\rightarrow 0}\lim_{N\rightarrow\infty}\left|N\mathbb{E}(g(T,X^\lambda_T)\mathbf{1}_{T<R})-\int_0^R\int_{\mathbb{R}^d}Kp_1(0,x_1;s,z)p_2(0,x_2;s,z)g(s,z)dzds\right|\rightarrow 0,
\end{align*}
where $K=c_dar^{d-2}$ and $\bar{K}=c_d\bar{a}r^{d-2}$ and $p_1$ and $p_2$ are the transition densities of Brownian motions with diffusivities $\bar{a}_1$ and $\bar{a}_2$ respectively.
\end{cor}
The intuition behind this corollary is that when the particles' sizes are small but fixed and if we let $\lambda\rightarrow 0$, then we know that the motions of the two particles will converge to Brownian motions with diffusivities $\bar{a}_1$ and $\bar{a}_2$ respectively. Thus, we should expect that the distribution of the collision time and position of the two particles also converges to that of two Brownian particles with diffusivities $\bar{a}_1$ and $\bar{a}_2$ and thus the coagulation kernel will be $\bar{K}$. On the other hand, when $\lambda$ is fixed, and let $N\rightarrow\infty$, Theorem (\ref{diff}) says that the coagulation kernel depends only on the local diffusivities and equals to $K$.

Next, we will show analogous results for Ornstein-Uhlenbeck particles.
\begin{thm}\label{ou}
  For $d\geq 3$ and $i=1,2$, let $x_i\in\mathbb{R}^d$ and $y_i,\tau_i,b_i>0$ be given. Assume $x_1\neq x_2$. Further, for natural number $N$, let $V^N_i, X^N_i$ be Ornstein-Uhlenbeck velocity-position processes satisfying
\begin{align*}
 dV^N_i(t)&=Nb_idB^i_t-N\tau_iV^N_idt,\\
 dX^N_i(t)&=V^N_i(t)dt,\\
 V^N_i(0)&=0,\\
 X^N_i(0)&=x_i,
\end{align*}
 with $B^1$, $B^2$ independent standard Brownian motions. Let $r_N$ denote the sum of the radii of the two particles. 
Set
\begin{displaymath}
 T=\inf\{t\geq 0: |X^N_1(t)-X^N_2(t)|\leq r_N\},\:X(T)=(y_1X^N_1(T)+y_2X^N_1(T))/(y_1+y_2).
\end{displaymath}
Suppose $r_N<N^{-\alpha}$ for some $\alpha>\frac{1}{2}$. Let $g$ be a uniformly continuous and bounded function on $[0,\infty)\times\mathbb{R}^d$, supported on $[t_0,t_1]\times\mathbb{R}^d$ with $0<t_0<t_1$.
Then as $N\rightarrow\infty$
\begin{displaymath}
\left|N^{-\frac{1}{2}}(r_N)^{1-d}\mathbb{E}[g(T,X(T))]-c_d\sqrt{\frac{b_1^2}{\tau_1}+\frac{b_2^2}{\tau_2}}\int_{t_0}^{t_1}\int_{\mathbb{R}^d}q_1(0,x_1;t,z)q_2(0,x_2;t,z)g(t,z)dtdz\right|\rightarrow 0,
\end{displaymath}
where $q_i$ is the transition density for the $d$-dimensional Brownian motion with diffusivity $a_i=(\frac{b_i}{\tau_i})^2$ and $c_d$ is $\frac{1}{\sqrt{2}}$ times the product of the volume of a unit ball in $\mathbb{R}^{d-1}$ and the expected norm of a standard normal vector in $\mathbb{R}^d$. More explicitly, we have
\begin{displaymath}
c_d=\frac{\pi^{\frac{d-1}{2}}}{\Gamma(\frac{d}{2})},
\end{displaymath} 
where $\Gamma$ denotes the gamma function.
\end{thm}
\begin{thm}\label{brown}
Under the same setting as Theorem \ref{ou}, but suppose now that $r_N>N^{-\alpha}$ for some $\alpha<\frac{1}{2}$ and $r_N\rightarrow\ 0$ as $N\rightarrow\infty$. Let $g$ be a uniformly continuous and bounded function on $[0,\infty)\times\mathbb{R}^d$, supported on $[t_0,t_1]\times\mathbb{R}^d$ with $0<t_0<t_1$.
Then as $N\rightarrow\infty$
\begin{displaymath}
\left|(r_N)^{2-d}\mathbb{E}[g(T,X(T))]-c_d[(\frac{b_1}{\tau_1})^2+(\frac{b_2}{\tau_2})^2]\int_{t_0}^{t_1}\int_{\mathbb{R}^d}q_1(0,x_1;t,z)q_2(0,x_2;t,z)g(t,z)dtdz\right|\rightarrow 0,
\end{displaymath}
where $q_i$ is the transition density for the $d$-dimensional Brownian motion with diffusivity $a_i=(\frac{b_i}{\tau_i})^2$ and
\begin{displaymath}
 \frac{1}{c_d}=\int_0^\infty\frac{1}{(2\pi t)^{\frac{d}{2}}}e^{-\frac{1}{2t}}dt.
\end{displaymath}
\end{thm}
We know that the underlying motion of $X^N_i$ converges weakly to Brownian motion with diffusivity $a_i$ and $N$ here represents how close the motions are from Brownian motions. So, as expected, we see that when $r_N$ converges to zero relatively slowly compared to the convergence of the particles' free motions to Brownian motions, the coagulation kernel is the same as if the particles' motions are Brownian with diffusivities $a_i$. On the other hand, if $r_N$ converges to zero relatively fast compared to the convergence of the particles' free motions to Brownian motions, then the coagulation kernel is very different. In particular, the probability density that two particles collide at $dz$ during time $dt$ in this case is proportional to $N^{1/2}r_N^{d-1}$ while in the Brownian case it is proportional to $r_N^{d-2}$. We can think $N^{1/2}$ as the scale of average speed of the particles and thus when $r_N$ converges to zero relatively fast, the probability the two particles collide will depend both on their sizes and their average speed. Also note that, because we assumed $r_N< rN^{-\alpha}$ for some $\alpha>\frac{1}{2}$, we know that for large $N$ and small $r_N$, the probability density that two particles collide will be smaller than the density in the Brownian case. This confirms the intuition we have discussed about in the Introduction.
\subsection{Existence and uniqueness for coagulation-diffusion equations}\label{s2.2}
As we can see, the form of $K$ is different under different microscopic dynamics of the particles and this will also change the properties of Smoluchowski coagulation equations. Now, we assume $d=3$ and the particles have same density, i.e. their mass $y\sim r^{3}$. Then in the Brownian case, Einstein-Stokes relation suggests that $a(y)\sim \frac{1}{y^{1/3}}$. So, we have $K(y_1+y_2)=c_d(a_1+a_2)(r_1+r_2)^{d-2}\sim (y_1^{1/3}+y_2^{1/3})(\frac{1}{y_1^{1/3}}+\frac{1}{y_2^{1/3}})$. In \cite{BM}, Norris proved that (\ref{eqn}) has a unique solution when $K(y,y')\leq w(y)w(y')$ for some sublinear function $w$ such that $a^{-\frac{d}{2}}w$ is also sublinear. So, in the Brownian case, this result applies when we pick $w(y)=c(y^{1/3}+y^{-1/3})$ for some constant $c$. If we assume the particles are making diffusions under periodic drift, then Corollary \ref{period} suggests that under certain scaling limit, we should take $K(y_1+y_2)=c_d(\bar{a}_1+\bar{a}_2)(r_1+r_2)^{d-2}$. In the example discussed earlier, we would have $K(y_1+y_2)\sim (y_1^{1/3}+y_2^{1/3})(\frac{1}{y_1^{1/6}}+\frac{1}{y_2^{1/6}})$ and the result still applies if we pick $w(y)=c(y^{1/3}+y^{-1/6})$.

However, in the Ornstein-Uhlenbeck case, Theorem \ref{ou} suggests that $K\sim\sqrt{\frac{b_1^2}{\tau_1}+\frac{b_2^2}{\tau_2}}(y_1^{1/3}+y_2^{1/3})^2$ and the effective diffusivities of the two particles are $\frac{b_1^2}{\tau_1^2}$ and $\frac{b_2^2}{\tau_2^2}$ respectively. For $i=1,2$, in \cite{mechanical}, it is assumed that the drag force on a particle is caused by the particle being hit by random particles of much smaller sizes and higher speed and it has been shown that under certain scaling limit it is appropriate to choose $\tau_i=y_i^{-1/3}$ and $b_i=y_i^{-2/3}$. On the other hand, according to Einstein relation, where it is assumed that the drag force is caused by friction, the appropriate choice would be $\tau_i=y_i^{-2/3}$ and $b_i=y_i^{-5/6}$. In both cases, we have 
\begin{displaymath}
K\sim (y_1^{1/3}+y_2^{1/3})^2\sqrt{\frac{1}{y_1}+\frac{1}{y_2}},
\end{displaymath}
and the effective diffusivity of a particle with mass $y$ would be $y^{-2/3}$ according to \cite{mechanical} and $y^{-1/3}$ according to Einstein's relation. In both cases, we cannot directly apply prior results to obtain existence and uniqueness of the solution. Therefore, we will investigate alternative approaches to the well-posedness of (\ref{eqn}).

We assume the following conditions throughout this thesis 

(i) $K(y,z)\leq w(y)w(z)$ with $w:(0,\infty)\rightarrow (0,\infty)$ a non-decreasing sublinear function.

(ii) For some $\delta>0$, $\mu_0\mathbf{1}_{y<\delta}=0$. 

(iii) The diffusivity $a$ is strictly positive and measurable. 

Write $\mathcal{M}[0,T]$ for the set of measurable kernels 
\begin{displaymath}
\mu:[0,T]\times\mathbb{R}^d\times\mathcal{B}(0,\infty)\rightarrow [0,\infty].
\end{displaymath}
So, for time $t$ and position $x$, $\mu_t(x,.)$ is a measure on $\mathcal{B}(0,\infty)$. We will also use the notation $\langle f,\mu_t\rangle(x)=\int_0^\infty f(y)\mu_t(x,dy)$ for $f:(0,\infty)\rightarrow (0,\infty)$. We call a process $\mu_t\in\mathcal{M}[0,T]$ a solution of $(\ref{eqn})$ if it satisfies $(\ref{eqn})$ for $t\leq T$ and
\begin{displaymath}
\sup_{t\leq T}\|\langle y,\mu_t\rangle\|_1<\infty.
\end{displaymath}
This notion of solutions will also be used throughout this thesis for other pdes. It has been shown that
\begin{displaymath}
\|\langle y,\mu_t\rangle\|_1\leq\|\langle y,\mu_0\rangle\|_1,
\end{displaymath}
provided both sides are finite, see \cite{continuous}.
\begin{thm}\label{ssthm1}
Assume conditions (i), (ii) and (iii) hold. Let $(\mu_t^1)_{t\leq T}$ and $(\mu_t^2)_{t\leq T}$ be solutions of $(\ref{eqn})$ such that for $i=1,2$, $\sup_{t\leq T}\|\langle w^2,\mu^i_t \rangle\|_\infty<\infty$. Then $\mu^1=\mu^2$.
\end{thm}
In \cite{moment}, Hammond and Rezakhanlou proved that when the mass $y$ takes integer values, there is at most one solution $\mu$ such that $\sup_{t\leq T}\|\langle w^2,\mu_t \rangle\|_\infty<\infty$. Our result works in the case when $y$ can take values in positive real numbers, and we will see that the method we used gives a natural iteration scheme which can prove the existence result under certain conditions. Moreover this theorem works for a wide range of situations. There is no explicit requirement for the diffusivities and the condition $\sup_{t\leq T}\|\langle w^2,\mu_t \rangle\|_\infty<\infty$ looks reasonable.

\begin{thm}\label{ssthm2}
Write $p(y)=p^{t,x',x}(y)$. We assume that the function $w$ can be chosen so that for some constant $C$
\begin{equation}\label{ssineq}
\frac{y}{y+y'}w^2(y+y')p(y+y')-w^2(y)p(y)\leq C[w(y)w(y')p(y)+w(y)w(y')p(y')].
\end{equation}
If in addition, the initial kernel $\mu_0$ satisfies $\sup_{t>0}\|\langle w^2,P_t(\mu_0) \rangle\|_\infty<\infty$ and $\|\langle y,\mu_0\rangle\|_1<\infty$, then there exists $T>0$ such that there exists a unique solution $\mu$ to our PDEs up to time $T$. Moreover $\mu$ satisfies $sup_{t\leq T}\|\langle w^2,\mu_t \rangle\|_\infty<\infty$.\end{thm}
Note that (\ref{ssineq}) is satisfied if $w(y)=c_1y^u$ and $a(y)=c_2y^{-v}$ with $0<u\leq 1$ and $c_1,c_2,v>0$. To see this, we note that 
$\frac{p(y)}{p(y+y')}\geq (\frac{y}{y+y'})^{vd/2}$. By dividing both side of (\ref{ssineq}) by $p(y+y')$, it suffices to show that
\begin{displaymath}
\frac{y}{y+y'}(y+y')^{2u}-(y)^{2u}(\frac{y}{y+y'})^{vd/2}\leq C[(yy')^u(\frac{y}{y+y'})^{vd/2}+(yy')^u(\frac{y'}{y+y'})^{vd/2}],
\end{displaymath}
for some $C$. As this inequality is homogeneous, we can assume $y'=1$. Then, multiplying both sides by $\frac{(y+1)^{vd/2}}{y}$, it suffices to show that
\begin{displaymath}
(y+1)^{2u+vd/2-1}-y^{2u+vd/2-1}\leq C(y^{u+vd/2-1}+y^{u-1})
\end{displaymath}
for some $C$. When $y\leq 1$, this is true because the left hand side of the above inequality is at most $2^{2u+vd/2-1}$ while $y^{u-1}\geq 1$. When $y\geq 1$, this is also true because $\frac{(y+1)^{2u+vd/2-1}-y^{2u+vd/2-1}}{y^{u+vd/2-1}}$ is continuous on $y\geq 1$ and
\begin{displaymath}
\limsup_{y\rightarrow\infty}\frac{(y+1)^{2u+vd/2-1}-y^{2u+vd/2-1}}{y^{u+vd/2-1}}<\infty.
\end{displaymath}
\begin{cor}\label{cor1}
When $
K(y_1,y_2)=(y_1^{1/3}+y_2^{1/3})^2\sqrt{\frac{1}{y_1}+\frac{1}{y_2}}
$, and when $a(y)=y^{-\frac{1}{3}}$ or $a(y)=y^{-2/3}$, there exists $T>0$ such that there exists a unique solution to our PDEs up to time $T$.
\end{cor}
\begin{proof}
Note that
\begin{displaymath}
K(y_1,y_2)\leq 2(y_1^{2/3}+y_2^{2/3})(y_1^{-1/2}+y_2^{-1/2})
\end{displaymath}
Since $\mu_0\mathbf{1}_{y<\delta}=0$, we only need to care about the case when $y_1, y_2 \geq\delta$. So, we can pick $w(y)=4\delta^{-7/6}y^{2/3}$.
\end{proof}
Now, we will give two cases where we can show the global existence of the solutions.
\begin{thm}\label{ssthm3}
If all conditions in Theorem \ref{ssthm2} are satisfied, then there exists a unique global solution to our PDEs in the following two cases:

(a)$K(y,y')\leq w(y)v(y')+w(y')v(y)$ for some $v$ such that $wvp$ is sublinear.

(b)
$\sup_{t>0}(1+t)^{1+\epsilon}\|\langle w^2,P_t(\mu_0)\rangle\|_{\infty}<c,$
for some $\epsilon>0$ and sufficiently small $c>0$ depending on $\epsilon$ and $C$.
\end{thm}
Taking $v(y)=y^{-\frac{1}{2}}$ and $w(y)=4\sqrt{2}y^{2/3}$, the condition $(a)$ is satisfied for our case where
\begin{displaymath}
K(y_1,y_2)=(y_1^{1/3}+y_2^{1/3})^2\sqrt{\frac{1}{y_1}+\frac{1}{y_2}}
\end{displaymath}
and the diffusivity $a(y)=y^{-\frac{1}{3}}$. Condition $(b)$ is satisfied if, for example,
\begin{displaymath}
\int_{\mathbb{R}^d}\int_0^\infty\mu_0(x,dy)w^2(y)(1+a(y)^{-\frac{d}{2}})dx<h,
\end{displaymath}
for sufficiently small $h$. We can now conclude the following result.
\begin{cor}\label{cor2}
Assume $
K(y_1,y_2)=(y_1^{1/3}+y_2^{1/3})^2\sqrt{\frac{1}{y_1}+\frac{1}{y_2}}
$.
If $a(y)=y^{-\frac{1}{3}}$, then there exists a unique global solution. If $a=y^{-\frac{2}{3}}$ and $(b)$ is satisfied, then there also exists a unique global solution.
\end{cor}

So far, we have investigated the Smoluchowski equations modeling coagulating particles whose free motions are (approximately) Brownian. A natural question to ask is what if the particles' free motions are Brownian with a space and mass dependent drift. In (\ref{eqn}), $P_t$ was defined to be
\begin{displaymath}
P_t\mu(x,dy)=\int_{\mathbb{R}^d}\mu(x',dy)p^{t,x',x}(y)dx',
\end{displaymath}
with $p^{t,x',x}(y)$ is the transition density of a Brownian particle with diffusivity $a(y)$. We now consider the case where $p$ is instead the transition density of a Brownian particle with a space and mass dependent drift. More precisely, consider a particle whose free motion satisfies $X_0=x'$ and
\begin{displaymath}
dX_t=\sqrt{a(y)}dB_t+b_t(x,y)dt,
\end{displaymath}
with $b$ bounded and measurable in $x$, then we let $p^{t,x',x}(y)$ denote the probability density function of $X_t$ evaluated at $x$. The following theorem gives sufficient conditions for the well-posedness of (\ref{eqn}) in this case.
\begin{thm}\label{driftcase}
Theorem \ref{ssthm1}, Theorem \ref{ssthm2} and Theorem \ref{ssthm3} still hold in the case described above.
\end{thm}
However, Theorem \ref{ssthm2} and Theorem \ref{ssthm3} might not be very useful in the case when a drift term is involved, because the conditions required are usually not satisfied or hard to verify.  Therefore, we formulate some easy to check conditions for the well-posedness of (\ref{eqn}).

Suppose we have a function $B:(0,\infty)\rightarrow [0,\infty)$ such that for all $x\in\mathbb{R}^d,y\in (0,\infty)$, $t\geq 0$ and $i=1,2,...,d$, $|b_t^i(x,y)|\leq B(y)$. For $x,x'\in\mathbb{R}^d$, consider the process $X_0=x'$ and
\begin{displaymath}
dX_t=\sqrt{a(y)}dB_t+\mathbf{B}(X_t,y)dt,
\end{displaymath} 
where $\mathbf{B}(X_t,y)$ is the d-dimensional vector with $\mathbf{B}^i(X_t,y)=B(y)sgn(x^i-X^i_t)$. Let $q^{t,x',x}(y)$ be the probability density function of $X_t$ evaluated at $x$.  Define now
\begin{displaymath}
Q_t\mu(x,dy)=\int_{\mathbb{R}^d}\mu(x',dy)q^{t,x',x}(y)dx'.
\end{displaymath}
Usually, it is hard to compute $P$, but $Q$ can be evaluated explicitly. The following theorem allows us to check well-posedness of (\ref{eqn}) using properties on $q$.
\begin{thm}\label{driftcase2}
Write $q(y)=q^{t,x',x}(y)$. We assume that the function $w$ can be chosen so that for some constant $C$
\begin{equation}\label{drifteqn}
\frac{y}{y+y'}w^2(y+y')q(y+y')-w^2(y)q(y)\leq C[w(y)w(y')q(y)+w(y)w(y')q(y')].
\end{equation}
If in addition, the initial kernel $\mu_0$ satisfies $\sup_{t\geq 0}\|\langle w^2,Q_t(\mu_0) \rangle\|_\infty<\infty$ and $\|\langle y,\mu_0\rangle\|_1<\infty$, then there exists $T>0$ such that there exists a unique solution to our PDEs up to time $T$ satisfying
\begin{displaymath}
\sup_{t\leq T}\|\langle w^2,\mu_{t} \rangle\|_\infty<\infty.\end{displaymath}
Moreover, if $K(y,y')\leq w(y)v(y')+w(y')v(y)$ for some $v$ such that $wvq$ is sublinear, then there exists a unique global solution for our PDEs satisfying
\begin{displaymath}
\sup_{t\geq 0}\|\langle w^2,\mu_{t} \rangle\|_\infty<\infty.\end{displaymath}
\end{thm}
We see that this theorem is similar with Theorem \ref{ssthm2} and \ref{ssthm3}, we just replace $p$ by $q$. We then investigate the properties of $q$.
\begin{lem}\label{driftlem}
If $B/\sqrt{a}$ is non-increasing and $a$ is non-increasing. Then for $y>y'>0$, we have
\begin{displaymath}
q(y)/q(y')\leq [a(y)/a(y')]^{-\frac{d}{2}}.
\end{displaymath}
If $a,B$ are both non-increasing and $\frac{B}{\sqrt{a}}$ is non-decreasing, we have for $y>y'>0$,
\begin{displaymath}
q(y)/q(y')\leq (\frac{B(y)/a(y)}{B(y')/a(y')})^d.
\end{displaymath}
\end{lem}
This lemma can be viewed as an analogy of the following statement in the non-drift case:

For $y>y'>0$, if $a$ is non-increasing, then $p(y)/p(y')\leq [a(y)/a(y')]^{-d/2}$.
This was the only property of $p$ we have used to show Corollary \ref{cor1} and Corollary \ref{cor2}. Therefore, we can use the same argument to obtain the following result.
\begin{cor}
Assume $B/\sqrt{a}$ is non-increasing and $a$ is non-increasing. If $w(y)=c_1y^u$ and $a(y)=c_2y^{-v}$ with $0<u\leq 1$ and $c_1,c_2,v>0$, then (\ref{drifteqn}) is satisfied. If $wva^{-d/2}$ is sublinear, then $wvq$ is also sublinear.

Assume now instead $a,B$ are both non-increasing and $\frac{B}{\sqrt{a}}$ is non-decreasing. If $w(y)=c_1y^u$ and $\frac{B(y)}{a(y)}=c_2y^{-v}$ with $0<u\leq 1$ and $c_1,c_2,v>0$, then (\ref{drifteqn}) is satisfied. If $wvB^{d}a^{-d}$ is sublinear, then $wvq$ is also sublinear.
\end{cor}

\section{Estimate for diffusion particles}
In \cite{bmk}, Norris proved Theorem \ref{diff} in the case $X^i$ are Brownian motions. Intuitively, the coagulation kernel $K(z)$ can be viewed as a quantity measuring the probability of collision happening provided the two particles are close to $z$ at time $s$. Also it is unlikely for two particles to collide at $z$ and at time $t$ unless they are both close to $z$ at a time $s$ slightly before $t$. Now, if the two particles are near $z$ at time $s$, then we can approximate their behaviour during $(s,t)$ as Brownian motions with diffusivities $a_i(x)$. In this section, we will use this idea to prove Theorem $\ref{diff}$.
\subsection{A formal proof of Theorem \ref{diff}}
We will now give a formal proof showing
\begin{displaymath}
\left|N\mathbb{E}(g(T,X(T))\mathbf{1}_{T<R})-\int_0^R\int_{\mathbb{R}^d} K(z)p_1(0,x_1;s,z)p_2(0,x_2;s,z)g(s,z)dzds\right|\rightarrow 0.
\end{displaymath}
We define for each $s\in(0,\infty)$ and $z\in\mathbb{R}^d$ the process 
\begin{displaymath}
 M_t=\mathbf{1}_{t<s}p_1(t,X^1_{t\wedge T};s,z)p_2(t,X_{t\wedge T}^2;s,z),\:t\geq 0.
\end{displaymath}
Recall $x_1\neq x_2$ and thus $T>0$ almost surely. Moreover, $M$ is continuous almost surely, $(M_t)_{t<s}$ is a martingale, $M_t=0$ for all $t\geq s$ and we can show that $M_t$ is uniformly bounded up to $T$. Hence, by optional stopping and bounded convergence theorem,
\begin{displaymath}
 M_0=\mathbb{E}[M_T].
\end{displaymath}
On multiplying by $g(s,z)K(z)$ and integrating over $(0,R)\times\mathbb{R}^d$ we obtain
\begin{align*}
&\int_0^R\int_{\mathbb{R}^d}p_1(0,x_1;s,z)p_2(0,x_2;s,z)g(s,z)K(z)dzds\\
&=\mathbb{E}[\int_T^R\int_{\mathbb{R}^d}p_1(T,X^1_T;s,z)p_2(T,X_T^2;s,z)g(s,z)K(z)dzds].\numberthis\label{eqn1}
\end{align*} 
The main part of the proof will be on estimating the right hand side of the above equation. When $r_N$ is small, the probability $T<R$ will also be small. Therefore, we can in fact ignore $p_1(T,X^1_T;s,z)p_2(T,X_T^2;s,z)g(s,z)K(z)$ unless it is large. We note that $|X^1_T-X^2_T|$ is small, and thus $p_1(T,X^1_T;s,z)p_2(T,X_T^2;s,z)$ can be large when $z$ is close to $X^1_T$ and $X^2_T$ and $s$ is slightly larger than $T$. Actually, we can ignore the contribution when $s$ is not sufficiently close to $T$ or $z$ is not sufficiently close to $X(T)$. By uniform continuity of $g$, we can simply estimate the expectation in equation (\ref{eqn1}) by 
\begin{displaymath}
\mathbb{E}[g(T,X(T))\int_T^R\int_{\mathbb{R}^d}p_1(T,X^1_T;s,z)p_2(T,X_T^2;s,z)K(z)dzds]
\end{displaymath}
and it remains for us to estimate 
\begin{displaymath}
\int_T^R\int_{\mathbb{R}^d}p_1(T,X^1_T;s,z)p_2(T,X_T^2;s,z)K(z)dzds
\end{displaymath}
when $T<R$. Again, we only need to care about the contribution when $s$ is close to $T$. We know that in a small time interval, the contribution of the drift to the motion of $X^i$ is relatively small in comparison to the contribution of the diffusion. Let $p'_i$ be the transition densities of the motion
\begin{displaymath}
 dX'^i_t=\sqrt{a_i(X'^i_t)}dB^i_t.
\end{displaymath}
We can actually approximate $p_i$ by $p'_i$. We now condition on $T<s$ and set $X'^i_T=X^i_T$. We have by Dubins Schwarz theorem, 
\begin{displaymath}
X'^1_s-X'^2_s=W_{A(s)},
\end{displaymath}
where $W$ is a Brownian motion with diffusivity $1$ and $W_0=X^1_T-X^2_T$ and $A_i(s)=\int_T^s a_1(X'^1_r)+a_2(X'^2_r)dr$. Let $q$ denote the transition density of $W$ and $V(h)$ the volume of a ball of radius $h$ in $\mathbb{R}^d$, we could have
\begin{align*}
&\int_T^\infty\int_{\mathbb{R}^d}p'_1(T,X^1_T;s,z)p'_2(T,X_T^2;s,z)K(z)dzds\\&=c_dr^{d-2}\int_T^\infty\int_{\mathbb{R}^d}p'_1(T,X^1_T;s,z)p'_2(T,X_T^2;s,z)(a_1(z)+a_2(z))dzds\\
&=c_dr^{d-2}\int_T^\infty\int_{\mathbb{R}^d}\lim_{h\rightarrow 0}\mathbb{E}[ \frac{\mathbf{1}_{|X'^1_s-z|<h}}{V(h)}p'_2(T,X_T^2;s,z)((a_1(z)+a_2(z))]dzds\\
&=c_dr^{d-2}\int_T^\infty\lim_{h\rightarrow 0} \mathbb{E}[\frac{\mathbf{1}_{|W_s|<h}}{V(h)}(a_1(X'^1_s)+a_2(X'^2_s))]ds\\
&=c_dr^{d-2}\int_T^\infty\lim_{h\rightarrow 0} \mathbb{E}[\frac{\mathbf{1}_{|W_{A(s)}|<h}}{V(h)}]dA(s)\\
&=c_dr^{d-2}\int_T^\infty\lim_{h\rightarrow 0} \mathbb{E}[\frac{\mathbf{1}_{|W_s|<h}}{V(h)}]ds\\
&=c_dr^{d-2}\int_T^\infty q(T,X^1_T-X_T^2;s,0)ds.\\
\end{align*}
Since $q$ is the transition density of a standard Brownian motion and $|X^1_T-X^2_T|=rN^{-1/(d-2)}$, we have
\begin{displaymath}
\int_T^\infty q(T,X^1_T-X_T^2;s,0)ds=\int_0^\infty\frac{1}{(2\pi t)^{\frac{d}{2}}}e^{\frac{-r^2N^{-2/(d-2)}}{2t}}dt.
\end{displaymath}
We now make the substitution $u=\frac{t}{r^2N^{-2/(d-2)}}$ and recall that
\begin{displaymath}
\frac{1}{c_d}=\int_0^\infty\frac{1}{(2\pi t)^{\frac{d}{2}}}e^{-\frac{1}{2t}}dt
\end{displaymath}
to obtain
\begin{align*}
&\int_0^\infty\frac{1}{(2\pi t)^{\frac{d}{2}}}e^{\frac{-r^2N^{-2/(d-2)}}{2t}}dt\\
&=(r^2N^{-2/(d-2)})^{\frac{2-d}{2}}\int_0^\infty\frac{1}{(2\pi u)^{\frac{d}{2}}}e^{-\frac{1}{2u}}du\\
&=\frac{1}{c_d}Nr^{2-d}.
\end{align*}
Hence, we have
\begin{displaymath}
\int_T^R\int_{\mathbb{R}^d}p_1(T,X^1_T;s,z)p_2(T,X_T^2;s,z)K(z)dzds=N.
\end{displaymath}
So far, we took integral from $T$ to $\infty$, but as we have discussed earlier, we can ignore the contribution when $s$ is not close to $T$ anyway. Therefore, the above calculation concludes that when $T<R$,
\begin{displaymath}
\int_T^R\int_{\mathbb{R}^d}p_1(T,X^1_T;s,z)p_2(T,X_T^2;s,z)K(z)dzds=N+o(N)
\end{displaymath}
as desired.
\subsection{Estimates on transition densities}
To make the proof rigorous, we will first review a number of estimates we can get regarding to the transition densities $p_i$, which will be useful for us to prove Theorem $\ref{diff}$. To start with, we want to have some idea about the behaviour of $p_i$ and we will use the main result in \cite{Davies}. They showed the following theorem.
\begin{thm}\label{aronson}
Using same notation as in Theorem \ref{diff}, there exists a constant $C$ depending only on $d$ and $R$ such that for all $x,y\in \mathbb{R}^d$,
\begin{displaymath}
 C^{-1}t^{-d/2}\exp\{-C|y-x|^2/t\}e^{-Ct}\leq p_i(0,x;t,y)\leq Ct^{-d/2}\exp\{-|y-x|^2/Ct\}e^{Ct}.
\end{displaymath}
Moreover, $p_i(0,x;t,y)$ is locally H\"older continuous in $t>0$ and $y$.
\end{thm}

Next, we note that it is intuitive to believe that $\int_{t_0}^{t_1}\int_{\mathbb{R}^d}p_1(0,x_1;s,z)p_2(0,x_2;s,z)dzds$ measures the expected amount of time when the two particles are ``close'', and more precisely, we would expect
\begin{displaymath}
\int_{t_0}^{t_1}\int_{\mathbb{R}^d}p_1(0,x_1;s,z)p_2(0,x_2;s,z)dzds=\lim_{h\rightarrow 0}V(h)^{-1}\mathbb{E}\big[\int_{t_0}^{t_1}\mathbf{1}_{|X^1_s-X^2_s|<h}ds\big],
\end{displaymath}
where $V(h)$ denotes the volume of the $d$-dimensional sphere with radius $h$. Actually, using the above theorem, we can prove the following more general result.
\begin{cor}\label{vol}
 Let $X=X^1-X^2$, for $0\leq t_0<t_1$ and $x_1\neq x_2$, we have for all bounded uniformly continuous function $f$,
\begin{displaymath} \int_{t_0}^{t_1}\int_{\mathbb{R}^d}p_1(0,x_1;s,z)p_2(0,x_2;s,z)f(z)dzds=\lim_{h\rightarrow 0}V(h)^{-1}\mathbb{E}\big[\int_{t_0}^{t_1}\mathbf{1}_{|X_s|<h}f(X_s^2)ds\big]. 
\end{displaymath}
\end{cor}
\begin{proof}
 Let $S_n=[\max\{\frac{1}{n},t_0\},t_1]\times\{x\in\mathbb{R}^d:|x|\leq n\}$. Note that
\begin{align*}
&V(h)^{-1}\mathbb{E}[\int_{t_0}^{t_1}\mathbf{1}_{|X_s|<h}f(X_s^2)ds]\\&=V(h)^{-1}\int_{t_0}^{t_1}\int_{\mathbb{R}^d}\int_{|y-z|\leq h}p_1(0,x_1;s,y)p_2(0,x_2;s,z)f(z)dydzds.\label{cor}\numberthis
\end{align*}
By continuity of $p$, we know that 
\begin{displaymath}
\lim_{h\rightarrow 0}V(h)^{-1}\int_{|y-z|\leq h}p_1(0,x_1;s,y)p_2(0,x_2;s,z)f(z)dy=p_1(0,x_1;s,z)p_2(0,x_2;s,z)f(z).
\end{displaymath}
So, if we let $h\rightarrow 0$ in (\ref{cor}) and justify changing the order of limit and integral on the right hand side, we would get the desired result. Now, using the H\"older continuity result in Theorem \ref{aronson} and uniform continuity of $f$,  
\begin{displaymath}
V(h)^{-1}\int_{|y-z|\leq h}p_1(0,x_1;s,y)p_2(0,x_2;s,z)f(z)dy
\end{displaymath}
actually converges uniformly to $p_1(0,x_1;s,z)p_2(0,x_2;s,z)f(z)$ in $S_n$. Therefore, we know that
\begin{align*}
&\lim_{h\rightarrow 0}V(h)^{-1}\int_{S_n}\int_{|y-z|\leq h}p_1(0,x_1;s,y)p_2(0,x_2;s,z)f(z)dydzds\\
&=\int_{S_n}\lim_{h\rightarrow 0}V(h)^{-1}\int_{|y-z|\leq h}p_1(0,x_1;s,y)p_2(0,x_2;s,z)f(z)dydzds\\
&=\int_{S_n}p_1(0,x_1;s,z)p_2(0,x_2;s,z)f(z)dzds
\end{align*}
Now, we use the Theorem \ref{aronson} to deduce that
\begin{align*}
&V(h)^{-1}\int_{(S_n)^c}\int_{|y-z|\leq h}p_1(0,x_1;s,y)p_2(0,x_2;s,z)f(z)dzdsdy \\
&\leq V(h)^{-1}e^{2Ct_1}\int_{(S_n)^c}\int_{|y-z|\leq h}C^2s^{-d}exp\{\frac{-|x_1-y|^2-|x_2-z|^2}{Cs}\}dzdsdy.
\end{align*}
We would like to show that the right hand side in the above inequality converges to zero uniformly in $h$ as $n\rightarrow\infty$. Now, if we let $h<\frac{|x_1-x_2|}{2}$ and assume $|y-z|\leq h$ then we have, by triangle inequality, that 
\begin{displaymath}
|x_1-y|+|x_2-z|+|y-z|\geq |x_1-x_2|,
\end{displaymath}
and thus
\begin{displaymath}
|x_1-y|+|x_2-z|\geq \frac{|x_1-x_2|}{2}.
\end{displaymath}
Therefore, we also have
\begin{displaymath}
|x_1-y|^2+|x_2-z|^2\geq \frac{|x_1-x_2|^2}{8}.
\end{displaymath}
We can further deduce that
\begin{align*}
&V(h)^{-1}e^{2Ct_1}\int_{(S_n)^c}\int_{|y-z|\leq h}C^2s^{-d}exp\{\frac{-|x_1-y|^2-|x_2-z|^2}{Cs}\}dzdsdy\\
&\leq C^2e^{2Ct_1}\int_{(S_n)^c}s^{-d}\min\{e^{\frac{-|x_1-x_2|^2}{8Cs}},e^{-\frac{|x_1-y|^2}{Cs}}\}dsdy.
\end{align*}
Note that for sufficiently large $m>0$, we have that for all $|y|>m$ and $0<s\leq t_1$ 
\begin{displaymath}
s^{-d}e^{-\frac{|x_1-y|^2}{Cs}}\leq t_1^{-d}e^{-\frac{|x_1-y|^2}{Ct_1}}.
\end{displaymath}
Then we obtain
\begin{align*}
&\int_{|y|>m}\int_0^{t_1}s^{-d}\min\{e^{\frac{-|x_1-x_2|^2}{8Cs}},e^{-\frac{|x_1-y|^2}{Cs}}\}dsdy\\
&\leq t_1\int_{|y|>m}t_1^{-d}e^{-\frac{|x_1-y|^2}{Ct_1}}dy<\infty,
\end{align*}
and 
\begin{align*}
&\int_{|y|\leq m}\int_0^{t_1}s^{-d}\min\{e^{\frac{-|x_1-x_2|^2}{8Cs}},e^{-\frac{|x_1-y|^2}{Cs}}\}dsdy\\
&\leq V(m)\int_0^{t_1}s^{-d}e^{-\frac{|x_1-x_2|^2}{8Cs}}ds<\infty.
\end{align*}
Summing up, we have 
\begin{displaymath}
\int_{\mathbb{R}^d}\int_0^{t_1}s^{-d}\min\{e^{\frac{-|x_1-x_2|^2}{8Cs}},e^{-\frac{|x_1-y|^2}{Cs}}\}dsdy<\infty.
\end{displaymath}
Thus,
\begin{displaymath}
C^2e^{2Ct_1}\int_{(S_n)^c}s^{-d}\min\{e^{\frac{-|x_1-x_2|^2}{8Cs}},e^{-\frac{|x_1-y|^2}{Cs}}\}dsdy\rightarrow 0
\end{displaymath}
and the convergence is uniform in $h$.
Now, we use H\"older continuity to obtain
\begin{align*}
 &\lim_{h\rightarrow 0}V(h)^{-1}\mathbb{E}(\int_{t_0}^{t_1}\mathbf{1}_{|X_s|<h}f(X_s^2)ds)\\
&=\lim_{h\rightarrow 0}V(h)^{-1}\int^{\mathbb{R}^d}\int_{|y-z|\leq h}\int_{t_0}^{t_1}p_1(0,x_1;s,y)p_2(0,x_2;s,z)f(z)dsdydz\\
&=\lim_{h\rightarrow 0}\lim_{n\rightarrow\infty}V(h)^{-1}\int_{(S_n)}\int_{|y-z|\leq h}p_1(0,x_1;s,y)p_2(0,x_2;s,z)f(z)dzdsdy\\
&=\lim_{n\rightarrow\infty}\lim_{h\rightarrow 0}V(h)^{-1}\int_{(S_n)}\int_{|y-z|\leq h}p_1(0,x_1;s,y)p_2(0,x_2;s,z)f(z)dzdsdy\\
&=\lim_{n\rightarrow\infty}\int_{(S_n)}p_1(0,x_1;s,y)p_2(0,x_2;s,y)f(y)dyds\\
&=\int_{t_0}^{t_1}\int_{\mathbb{R}^d}p_1(0,x_1;s,z)p_2(0,x_2;s,z)f(z)dzds.
\end{align*}
We could swap the order of limits in the third line to the fourth line because we have uniform convergence of the integral.
\end{proof}

As mentioned earlier, it is intuitive to believe that during a small amount of time, the drift term in the free motion won't affect the transition density much. To formalize this idea, we will need some estimates on the transition density of a Brownian motion with drift. The following theorem provides us a tight bound for it. 
\begin{thm}\label{drift}
Consider
\begin{displaymath}
dX_t=b_tdt+dB_t, 
\end{displaymath}
\begin{displaymath}
 X_0=x,
\end{displaymath}
where $B_t$ is a standard Brownian motion in $\mathbb{R}^d$ and $b_t$ is an $\mathcal{F}_t$ adapted process and $|b^i_t|<C$ for all $t$ and $i=1,2,\dots,d$.  Then for all $t>0$ and all $y\in\mathbb{R}^d$, the random variable $X_t$ has a density function $\rho$ such that
\begin{displaymath}
\frac{1}{(2\pi t)^{\frac{d}{2}}}\prod_{i=1}^d(\int^\infty_{|x^i-y^i|/\sqrt{t}}ze^{-(z+C\sqrt{t})^2/2}dz)\leq \rho(t,y),
\end{displaymath}
and
\begin{displaymath}
\frac{1}{(2\pi t)^{\frac{d}{2}}}\prod_{i=1}^d(\int^\infty_{|x^i-y^i|/\sqrt{t}}ze^{-(z-C\sqrt{t})^2/2}dz)\geq \rho(t,y),
\end{displaymath}
and the two bounds are attained when $b^i_t=Csgn(X^i_t-y^i)$ and when $b^i_t=Csgn(y^i-X^i_t)$ respectively. Moreover, $\rho$ can be chosen to be locally H\"older continuous.
\end{thm}\begin{proof}
In \cite{sharp}, they have shown that the above inequalities are true in the case when $b_t$ is a function of $X_t$. Their method also works if $b_t$ is any $\mathcal{F}_t$ adapted process. They have shown that the two bounds are attained when when $b^i_t=Csgn(X^i_t-y^i)$ and when $b^i_t=Csgn(y^i-X^i_t)$ respectively. Now, denote $p^+_y(0,x;t,z)$ for the transition density when $b^i_t=Csgn(y^i-X^i_t)$ and $p^-_y(0,x;t,z)$ for the transition density when $b^i_t=Csgn(X^i_t-y^i)$. Consider the measure $Q$ under which the process
\begin{displaymath}
dX_t=Csgn(y^i-X^i_t)dt+dB_t
\end{displaymath} 
in the usual measure is a standard Brownian motion. Then the process
\begin{displaymath}
dX_t=b_tdt+dB_t
\end{displaymath} 
in the usual measure will be a Brownian motion with drift $b_t-Csgn(y^i-X^i_t)$ under $Q$. Therefore, by Cameron-Martin formula, we can let $\rho$ be such that
\begin{displaymath}
\frac{\rho(t,y)}{p^+_y(0,x;t,y)}=\mathbb{E}_{x,y}\big\{\prod_{i=1}^d\exp[\int_0^t(b^i_s-Csgn(y^i-W^i_s))dB^i_s-\frac{1}{2}\int_0^t(b^i_s-Csgn(y^i-W^i_s))^2ds]\big\},
\end{displaymath}
where $W$ is the motion satisfying
\begin{displaymath}
 dW^i_s=Csgn(y^i-W^i_s)ds+dB^i_s
\end{displaymath}
and $\mathbb{E}_{x,y}$ denotes the expectation conditioning on $W_0=x$ and $W_t=y$. The conditional expectation $\mathbb{E}_{x,y}$ can be shown to be well defined using Gaussian heat kernel estimation. In \cite{conditional}, they have shown that for $s<t$, we can represent $B_s$ by
\begin{displaymath}
 B^i_s=B'^i_s+\int_0^s\frac{\partial}{\partial W^i}log[p^+_y(s,W_s;t,y)]ds,
\end{displaymath}
for some unconditional Brownian motion $B'$. Also note that $p^+_y(s,W_s;t,y)$ is non-increasing in $|W_s^i-y_i|$. So, $\frac{\partial}{\partial W^i_s}log[p^+_y(s,W_s;t,y)]$ has same sign as $y^i-W^i_s$. Also, $b^i_s\leq C$, so 
\begin{displaymath}
[b^i_s-Csgn(y^i-W^i_s)]\frac{\partial}{\partial W^i_s}log[p^+_y(s,W_s;t,y)]\leq 0.
\end{displaymath}
Now, we can plug these in and obtain
\begin{align*}
\frac{\rho(t,y)}{p^+_y(0,x;t,y)}&=\mathbb{E}_{x,y}\Big\{\prod_{i=1}^d\exp\big\{\int_0^t(b^i_s-Csgn(y^i-W^i_s))dB'^i_s-\frac{1}{2}\int_0^t(b^i_s-Csgn(y^i-W^i_s))^2ds\\&+\int_0^t(b^i_s-Csgn(y^i-W^i_s))\frac{\partial}{\partial W^i}log[p^+_y(s,W_s;t,y)]ds\big\}\Big\}\\
&\leq \mathbb{E}_{x,y}\big\{\prod_{i=1}^d\exp[\int_0^t(b^i_s-Csgn(y^i-W^i_s))dB'^i_s-\frac{1}{2}\int_0^t(b^i_s-Csgn(y^i-W^i_s))^2ds]\big\}\\
&=1.
\end{align*}
This proves the upper bound for the density and the lower bound can be obtained similarly. 

Now, we will use these bounds to prove the local H\"older continuity of the transition density. Intuitively, when $t$ is small, $C\sqrt{t}$ is also small and the two bounds in the theorem are close to each other. This means that the drift has little influence on the transition kernel when $t$ is small. So, if we condition on the distribution of the $X_s$ for $s$ slightly smaller than $t$, then because we know that the transition kernel of a Brownian motion without drift is locally H\"older continuous and the drift term's influence on the transition kernel is small, we expect that the probability density function of $X_t$ is also locally H\"older continuous. Then we use tower law to express the expectation of the probability density function of $X_t$ without conditioning on $X_s$ and that should also be locally H\"older continuous.

Now, we formalize this idea. Let $s>0$ be fixed, and let $Y_t=X_{t+s}$ and $\mathcal{G}_t=\mathcal{F}_{t+s}$, then $Y$ is a $\mathcal{G}_t$ adapted process satisfying the conditions in the theorem. So, the random variable $Y_t$ conditional on $X_s=u$ has a density function, $\bar{\rho}_{s,u}(t,y)$, satisfying
\begin{displaymath}
\bar{\rho}_{s,u}(t,y)\leq\frac{1}{(2\pi t)^{\frac{d}{2}}}\prod_{i=1}^d(\int^\infty_{|u^i-y^i|/\sqrt{t}}we^{-(w-C\sqrt{t})^2/2}dw)
\end{displaymath}
and
\begin{displaymath}
\bar{\rho}_{s,u}(t,y)\geq\frac{1}{(2\pi t)^{\frac{d}{2}}}\prod_{i=1}^d(\int^\infty_{|u^i-y^i|/\sqrt{t}}we^{-(w+C\sqrt{t})^2/2}dw),
\end{displaymath}
where $u$ denotes the position $X_s$. Now, let $|y-z|<\epsilon$ for some sufficiently small positive $\epsilon$, we have for $s<t$,
\begin{displaymath}
 \rho(t,y)=\int_{\mathbb{R}^d}\rho(s,u)\bar{\rho}_{s,u}(t-s,y)du,
\end{displaymath}
and 
\begin{displaymath}
\rho(t,y)-\rho(t,z)=\int_{\mathbb{R}^d}\rho(s,u)[\bar{\rho}_{s,u}(t-s,y)-\bar{\rho}_{s,u}(t-s,z)]du.
\end{displaymath}
If we let $s=t-\epsilon$, then we have
\begin{align*}
 \bar{\rho}_{s,u}(t-s,y)&\leq \frac{1}{(2\pi \epsilon)^{\frac{d}{2}}}\prod_{i=1}^d(\int^\infty_{|u^i-y^i|/\sqrt{\epsilon}}we^{-(w-C\sqrt{\epsilon})^2/2}dw)\\
&=\frac{1}{(2\pi \epsilon)^{\frac{d}{2}}}\prod_{i=1}^d(\int^\infty_{|u^i-y^i|/\sqrt{\epsilon}-C\sqrt{\epsilon}}(w+C\sqrt{\epsilon})e^{-w^2/2}dw)
\end{align*}
Because 
\begin{displaymath}
e^{-\frac{(|u^i-y^i|/\sqrt{\epsilon}-C\sqrt{\epsilon})^2}{2}}=\int^\infty_{|u^i-y^i|/\sqrt{\epsilon}-C\sqrt{\epsilon}}we^{-w^2/2}dw,
\end{displaymath}
we have
\begin{align*}
&\frac{1}{(2\pi \epsilon)^{\frac{d}{2}}}\prod_{i=1}^d(\int^\infty_{|u^i-y^i|/\sqrt{\epsilon}-C\sqrt{\epsilon}}(w+C\sqrt{\epsilon})e^{-w^2/2}dw)\\
&=\frac{1}{(2\pi \epsilon)^{\frac{d}{2}}}\prod_{i=1}^d[e^{-\frac{(|u^i-y^i|/\sqrt{\epsilon}-C\sqrt{\epsilon})^2}{2}}(\frac{\int^\infty_{|u^i-y^i|/\sqrt{\epsilon}-C\sqrt{\epsilon}}(w+C\sqrt{\epsilon})e^{-w^2/2}dw}{\int^\infty_{|u^i-y^i|/\sqrt{\epsilon}-C\sqrt{\epsilon}}we^{-w^2/2}dw})].\\
\end{align*}
Now, consider the function
\begin{displaymath}
f(x)=\frac{\int^\infty_{x}(w+C\sqrt{\epsilon})e^{-w^2/2}dw}{\int^\infty_{x}we^{-w^2/2}dw},
\end{displaymath}
we have
\begin{align*}
\frac{f(x)}{dx}&=\frac{xe^{-x^2/2}\int^\infty_{x}(w+C\sqrt{\epsilon})e^{-w^2/2}dw-(x+C\sqrt{\epsilon})e^{-x^2/2}\int^\infty_{x}we^{-w^2/2}dw}{(\int^\infty_{x}we^{-w^2/2}dw)^2}\\
&=\frac{e^{-x^2/2}\int^\infty_{x}[x(w+C\sqrt{\epsilon})-(x+C\sqrt{\epsilon})w]e^{-w^2/2}dw}{(\int^\infty_{x}we^{-w^2/2}dw)^2}\\
&=\frac{e^{-x^2/2}\int^\infty_{x}[xC\sqrt{\epsilon}-wC\sqrt{\epsilon}]e^{-w^2/2}dw}{(\int^\infty_{x}we^{-w^2/2}dw)^2}\\
&\leq 0.
\end{align*}
So, $f$ is a non-increasing function and thus
\begin{align*}
&\frac{1}{(2\pi \epsilon)^{\frac{d}{2}}}\prod_{i=1}^d[e^{-\frac{(|u^i-y^i|/\sqrt{\epsilon}-C\sqrt{\epsilon})^2}{2}}(\frac{\int^\infty_{|u^i-y^i|/\sqrt{\epsilon}-C\sqrt{\epsilon}}(w+C\sqrt{\epsilon})e^{-w^2/2}dw}{\int^\infty_{|u^i-y^i|/\sqrt{\epsilon}-C\sqrt{\epsilon}}we^{-w^2/2}dw})]\\
&\leq\frac{1}{(2\pi \epsilon)^{\frac{d}{2}}}\prod_{i=1}^d[e^{-\frac{(|u^i-y^i|/\sqrt{\epsilon}-C\sqrt{\epsilon})^2}{2}}(\frac{\int^\infty_{-C\sqrt{\epsilon}}(w+C\sqrt{\epsilon})e^{-w^2/2}dw}{\int^\infty_{-C\sqrt{\epsilon}}we^{-w^2/2}dw})]\\
&=\frac{1}{(2\pi \epsilon)^{\frac{d}{2}}}\prod_{i=1}^d[e^{-\frac{(|u^i-y^i|/\sqrt{\epsilon}-C\sqrt{\epsilon})^2}{2}}(1+\frac{C\sqrt{\epsilon}\int^\infty_{-C\sqrt{\epsilon}}e^{-w^2/2}dw}{\int^\infty_{-C\sqrt{\epsilon}}we^{-w^2/2}dw})]\\
&\leq \frac{1}{(2\pi \epsilon)^{\frac{d}{2}}}\prod_{i=1}^d[e^{-\frac{(|u^i-y^i|/\sqrt{\epsilon}-C\sqrt{\epsilon})^2}{2}}(1+\frac{C\sqrt{\epsilon}\int^\infty_{-C}e^{-w^2/2}dw}{\int^\infty_{-C}we^{-w^2/2}dw})]\\
&\leq\frac{1}{(2\pi \epsilon)^{\frac{d}{2}}}\prod_{i=1}^d[e^{-\frac{(|u^i-y^i|/\sqrt{\epsilon}-C\sqrt{\epsilon})^2}{2}}(1+c\sqrt{\epsilon})].
\end{align*}
for some constant $c$ which does not depend on $u$. Now we will denote $c$ as a constant depending only on $C$ and $d$ and its value might change from line to line. This concludes that
\begin{displaymath}\numberthis\label{eqn0}
 \bar{\rho}_{s,u}(t-s,y)\leq\frac{1}{(2\pi \epsilon)^{\frac{d}{2}}}\prod_{i=1}^d[e^{-\frac{(|u^i-y^i|/\sqrt{\epsilon}-C\sqrt{\epsilon})^2}{2}}(1+c\sqrt{\epsilon})].
\end{displaymath}
Similarly, we have
\begin{align*}
 \bar{\rho}_{s,u}(t-s,z)&\geq\frac{1}{(2\pi \epsilon)^{\frac{d}{2}}}\prod_{i=1}^d(\int^\infty_{|u^i-y^i|/\sqrt{\epsilon}}we^{-(w+C\sqrt{\epsilon})^2/2}dw)\\
&=\frac{1}{(2\pi \epsilon)^{\frac{d}{2}}}\prod_{i=1}^d(\int^\infty_{|u^i-y^i|/\sqrt{\epsilon}+C\sqrt{\epsilon}}(w-C\sqrt{\epsilon})e^{-w^2/2}dw)\\
&=\frac{1}{(2\pi \epsilon)^{\frac{d}{2}}}\prod_{i=1}^d[e^{-\frac{(|u^i-y^i|/\sqrt{\epsilon}+C\sqrt{\epsilon})^2}{2}}(\frac{\int^\infty_{|u^i-y^i|/\sqrt{\epsilon}+C\sqrt{\epsilon}}(w-C\sqrt{\epsilon})e^{-w^2/2}dw}{\int^\infty_{|u^i-y^i|/\sqrt{\epsilon}+C\sqrt{\epsilon}}we^{-w^2/2}dw})]\\
&\geq \frac{1}{(2\pi \epsilon)^{\frac{d}{2}}}\prod_{i=1}^d[e^{-\frac{(|u^i-y^i|/\sqrt{\epsilon}+C\sqrt{\epsilon})^2}{2}}(\frac{\int^\infty_{0}(w-C\sqrt{\epsilon})e^{-w^2/2}dw}{\int^\infty_{0}we^{-w^2/2}dw})]\\
&=\frac{1}{(2\pi \epsilon)^{\frac{d}{2}}}\prod_{i=1}^d[e^{-\frac{(|u^i-y^i|/\sqrt{\epsilon}+C\sqrt{\epsilon})^2}{2}}(1-\frac{C\sqrt{\epsilon}\int^\infty_{0}e^{-w^2/2}dw}{\int^\infty_{0}we^{-w^2/2}dw})]\\
&\geq\frac{1}{(2\pi \epsilon)^{\frac{d}{2}}}\prod_{i=1}^d[e^{-\frac{(|u^i-y^i|/\sqrt{\epsilon}+C\sqrt{\epsilon})^2}{2}}(1-c\sqrt{\epsilon})].
\end{align*}
So, if $|u-y|\leq\epsilon^{1/4}$, then we have
\begin{align*}
 \frac{\bar{\rho}_{s,u}(t-s,y)}{\bar{\rho}_{s,u}(t-s,z)}&\leq \frac{1+c\sqrt{\epsilon}}{1-c\sqrt{\epsilon}}\prod_{i=1}^d[e^{-\frac{(|u^i-y^i|/\sqrt{\epsilon}-C\sqrt{\epsilon})^2}{2}}e^{\frac{(|u^i-z^i|/\sqrt{\epsilon}+C\sqrt{\epsilon})^2}{2}}]\\
&\leq\frac{1+c\sqrt{\epsilon}}{1-c\sqrt{\epsilon}}\prod_{i=1}^de^{c\epsilon^{\frac{1}{4}}}\\
&\leq 1+c\epsilon^{\frac{1}{4}}.
\end{align*}
If $|u-y|\geq\epsilon^{1/4}$, then we simply have
\begin{displaymath}
 \bar{\rho}_{s,u}(t-s,y)\leq c\epsilon^{-\frac{d}{2}}e^{-\frac{c}{\sqrt{\epsilon}}}\leq\epsilon,
\end{displaymath}
say. So, for general $u\in\mathbb{R}^d$, we would have
\begin{displaymath}
 \bar{\rho}_{s,u}(t-s,y)-\bar{\rho}_{s,u}(t-s,z)\leq c\epsilon^{1/4}\bar{\rho}_{s,u}(t-s,z)+\epsilon.
\end{displaymath}
Thus, we have
\begin{align*}
\rho(t,y)-\rho(t,z)&= \int_{\mathbb{R}^d}\rho(s,u)[\bar{\rho}_{s,u}(t-s,y)-\bar{\rho}_{s,u}(t-s,z)]du\\
&\leq\int_{\mathbb{R}^d}\rho(s,u)[c\epsilon^{1/4}\bar{\rho}_{s,u}(t-s,z)+\epsilon]du\\
&=c\epsilon^{1/4}\rho(t,z)+\epsilon.
\end{align*}
Similarly, we have
\begin{displaymath}
 \rho(t,z)-\rho(t,y)\leq c\epsilon^{1/4}\rho(t,z)+\epsilon,
\end{displaymath}
which gives the local H\"older continuity of $\rho$ in $y$. 

Now, we will use the same idea to obtain the local H\"older continuity of $\rho$ in $t$. Let $0< t_2-t_1\leq\epsilon$ and let $t_1-s=\sqrt{\epsilon}$. Then we have
\begin{displaymath}
\bar{\rho}_{s,u}(t_1-s,y)\leq \frac{1}{(2\pi \sqrt{\epsilon})^{\frac{d}{2}}}\prod_{i=1}^d[e^{-\frac{(|u^i-y^i|\epsilon^{-\frac{1}{4}}-C\epsilon^{\frac{1}{4}})^2}{2}}(1+c\epsilon^\frac{1}{4})]
\end{displaymath}
and
\begin{displaymath}
\bar{\rho}_{s,u}(t_2-s,y)\geq \frac{1}{[2\pi(\sqrt{\epsilon}+\epsilon)]^{\frac{d}{2}}}\prod_{i=1}^d[e^{-\frac{[|u^i-y^i|(\epsilon+\sqrt{\epsilon})^{-\frac{1}{2}}+C(\sqrt{\epsilon}+\epsilon)^{\frac{1}{2}}]^2}{2}}(1-c\epsilon^{1/4})].
\end{displaymath}
When $|u-y|\leq\epsilon^{\frac{1}{8}}$, we have
\begin{align*}
&[|u^i-y^i|(\epsilon+\sqrt{\epsilon})^{-\frac{1}{2}}+C(\sqrt{\epsilon}+\epsilon)^{\frac{1}{2}}]^2\\
&=|u^i-y^i|(\epsilon+\sqrt{\epsilon})^{-1}+2C|u^i-y^i|+C^2(\sqrt{\epsilon}+\epsilon)\\
&\leq |u^i-y^i|\epsilon^{-\frac{1}{2}}+c\epsilon^{\frac{1}{8}}\\
&\leq (|u^i-y^i|\epsilon^{-\frac{1}{4}}-C\epsilon^{\frac{1}{4}})^2+c\epsilon^{\frac{1}{8}}.
\end{align*}
Therefore, we can deduce
\begin{align*}
&\frac{1}{[2\pi(\sqrt{\epsilon}+\epsilon)]^{\frac{d}{2}}}\prod_{i=1}^d[e^{-\frac{[|u^i-y^i|(\epsilon+\sqrt{\epsilon})^{-\frac{1}{2}}+C(\sqrt{\epsilon}+\epsilon)^{\frac{1}{2}}]^2}{2}}(1-c\epsilon^\frac{1}{4})]\\
&\geq\frac{1}{(2\pi\sqrt{\epsilon})^{\frac{d}{2}}(1+\sqrt{\epsilon})^{\frac{d}{2}}}\prod_{i=1}^d[e^{-\frac{(|u^i-y^i|\epsilon^{-\frac{1}{4}}-C\epsilon^{1/4})^2+c\epsilon^{\frac{1}{8}}}{2}}(1-c\epsilon^{\frac{1}{4}})]\\
&\geq  (1-\epsilon^{\frac{1}{8}})\frac{1}{(2\pi \sqrt{\epsilon})^{\frac{d}{2}}}\prod_{i=1}^d[e^{-\frac{(|u^i-y^i|\epsilon^{-\frac{1}{4}}-C\epsilon^{\frac{1}{4}})^2}{2}}].
\end{align*}
This gives us
\begin{displaymath}
\frac{\bar{\rho}_{s,u}(t_1-s,y)}{\bar{\rho}_{s,u}(t_2-s,y)}\leq 1+c\epsilon^{\frac{1}{8}}.
\end{displaymath}
Again, when $|u-y|\geq \epsilon^{\frac{1}{8}}$, we have
\begin{displaymath}
 \bar{\rho}_{s,u}(t_1-s,y)\leq c\epsilon^{-\frac{d}{4}}e^{-c\epsilon^{-\frac{1}{4}}}\leq\epsilon.
\end{displaymath}
So, for general $u\in\mathbb{R}^d$, we would have
\begin{displaymath}
 \bar{\rho}_{s,u}(t_1-s,y)-\bar{\rho}_{s,u}(t_2-s,y)\leq c\epsilon^{1/8}\bar{\rho}_{s,u}(t_2-s,y)+\epsilon.
\end{displaymath}
Thus, we have
\begin{align*}
\rho(t_1,y)-\rho(t_2,y)&= \int_{\mathbb{R}^d}\rho(s,u)[\bar{\rho}_{s,u}(t_1-s,y)-\bar{\rho}_{t_2,u}(t_2-s,y)]du\\
&\leq\int_{\mathbb{R}^d}\rho(s,u)[c\epsilon^{1/8}\bar{\rho}_{s,u}(t_2-s,y)+\epsilon]du\\
&=c\epsilon^{1/8}\rho(t_2,y)+\epsilon.
\end{align*}
Similarly, we would have 
\begin{displaymath}
\rho(t_2,y)-\rho(t_1,y)\leq c\epsilon^{1/8}\rho(t_2,y)+\epsilon
\end{displaymath}
and this proves that the transition density is locally H\"older continuous in $t$.
\end{proof} 

\subsection{Proof of Theorem \ref{diff}}
Now, we can start proving Theorem \ref{diff}. As mentioned earlier, in \cite{bmk}, Norris has already proved the result in the case when particles' free motions are pure Brownian. Our proof of Theorem \ref{diff} will be based on his approach and we will use Corollary \ref{vol} to deal with the diffusivity term that depends on the position of the particles and use Theorem \ref{drift} to deal with the drift term.
\begin{proof}
From now on, we shall write $C$ as a constant depending only on $d$ and $R$, and the value of $C$ might change from line to line. We set $X(t)=X^1(t)-X^2(t)$. As $a_i$ are scalars, we would like to use Dubins-Schwarz theorem to relate $X$ to a Brownian motion with drift. Let
\begin{displaymath}
 A(t)=\int_0^t[a_1(X^1(s))+a_2(X^2(s))]ds.
\end{displaymath}
So, $A(t)$ is the quadratic variation process of $X(t)$. Set $\tau_t$ be the stopping time such that $A(\tau_t)=t$, then we have
\begin{displaymath}
dt=dA(\tau_t)=[a_1(X^1(\tau_t))+a_2(X^2(\tau_t))]d\tau_t.
\end{displaymath}
Let $Y(t)=X(\tau_t)$ and
\begin{displaymath}
B_t=\int_0^{\tau_t}\sqrt{a_1(X^1(s))}dB^1_s-\int_0^{\tau_t}\sqrt{a_2(X^2(s))}dB^2_s.
\end{displaymath}
Then we have, by Dubins-Schwarz theorem, that $B$ is a Brownian motion and
\begin{displaymath}
dB_t=\sqrt{a_1(X^1(s))}dB^1_{\tau_t}-\sqrt{a_2(X^2(s))}dB^2_{\tau_t}.
\end{displaymath}
Note that $A(t)$ is continuous and strictly increasing and goes from $0$ to infinity, we have $\tau_{A(t)}=t$ and $Y(A(t))=X(\tau_{A(t)})=X(t)$. Moreover, let 
\begin{displaymath}
b(t)=\frac{b_1(X^1_{\tau_t})-b_2(X^2_{\tau_t})}{a_1(X^1_{\tau_t})+a_2(X^1_{\tau_t})}, 
\end{displaymath}
 we have
\begin{align*}
 dY(t)&=dX^1(\tau_t)-dX^2(\tau_t)\\
&=\sqrt{a_1(X^1(\tau_t))}dB^1(\tau_t)+b_1(X^1_{\tau_{t}})d(\tau_t)-\sqrt{a_2(X^2(\tau_t))}dB^2(\tau_t)-b_2(X^2_{\tau_t})d(\tau_t)\\
&=dB_t+[b_1(X^1_{\tau_t})-b_2(X^2_{\tau_t})][a_1(X^1_{\tau_t})+a_2(X^1_{\tau_t})]^{-1}dt\\
&=dB_t+b(t)dt.
\end{align*}
Because $|b_i|$'s are bounded above by $R$ and $a_i$'s are bounded below by $R^{-1}$, we have $|b|$ is bounded above by $R^2$. The above equality relates $X_t$ to a Brownian motion with bounded drift, which we can then deal with using Theorem \ref{drift}.

Now, we look back at the derivation of (\ref{eqn1}). To make the argument rigorous, it remains to show that $M_t$ is uniformly bounded up to $T$. Using Theorem \ref{aronson}, we have when $s,T>t>0$,
\begin{align*}
p_1(t,X^1_t;s,z)p_2(t,X_t^2;s,z)&\leq C(s-t)^{-d}e^{C(s-t)}\exp(-\frac{|X^1_t-z|^2+|X^2_t-z|^2}{C(s-t)})\\
&\leq C(s-t)^{-d}e^{Cs}\exp(-\frac{|X^1_t-X^2_t|^2}{4C(s-t)})\\
&\leq C(s-t)^{-d}e^{Cs}\exp(-\frac{r_N^2}{4C(s-t)}),
\end{align*}
which is indeed bounded in $t$.
We will now estimate the right hand side of equation (\ref{eqn1}). We will show that the contribution when $s$ is "far" from $T$ or when $z$ is "far" from $X(T)$ to the integral inside the expectation is small so that we can approximate it by 
\begin{displaymath}
g(T,X(T))\int_T^R\int_{\mathbb{R}^d}p_1(T,X^1_T;s,z)p_2(T,X_T^2;s,z)K(z)dzds.
\end{displaymath}
Then, by using Corollary \ref{vol} we can relate the above integral with the expectation of a functional of $X$. Let $\rho(s,.)$ be the probability density of $Y(s)$. We can then use the relation $Y(A(t))=X(t)$ to write it as a functional of $Y$ and use the same idea as in Corollary \ref{vol} again to approximate it as an integral involving $\rho$ instead of $p_i$. Finally, when $s$ is "close" to $T$, we will use Theorem \ref{drift} to approximate $\rho$ as the transition density of a standard Brownian motion which we can then evaluate. 

We first claim that
\begin{displaymath}
\int_{\epsilon^2}^{R}\int_{\mathbb{R}^d}p_1(0,x_1;s,z)p_2(0,x_2;s,z)a(z)dzds\leq C\epsilon^{2-d}.\numberthis\label{eqn2}
\end{displaymath}
Because we assumed $a_i$ to be H\"older continuous, and thus also uniformly continuous, we can use Corollary \ref{vol} to obtain
\begin{align*}
&\int_{\epsilon^2}^{R}\int_{\mathbb{R}^d}p_1(0,x_1;s,z)p_2(0,x_2;s,z)a(z)dzds\\
&=\lim_{h\rightarrow 0}V(h)^{-1}\mathbb{E}[\int_{\epsilon^2}^{R}\mathbf{1}_{|X_s|<h}a(X_s^2)ds].
\end{align*}
By uniform continuity of $a_i$, we can also deduce
\begin{align*}
&\lim_{h\rightarrow 0}V(h)^{-1}\mathbb{E}[\int_{\epsilon^2}^{R}\mathbf{1}_{|X_s|<h}a(X_s^2)ds]\\
&=\lim_{h\rightarrow 0}V(h)^{-1}\mathbb{E}[\int_{\epsilon^2}^{R}\mathbf{1}_{|X_s|<h}(a_1(X_s^2)+a_2(X_s^2))ds]\\
&=\lim_{h\rightarrow 0}V(h)^{-1}\mathbb{E}[\int_{\epsilon^2}^{R}\mathbf{1}_{|X_s|<h}(a_1(X_s^1)+a_2(X_s^2))ds].\\
\end{align*}
Now, recall that $Y(A(t))=X(t)$ and
\begin{displaymath}
dA(\tau_t)=[a_1(X^1(\tau_t))+a_2(X^2(\tau_t))]d\tau_t,
\end{displaymath}
\begin{displaymath}
2R^{-1}\leq[a_1(X^1(\tau_t))+a_2(X^2(\tau_t))]\leq 2R,
\end{displaymath}
we have
\begin{align*}
&\lim_{h\rightarrow 0}V(h)^{-1}\mathbb{E}[\int_{\epsilon^2}^{R}\mathbf{1}_{|X_s|<h}(a_1(X_s^1)+a_2(X_s^2))ds]\\
&=\lim_{h\rightarrow 0}V(h)^{-1}\mathbb{E}[\int_{\epsilon^2}^{R}\mathbf{1}_{|Y_{A(s)}|<h}(a_1(X_s^1)+a_2(X_s^2))ds]\\
&=\lim_{h\rightarrow 0}V(h)^{-1}\mathbb{E}(\int_{\epsilon^2}^{R}\mathbf{1}_{|Y_{A(s)}|<h}dA(s))\\
&\leq\lim_{h\rightarrow 0}V(h)^{-1}\mathbb{E}(\int_{2R^{-1}\epsilon^2}^{2R^2}\mathbf{1}_{|Y_s|<h}ds).
\end{align*}
Because $\lim_{h\rightarrow 0}V(h)^{-1}\mathbb{E}(\mathbf{1}_{|Y_s|<h})$ converges uniformly in $s\in [2R^{-1}\epsilon^2,2R^2]$ to $\rho(s,0)$, we have
\begin{align*}
&\lim_{h\rightarrow 0}V(h)^{-1}\mathbb{E}(\int_{2R^{-1}\epsilon^2}^{2R^2}\mathbf{1}_{|Y_s|<h}ds)\\
&=\mathbb{E}(\int_{2R^{-1}\epsilon^2}^{2R^2}\lim_{h\rightarrow 0}V(h)^{-1}\mathbf{1}_{|Y_s|<h}ds)\\
&=\int_{2R^{-1}\epsilon^2}^{2R^2}\rho(s,0)ds.
\end{align*}
Now, using the result from Theorem \ref{drift}, we have for $2R^{-1}\epsilon^2<s<2R^2$,
\begin{align*}
 \rho(s,0)&\leq\frac{1}{(2\pi s)^{\frac{d}{2}}}\prod_{i=1}^d(\int^\infty_{|x_1-x_2|/\sqrt{s}}ze^{-(z-C\sqrt{s})^2/2}dz)\\
&\leq Cs^{-\frac{d}{2}}\prod_{i=1}^d(\int^\infty_{0}ze^{-(z-CR)^2/2}dz)\\
&\leq Cs^{-\frac{d}{2}}.
\end{align*}
So, we have
\begin{displaymath}
 \int_{2R^{-1}\epsilon^2}^{2R^2}\rho(s,0)ds\leq C\int_{2R^{-1}\epsilon^2}^{2R^2}s^{-\frac{d}{2}}\leq C\epsilon^{2-d},
\end{displaymath}
as required. 

Now, we claim that when $|x_1-x|\leq \frac{\epsilon}{2}$ and $|x_2-x|\leq \frac{\epsilon}{2}$, then
\begin{displaymath}
 \int_0^R\int_{|z-x|>\epsilon}p_1(0,x_1;s,z)p_2(0,x_2;s,z)dzds\leq C\epsilon^{2-d}.\numberthis\label{eqn3}
\end{displaymath}
For this, we can just apply Theorem \ref{aronson} to obtain
\begin{align*}
 &\int_0^R\int_{|z-x|>\epsilon}p_1(0,x_1;s,z)p_2(0,x_2;s,z)dzds\\
&\leq\int_0^R\int_{|z-x|>\epsilon}C^2s^{-d}\exp\{-\frac{|x_1-z|^2+|x_2-z|^2}{Cs}\}e^{2Cs}dzds\\
&\leq\int_0^R\int_{|z-x|>\epsilon}C^2e^{2CR}s^{-d}\exp\{-\frac{|x_1-z|^2+|x_2-z|^2}{Cs}\}dzds\\
&\leq C\int_{|z-x|>\epsilon}\int_0^Rs^{-d}\exp\{-\frac{|x_1-z|^2+|x_2-z|^2}{Cs}\}dsdz.
\end{align*}
Substitute $u=\frac{s}{|x_1-z|^2+|x_2-z|^2}$, we have
\begin{align*}
&C\int_{|z-x|>\epsilon}\int_0^Rs^{-d}\exp\{-\frac{|x_1-z|^2+|x_2-z|^2}{Cs}\}dsdz\\
&\leq C\int_{|z-x|>\epsilon}(|x_1-z|^2+|x_2-z|^2)^{-d+1}\int_0^{\frac{R}{|x_1-z|^2+|x_2-z|^2}}u^{-d}\exp\{-\frac{1}{Cu}\}dudz\\
&\leq C\int_{|z-x|>\epsilon}(|x_1-z|^2+|x_2-z|^2)^{-d+1}\int_0^{\infty}u^{-d}\exp\{-\frac{1}{Cu}\}dudz\\
&\leq C\int_{|z-x|>\epsilon}((|z-x|-\frac{\epsilon}{2})^2)^{-d+1}dz\\
&\leq C\int_{\epsilon}^\infty r^{d-1}((r-\frac{\epsilon}{2})^2)^{-d+1}dr\\
&\leq C\int_{\epsilon}^\infty r^{d-1}((r)^2)^{-d+1}dr\\
&\leq C\epsilon^{2-d},
\end{align*}
as required. 

Combining (\ref{eqn2}) and (\ref{eqn3}) we can estimate the integral on the right hand side of (\ref{eqn1}) by
\begin{align*}
&\big |\int_T^R\int_{\mathbb{R}^d}p_1(T,X^1_T;s,z)p_2(T,X_T^2;s,z)g(s,z)K(z)dzds\\
&-g(T,X(T))\int_T^R\int_{\mathbb{R}^d}p_1(T,X^1_T;s,z)p_2(T,X_T^2;s,z)K(z)dzds\big |\\
&\leq C\epsilon^{2-d}+\phi_g(\epsilon)\int_T^R\int_{\mathbb{R}^d}p_1(T,X^1_T;s,z)p_2(T,X_T^2;s,z)K(z)dzds,\numberthis\label{eqn4}
\end{align*}
provided $T<R$. So, we will now aim to estimate the value of
\begin{displaymath}
\int_T^R\int_{\mathbb{R}^d}p_1(T,X^1_T;s,z)p_2(T,X_T^2;s,z)K(z)dzds.
\end{displaymath}
By the same argument as earlier, we have
\begin{align*}
&\int_{0}^{\epsilon^2}\int_{\mathbb{R}^d}p_1(0,x_1;s,z)p_2(0,x_2;s,z)a(z)dzds\\
&=\lim_{h\rightarrow 0}V(h)^{-1}\mathbb{E}(\int_{0}^{\epsilon^2}\mathbf{1}_{|Y_{A(s)}|<h}dA(s))\\
&\geq\int_{0}^{R^{-1}\epsilon^2}\rho(s,0)ds,
\end{align*}
and
\begin{displaymath}
\int_{0}^{\epsilon^2}\int_{\mathbb{R}^d}p_1(0,x_1;s,z)p_2(0,x_2;s,z)a(z)dzds\leq\int_{0}^{2R\epsilon^2}\rho(s,0)ds.
\end{displaymath}
Next, we will use Theorem $\ref{drift}$ to approximate $\rho(t,0)$ for small $t$. Let $q$ be the transition density of a standard Brownian motion in $\mathbb{R}^d$, and suppose $2|x_1-x_2|<\epsilon$, then for $t\leq R\epsilon^2$ we can use similar derivation as in (\ref{eqn0}) to obtain
\begin{align*}
\rho(t,0)&\geq\frac{1}{(2\pi t)^{\frac{d}{2}}}\prod_{i=1}^d[e^{-\frac{(|x_1^i-x_2^i|/\sqrt{t}+C\sqrt{t})^2}{2}}(1-C\sqrt{t})]\\
&\geq \frac{1}{(2\pi t)^{\frac{d}{2}}}\prod_{i=1}^de^{-\frac{(|x_1^i-x_2^i|/\sqrt{t})^2}{2}}e^{-\frac{2C(|x_1^i-x_2^i|)+C^2t}{2}}(1-C\epsilon)\\
&\geq \frac{1}{(2\pi t)^{\frac{d}{2}}}\prod_{i=1}^d[e^{-\frac{(|x_1^i-x_2^i|/\sqrt{t})^2}{2}}](1-C\epsilon)\\
&\geq q(0,x_1-x_2;t,0)(1-C\epsilon),
\end{align*}
and 
\begin{align*}
\rho(t,0)&\leq\frac{1}{(2\pi t)^{\frac{d}{2}}}\prod_{i=1}^d[e^{-\frac{(|x_1^i-x_2^i|/\sqrt{t}-C\sqrt{t})^2}{2}}(1+C\sqrt{t})]\\
&\leq \frac{1}{(2\pi t)^{\frac{d}{2}}}\prod_{i=1}^de^{-\frac{(|x_1^i-x_2^i|/\sqrt{t})^2}{2}}e^{\frac{2C(|x_1^i-x_2^i|)-C^2t}{2}}(1+C\epsilon)\\
&\leq \frac{1}{(2\pi t)^{\frac{d}{2}}}\prod_{i=1}^d[e^{-\frac{(|x_1^i-x_2^i|/\sqrt{t})^2}{2}}](1+C\epsilon)\\
&\leq q(0,x_1-x_2;t,0)(1+C\epsilon).
\end{align*}
Note that
\begin{displaymath}
\int_0^{\infty}q(0,x_1-x_2;s,0)ds=\int_0^\infty\frac{1}{(2\pi t)^{\frac{d}{2}}}e^{-\frac{|x_1-x_2|^2}{2t}}dt.
\end{displaymath}
We now make the substitution $u=\frac{t}{|x_1-x_2|^2}$ and recall that
\begin{displaymath}
\frac{1}{c_d}=\int_0^\infty\frac{1}{(2\pi t)^{\frac{d}{2}}}e^{-\frac{1}{2t}}dt
\end{displaymath}
to obtain
\begin{align*}
&\int_0^\infty\frac{1}{(2\pi t)^{\frac{d}{2}}}e^{-\frac{|x_1-x_2|^2}{2t}}dt\\
&=|x_1-x_2|^{2-d}\int_0^\infty\frac{1}{(2\pi u)^{\frac{d}{2}}}e^{-\frac{1}{2u}}du\\
&=\frac{1}{c_d}|x_1-x_2|^{2-d}.
\end{align*}
We also know that
\begin{displaymath}
 \int_{R^{-1}\epsilon^2}^{\infty}q(0,x_1-x_2;s,0)ds\leq \int_{R^{-1}\epsilon^2}^\infty\frac{1}{(2\pi s)^{\frac{d}{2}}}ds\leq C\epsilon^{2-d}, 
\end{displaymath}
and thus we have for $\epsilon^2\leq t\leq R$,
\begin{align*}
 &|x_1-x_2|^{2-d}-c_d\int_{0}^{t}\int_{\mathbb{R}^d}p_1(0,x_1;s,z)p_2(0,x_2;s,z)a(z)dzds\\
&\leq |x_1-x_2|^{2-d}-c_d\int_{0}^{\epsilon^2}\int_{\mathbb{R}^d}p_1(0,x_1;s,z)p_2(0,x_2;s,z)a(z)dzds\\
&\leq |x_1-x_2|^{2-d}-c_d\int_{0}^{R^{-1}\epsilon^2}\rho(s,0)ds\\
&\leq |x_1-x_2|^{2-d}-(1-C\epsilon)c_d\int_{0}^{R^{-1}\epsilon^2}q(0,x_1-x_2;s,0)ds\\
&\leq |x_1-x_2|^{2-d}-(1-C\epsilon)c_d[\int_{0}^{\infty}q(0,x_1-x_2;s,0)ds-\int_{R^{-1}\epsilon^2}^{\infty}q(0,x_1-x_2;s,0)ds]\\
&\leq |x_1-x_2|^{2-d}-(1-C\epsilon)c_d[\frac{|x_1-x_2|^{2-d}}{c_d}-C\epsilon^{2-d}]\\
&\leq C\epsilon^{2-d}+C\epsilon|x_1-x_2|^{2-d}.
\end{align*}

By a similar method we can obtain for $\epsilon^2\leq t\leq R$,
\begin{align*}
 &c_d\int_{0}^{t}\int_{\mathbb{R}^d}p_1(0,x_1;s,z)p_2(0,x_2;s,z)a(z)dzds-|x_1-x_2|^{2-d}\\
&\leq c_d\int_{0}^{\epsilon^2}\int_{\mathbb{R}^d}p_1(0,x_1;s,z)p_2(0,x_2;s,z)a(z)dzds-|x_1-x_2|^{2-d}\\&+c_d\int_{\epsilon^2}^{R}\int_{\mathbb{R}^d}p_1(0,x_1;s,z)p_2(0,x_2;s,z)a(z)dzds\\
&\leq c_d\int_{0}^{2R\epsilon^2}\rho(s,0)ds-|x_1-x_2|^{2-d}+C\epsilon^{2-d}\\
&\leq c_d(1+C\epsilon)\int_{0}^{2R\epsilon^2}q(0,x_1-x_2;s,0)ds-|x_1-x_2|^{2-d}+C\epsilon^{2-d}\\
&\leq c_d(1+C\epsilon)\int_{0}^{\infty}q(0,x_1-x_2;s,0)ds-|x_1-x_2|^{2-d}+C\epsilon^{2-d}\\
&\leq c_d(1+C\epsilon)\frac{|x_1-x_2|^{2-d}}{c_d}-|x_1-x_2|^{2-d}+C\epsilon^{2-d}\\
&\leq C\epsilon|x_1-x_2|^{2-d}+C\epsilon^{2-d}.
\end{align*}

From these combined with (\ref{eqn2}) and (\ref{eqn3}) , we will have that if $|x_1-x|\leq \frac{\epsilon}{2}$ and $|x_2-x|\leq \frac{\epsilon}{2}$, then for $\epsilon^2\leq t\leq R$,
\begin{align*}
&\abs{g(0,x)|x_1-x_2|^{2-d}-c_d\int_0^{t}\int_{\mathbb{R}^d}p_1(0,x_1;s,z)p_2(0,x_2;s,z)a(z)g(s,z)dzds}\\
&\leq |g(0,x)|(\big| |x_1-x_2|^{2-d}-c_d\int_{0}^{t}\int_{\mathbb{R}^d}p_1(0,x_1;s,z)p_2(0,x_2;s,z)a(z)dzds\big|)\\
&+\big|c_d\int_0^{t}\int_{\mathbb{R}^d}p_1(0,x_1;s,z)p_2(0,x_2;s,z)a(z)|g(s,z)-g(0,x)|dzds\big|\\
&\leq\|g\|(C\epsilon^{2-d}+C\epsilon|x_1-x_2|^{2-d})+\big|c_d\int_0^{\epsilon^2}\int_{|z-x|\leq\epsilon}p_1(0,x_1;s,z)p_2(0,x_2;s,z)a(z)|g(s,z)-g(0,x)|dzds\big|\\
&+\big|c_d\int_0^{\epsilon^2}\int_{\mathbb{R}^d}p_1(0,x_1;s,z)p_2(0,x_2;s,z)a(z)|g(s,z)-g(0,x)|dzds\big|\\
&+\big|c_d\int_0^{t}\int_{|z-x|\geq\epsilon}p_1(0,x_1;s,z)p_2(0,x_2;s,z)a(z)|g(s,z)-g(0,x)|dzds\big|\\
&\leq\|g\|(C\epsilon^{2-d}+C\epsilon|x_1-x_2|^{2-d})+\phi_g(\epsilon)c_d\int_0^{t}\int_{\mathbb{R}^d}p_1(0,x_1;s,z)p_2(0,x_2;s,z)a(z)dzds\\
&\leq\|g\|(C\epsilon^{2-d}+C\epsilon|x_1-x_2|^{2-d})+\phi_g(\epsilon)|x_1-x_2|^{2-d}.
\end{align*}
Recall $K(z)=c_da(z)r^{d-2}$ and on the event $\{T<R\}$ we have $X^1_T-X^2_T=rN^{-\frac{1}{d-2}}$. Therefore, 
\begin{align*}\label{stop}\numberthis
 &|Ng(T,X(T))-\int_T^{R}\int_{\mathbb{R}^d} K(z)p_1(T,X^1_T;s,z)p_2(T,X^2_T;s,z)g(s,z)dzds|\\
&=r^{d-2}\big|g(T,X(T))|X_T^1-X^2_T|^{2-d}-c_d\int_T^{R}\int_{\mathbb{R}^d} a(z)p_1(T,X^1_T;s,z)p_2(T,X^2_T;s,z)g(s,z)dzds\big|\\
&=r^{d-2}\big|g(T,X(T))|X_T^1-X^2_T|^{2-d}-c_d\int_0^{(R-T)}\int_{\mathbb{R}^d} a(z)p_1(0,X^1_T;s,z)p_2(0,X^2_T;s,z)g(s,z)dzds\big|\\
&\leq C[\epsilon^{2-d}\|g\|+N(\phi_g(\epsilon)+\epsilon)].
\end{align*}
Recall (\ref{eqn1}), we have 
\begin{align*}
 &|N\mathbb{E}(g(T,X(T))\mathbf{1}_{T<R})-\int_0^R\int_{\mathbb{R}^d} K(z)p_1(0,x_1;s,z)p_2(0,x_2;s,z)g(s,z)dzds|\numberthis\label{eqnnn}\\
&\leq C[\epsilon^{2-d}\|g\|+N(\epsilon+\phi_g(\epsilon))]\mathbb{P}(T\leq R).
\end{align*}
Now, we will give a bound on $\mathbb{P}(T\leq R)$. We have
\begin{align*}
 &\int_{0}^{2R}\int_{\mathbb{R}^d}p_1(0,x_1;s,z)p_2(0,x_2;s,z)a(z)dzds\\
&=\mathbb{E}[\int_{T}^{2R}\int_{\mathbb{R}^d}p_1(T,X^1_T;s,z)p_2(T,X_T^2;s,z)a(z)dzds]\\
&\geq \mathbb{E}[\mathbf{1}_{T<R}\int_{T}^{T+R}\int_{\mathbb{R}^d}p_1(T,X^1_T;s,z)p_2(T,X_T^2;s,z)a(z)dzds].
\end{align*}
Using the result from Theorem \ref{aronson} we would have
\begin{align*}
&\int_{T}^{T+R}p_1(T,X^1_T;s,z)p_2(T,X_T^2;s,z)a(z)dzds\\
&\geq \frac{1}{C}\int_{\mathbb{R}^d}\int_0^Rs^{-d}\exp\{-C\frac{|X^1_T-z|^2+|X^2_T-z|^2}{s}\}dsdz\\
&\geq \frac{(rN^{-1/(d-2)})^{2-d}}{C}\geq\frac{N}{C}.
\end{align*}
Therefore, we have
\begin{displaymath}
 N\mathbb{P}( T\leq R)\leq C\int_{0}^{2R}\int_{\mathbb{R}^d}p_1(0,x_1;s,z)p_2(0,x_2;s,z)a(z)dzds.
\end{displaymath}
Now, using Theorem \ref{aronson} and similar argument as earlier, we obtain
\begin{displaymath}
\int_{0}^{2R}\int_{\mathbb{R}^d}p_1(0,x_1;s,z)p_2(0,x_2;s,z)a(z)dzds.
\leq C\int_{0}^{4R^2}\rho(s,0)ds\leq C\int_{0}^{4R^2}s^{-\frac{d}{2}}e^{-\frac{|x_1-x_2|^2}{Cs}}ds.
\end{displaymath}
By substituting $u=\frac{Cs}{|x_1-x_2|^2}$, we have
\begin{align*}
&C\int_{0}^{4R^2}s^{-\frac{d}{2}}e^{-\frac{|x_1-x_2|^2}{Cs}}ds\\
&\leq C|x_1-x_2|^{2-d}\int_{0}^{4R^2}u^{-\frac{d}{2}}e^{\frac{-1}{u}}du\\&
\leq C|x_1-x_2|^{2-d}.
\end{align*}
So, we conclude that
\begin{displaymath}
\mathbb{P}( T\leq R)\leq \frac{C}{N}(|x_1-x_2|)^{2-d}.
\end{displaymath}
Plugging this into (\ref{eqnnn}), we obtain 
\begin{align*}
 &|N\mathbb{E}(g(T,X(T))\mathbf{1}_{T<R})-\int_0^R\int_{\mathbb{R}^d} K(z)p_1(T,X^1_T;s,z)p_2(T,X^2_T;s,z)g(s,z)dzds|\\
&\leq C[\epsilon^{2-d}\|g\|/N+\epsilon+\phi_g(\epsilon)](|x_1-x_2|)^{2-d},
\end{align*}
as desired.
\end{proof}

\subsection{Application}
As an application, we now prove Corollary \ref{period}.

\begin{proof}
We first look at the limit
\begin{displaymath}
\lim_{\lambda\rightarrow 0}\lim_{N\rightarrow\infty}N\mathbb{E}(g(T,X^\lambda_T)\mathbf{1}_{T<R}).
\end{displaymath}
For fixed $\lambda$, Theorem $\ref{diff}$ provides us the limit of $N\mathbb{E}(g(T,X^\lambda_T)\mathbf{1}_{T<R}$ as $N\rightarrow\infty$. Then we can use homogenization results for Brownian motions under periodic drift to find the limit as $\lambda\rightarrow 0$. 

For $i=1,2$, let $p^{\lambda}_i$ be the transition density of $X^\lambda_i$. It is known that
\begin{displaymath}
p^{\lambda}_i(0,x_1;s,z)\rightarrow p_i(0,x_1;s,z)
\end{displaymath} 
pointwise as $\lambda\rightarrow 0$, see \cite{JKO} for example. On the other hand, Theorem \ref{diff} tells us that \begin{displaymath}
\lim_{N\rightarrow\infty}|N\mathbb{E}(g(T,X^\lambda_T)\mathbf{1}_{T<R})-\int_0^R\int_{\mathbb{R}^d}Kp^\lambda_1(0,x_1;s,z)p^\lambda_2(0,x_2;s,z)g(s,z)dzds|\rightarrow 0.
\end{displaymath}
Now, we want to prove that 
\begin{displaymath}
\int_0^R\int_{\mathbb{R}^d}p^{\lambda}_1(0,x_1;s,z)p^{\lambda}_2(0,x_2;s,z)g(s,z)dzds\rightarrow \int_0^R\int_{\mathbb{R}^d}p_1(0,x_1;s,z)p_2(0,x_2;s,z)g(s,z)dzds,
\end{displaymath}
as $\lambda\rightarrow 0$. We know that
\begin{displaymath}
p^{\lambda}_1(0,x_1;s,z)p^{\lambda}_2(0,x_2;s,z)g(s,z)\rightarrow p_1(0,x_1;s,z)p_2(0,x_2;s,z)g(s,z)
\end{displaymath}
pointwise as $\lambda\rightarrow 0$. The results in $\cite{Norris}$ suggests that, for $t\leq R$, 
\begin{displaymath}
p^\lambda_i(0,x;t,y)\leq Ct^{-d/2}\exp\{-|y-x|^2/Ct\},
\end{displaymath}
for some constant $C$ independent of $\lambda$. So, we have
\begin{align*}
p^{\lambda}_1(0,x_1;s,z)p^{\lambda}_2(0,x_2;s,z)g(s,z)&\leq Ct^{-d}\exp\{-(|z-x_1|^2+|z-x_2|^2)/Ct\}\\
&\leq Ct^{-d}\exp\{-(|z-\frac{x_1+x_2}{2}|^2+\frac{|x_1-x_2|^2}{2})/Ct\}.
\end{align*}
Therefore, 
\begin{align*}
&\int_{\mathbb{R}^d}Cs^{-d}\exp\{(-|z-x_1|^2-|z-x_2|^2)/Cs\}dz\\
&\leq \int_{\mathbb{R}^d}Cs^{-d}\exp\{-(|z-\frac{x_1+x_2}{2}|^2+\frac{|x_1-x_2|^2}{2})/Cs\}dz\\
&\leq e^{-\frac{|x_1-x_2|^2}{Cs}}\int_{\mathbb{R}^d}Cs^{-d}\exp\{-(|z-\frac{x_1+x_2}{2}|^2)/Cs\}dz\\
&\leq Cs^{-d/2}e^{-\frac{|x_1-x_2|^2}{Cs}}.
\end{align*}
Now, because $\lim_{s\rightarrow 0}Cs^{-d/2}e^{-\frac{|x_1-x_2|^2}{Cs}}=0$ and $Cs^{-d/2}e^{-\frac{|x_1-x_2|^2}{Cs}}$ is continuous in $s\geq 0$, we know that 
\begin{displaymath}
\sup_{0\leq s\leq R}\int_{\mathbb{R}^d}Cs^{-d}\exp\{(-|z-x_1|^2-|z-x_2|^2)/Cs\}dz<\infty,
\end{displaymath}
and thus
\begin{displaymath}
\int_0^R\int_{\mathbb{R}^d}Cs^{-d}\exp\{(-|z-x_1|^2-|z-x_2|^2)/Cs\}dzds<\infty.
\end{displaymath}
Then, by dominated convergence theorem, we obtain 
\begin{displaymath}
\int_0^R\int_{\mathbb{R}^d}p^{\lambda}_1(0,x_1;s,z)p^{\lambda}_2(0,x_2;s,z)g(s,z)dzds\rightarrow \int_0^R\int_{\mathbb{R}^d}p_1(0,x_1;s,z)p_2(0,x_2;s,z)g(s,z)dzds.
\end{displaymath}
Therefore, we can conclude
\begin{displaymath}
\lim_{\lambda\rightarrow 0}\lim_{N\rightarrow\infty}|N\mathbb{E}(g(T,X^\lambda_T)\mathbf{1}_{T<R})-\int_0^R\int_{\mathbb{R}^d}Kp_1(0,x_1;s,z)p_2(0,x_2;s,z)g(s,z)dzds|\rightarrow 0,
\end{displaymath}
as desired.

Now, we consider the limit
\begin{displaymath}
\lim_{N\rightarrow\infty}\lim_{\lambda\rightarrow 0}N\mathbb{E}(g(T,X^\lambda_T)\mathbf{1}_{T<R}).
\end{displaymath}
It is known \cite{N13} that the processes $X^\lambda_i$ converge weakly to $X^i$ with $X^i$ Brownian motions with diffusivities $\bar{a}_i$. We can view $g\mathbf{1}_{T<R}$ as a functional on the processes $X^1$ and $X^2$. Moreover, $g$ is continuous at $(X^1,X^2)$ unless one of the following events happens:

(i) $T=R$

(ii) $X^1$ or $X^2$ is not continuous

(iii)$T<R$ and there exists $\epsilon>0$ such that $\inf_{T\leq t\leq T+\epsilon}|X^1_t-X^2_t|=rN^{-1/(d-2)}$.

Because all these events happen with probability $0$, we know that $g(T,X(T))$ is a bounded  and almost everywhere continuous functional on $X^1$ and $X^2$. Therefore, by weak convergence we have 
\begin{displaymath}
\lim_{\lambda\rightarrow 0}\mathbb{E}(g(T,X^\lambda_T))=\mathbb{E}(g(\bar{T},X_{\bar{T}})),
\end{displaymath}
where $\bar{T}$ is the collision time of two Brownian particles with diffusivities $\bar{a}_1$ and $\bar{a}_2$ respectively and $X_{\bar{T}}$ is their centre of mass at time $\bar{T}$. Using Theorem \ref{diff} again we have
\begin{align*}
 &\lim_{N\rightarrow\infty}\lim_{\lambda\rightarrow 0}|N\mathbb{E}(g(T,X^\lambda_T)\mathbf{1}_{T<R})-\int_0^R\int_{\mathbb{R}^d} \bar{K}p_1(0,x_1;s,z)p_2(0,x_2;s,z)g(s,z)dzds|\rightarrow 0,
\end{align*}
as required.
\end{proof}

\section{Estimates for Ornstein-Uhlenbeck particles}
We consider in this section two particles starting from $x_1$ and $x_2$. For $i=1,2$ let $V_i^N$ and $X_i^N$ be their velocities and positions respectively and we assume they are modelled by the Ornstein-Uhlenbeck processes satisfying
\begin{align*}
 dV^N_i(t)&=Nb_idB^i_t-N\tau_iV^N_idt,\\
 dX^N_i(t)&=V^N_i(t)dt.
\end{align*}
As $N$ tends to infinity, the position $X^N_i$ converges weakly to a Brownian motion and the rate of the convergence is fixed. So, if we let the radii of the particles decrease slowly enough with $N$, the collision will happen in similar way as in the Brownian case. However, if we consider the case when the radii of the particles decrease sufficiently fast, then when the two particles come close, they are likely to move away from each other with almost constant velocities. We will exploit this to find the collision estimates. In the Brownian case, when two particles come close to each other, they are likely to stay around for a bit longer and will thus have more chance to collide. As a result, the scale of the collision rate will be smaller in the Ornstein-Uhlenbeck case. The aim of this section is to investigate the collision distribution for the Ornstein-Uhlenbeck particles and compare the result with the Brownian case.

\subsection{Proof of Theorem \ref{ou}}
Our strategy for proving Theorem \ref{ou} is to first consider the case where the particles just continue their free motions after they collide and allow them to recollide later. We will divide the time interval $[t_0,t_1]$ into little time intervals, so that in each little interval the velocities of the particles are unlikely to change much. Then, we can make good predictions about whether and where the particles are going to collide in a time interval based on their positions and velocities at the start of the interval. Also, we know the distribution of positions and velocities of the particles at any time, so we can estimate the distribution of the time and place where the particles collide. Then, we will show that allowing the particles to recollide won't change our estimation by much because the particles are very unlikely to collide more than once anyway. This is because after the particles collide, they are likely to continue their free motions with almost constant velocities for a small amount of time and this time turns out to be enough for them to get far away from each other so that they are unlikely to collide again. Now, we will start our proof.
\begin{proof} 
We shall write $C$ as a constant, and $C^N$ as a sequence of constants such that $C^N\rightarrow 0$ as $N\rightarrow\infty$. We allow the values of $C$ and $C^N$ to change from line to line. We know that when $t_1\geq t\geq t_0$, $(V^N_i,X^N_i)$ are bivariate normally distributed with 
\begin{align*}
 &Var(V^N_i)= \frac{Nb^2_i}{2\tau_i}(1-e^{-2N\tau_it}),\\
&Var(X^N_i)=\frac{b_i^2}{\tau_i^2}(t-\frac{2-2e^{-N\tau_it}}{N\tau_i}+\frac{1-e^{-2N\tau_it}}{2N\tau_i})\\
&Cov(X^N_i,V^N_i)=\frac{b_i^2}{2\tau^2_i}(1-2e^{-N\tau_it}+e^{-2N\tau_it}).
\end{align*}
As an approximation, we have 
\begin{align*}
 &|Var(V^N_i)-\frac{Nb^2_i}{2\tau_i}|\leq \frac{C}{N},\\
&|Var(X^N_i)-\frac{b_i^2t}{\tau_i^2}|\leq\frac{C}{N},\\
&|Cov(X^N_i,V^N_i)-\frac{b_i^2}{2\tau^2_i}|\leq\frac{C}{N}.
\end{align*}
Now, we choose a constant $\epsilon>0$ depending on $\alpha$, which is sufficiently small for all needs in the remaining of the proof. Let $k=\frac{1}{2}-\epsilon$, $\beta=\frac{1}{2}-2\epsilon$ and $m=\frac{1}{2}+\epsilon$. Then we can choose a constant $\lambda$ such that
\begin{displaymath}
2(k-1)>\lambda>\frac{2}{9}(m-2\alpha-4).
\end{displaymath}
Let $h_N=\frac{(t_1-t_0)}{\lfloor N^{\beta}/r_N\rfloor}$ and  $t^N_i=ih_N-h_N$. We subdivide $(t_0,t_1]$ into $S^N_1,S^N_2,...,S^N_{\lfloor N^{\beta}/r_N\rfloor}$ where $S^N_i=(t^N_i,t^N_{i+1}]$. 
Let $A_i^N$ be the event that $|X_1^N(t)-X^N_2(t)|\leq r_N$ for some $t\in S^N_i$ but $|X_1^N(t^N_i)-X^N_2(t^N_i)|> r_N$. So, $A_i^N$ can be understood as the event the particles collide during $S^N_i$. We let $B^N_i$ be the following event
\begin{align*}
B^N_i=&\{|V_1^N(t^N_i)-V_2^N(t^N_i)|>N^k\}\cap\{\max\{|V_1^N(t^N_i)|,|V_2^N(t^N_i)|\}<N^{m}\}\\
&\cap\{|X_1^N(t^N_i)-X^N_2(t^N_i)|> r_N\}\\
&\cap\{\exists 0\leq t\leq h_N:|X_1^N(t^N_i)-X^N_2(t^N_i)+t(V_1^N(t^N_i)-V_2^N(t^N_i))|\leq r_N\}.
\end{align*}
So, $B^N_i$ is the event that at the start of $S^N_i$, the particles' speeds are not too fast, their relative speed is not too slow and they would collide if their relative velocity doesn't change during $S^N_i$. The event $B^N_i$ can be determined by $V_i^N(t^N_i)$ and $X_i^N(t^N_i)$ and we want to use $B^N_i$ to approximate $A^N_i$ and estimate the probability of $B^N_i$ happening. Informally, for technical reasons, as the typical speeds of the particles are of order $\sqrt{N}$, we want to ignore the probability that either $\{|V_1^N(t^N_i)-V_2^N(t^N_i)|>N^k\}$ or $\max\{|V_1^N(t^N_i)|,|V_2^N(t^N_i)|\}<N^{m}\}$ happens. Moreover, $S^N_i$ is a small time interval during which the velocities of the particles are unlikely to change much, and thus we want to approximate $A^N_i$ by
\begin{displaymath}
\{|X_1^N(t^N_i)-X^N_2(t^N_i)|> r_N\}
\cap\{\exists 0\leq t\leq h_N:|X_1^N(t^N_i)-X^N_2(t^N_i)+t(V_1^N(t^N_i)-V_2^N(t^N_i))|\leq r_N\}.
\end{displaymath}
We will start by estimating the probability that $B^N_i$ happens. For $v\in\mathbb{R}^d$, let 
\begin{displaymath}
 D^N(v)=\{x\in\mathbb{R}^d:|x|\geq r_N\}\cap\{\exists 0\leq t\leq h_N:|x+tv|\leq r_N\}.
\end{displaymath}
 Note that
\begin{displaymath}\numberthis\label{21}
 Vol(D^N(v))=|v|(r_N)^{d-1}Vol(S_{d-1})h_N,
\end{displaymath}
where $S_{d-1}$ is the $d-1$ dimensional sphere with radius $1$. Also, we have 
\begin{displaymath}
\sup_{|v|<N^m}(\sup_{x\in D^N_v}|x|)\leq CN^{m-\beta}.
\end{displaymath}
Now, let $\bar{V}^N_i=N^{-\frac{1}{2}}V^N_i$, and let $p^N_i$ be the transition density of $(\bar{V}^N_i,X^N_i)$. Then we have $(\bar{V}^N_i,X^N_i)$ is bivariate normally distributed with 
\begin{align*}
 &|Var(\bar{V}^N_i)-\frac{b^2_i}{2\tau_i}|\leq\frac{C}{N},\\
&|Var(X^N_i)-\frac{b_i^2t}{\tau_i^2}|\leq\frac{C}{N},\\
&|Cov(X^N_i,\bar{V}^N_i)|\leq\frac{C}{\sqrt{N}}.
\end{align*}
So, we know the limiting distribution of $(\bar{V}^N_i,X^N_i)$. Let $H^N=\{v,u\in\mathbb{R}^d:N^{k-\frac{1}{2}}<|v-u|;|v|,|u|<N^{m-\frac{1}{2}}\}$, we have
\begin{displaymath}\numberthis\label{22}
 \mathbb{P}(B_i^N)=\int_{H^N}\int_{\mathbb{R}^d}\int_{y-z\in D^N(N^{\frac{1}{2}}(u-v))}p^N_1(0,0,x_1;t^N_i,u,y)p^N_2(0,0,x_2;t^N_i,v,z)dydzdvdu.
\end{displaymath}
Let $f_i$ denote the probability density function of a normal random variable in $\mathbb{R}^d$ with mean zero and variance $\frac{b^2_i}{2\tau_i}$ and let $f$ be the probability density function of a normal random variable with mean zero and variance $\frac{b^2_1}{2\tau_1}+\frac{b^2_2}{2\tau_2}$. Then we have
\begin{align*}
 &(1+CC^N)q_2(0,x_2;t^N_i(1+C^N),z)f_2(\frac{v}{1+C^N})\\
&\geq p^N_2(0,0,x_2;t^N_i,v,z)\\
&\geq (1-CC^N)q_2(0,x_2;t^N_i(1-C^N),z)f_2(\frac{v}{1-C^N}).
\end{align*}
Now, assume without loss of generality that $x_1=0$, then for $\{u,v\}\in H^N$ and $y-z\in D^N(N^{\frac{1}{2}}(u-v))$ we have
\begin{align*}
 &(1+CC^N)q_1(0,0;t^N_i(1+C^N),z\frac{(|z|-C^N)^+}{|z|})f_1(\frac{u}{1+C^N})\\
&\geq p^N_1(0,0,0;t^N_i,u,y)\\
&\geq (1-CC^N)q_1(0,0;t^N_i(1-C^N),z\frac{(|z|+C^N)}{|z|})f_1(\frac{u}{1-C^N}).
\end{align*}
 Combining these inequalities with (\ref{21}) and (\ref{22}), we have
\begin{align*}
 \mathbb{P}(B_i^N)&=\int_{H^N}\int_{\mathbb{R}^d}\int_{y-z\in D^N(N^{\frac{1}{2}}(u-v))}p^N_1(0,0,x_1;t^N_i,u,y)p^N_2(0,0,x_2;t^N_i,v,z)dydzdvdu\\
&\leq \int_{H^N}\int_{\mathbb{R}^d}Vol(D^N(N^{\frac{1}{2}}(u-v)))(1+CC^N)q_1(0,0;t^N_i(1+C^N),z\frac{(|z|-C^N)^+}{|z|})f_1(\frac{u}{1+C^N})\\
&\cdot (1+CC^N)q_2(0,x_2;t^N_i(1+C^N),z)f_2(\frac{v}{1+C^N})dzdvdu\\
&\leq \int_{H^N}\int_{\mathbb{R}^d}|N^{\frac{1}{2}}(u-v)|(r_N)^{d-1}Vol(S_{d-1})h_N(1+C^N)q_1(0,0;t^N_i(1+C^N),z\frac{(|z|-C^N)^+}{|z|})f_1(\frac{u}{1+C^N})\\
&\cdot (1+CC^N)q_2(0,x_2;t^N_i(1+C^N),z)f_2(\frac{v}{1+C^N})dzdvdu\\
&\leq N^{\frac{1}{2}}(r_N)^{d-1}Vol(S_{d-1})h_N(1+CC^N)\int_{H^N}|u-v|f_1(\frac{u}{1+C^N})f_2(\frac{v}{1+C^N})dudv\\
&\cdot\int_{\mathbb{R}^d}q_1(0,0;t^N_i(1+C^N),z\frac{(|z|-C^N)^+}{|z|})
 q_2(0,x_2;t^N_i(1+C^N),z)dz.
\end{align*}
Recall that $h_N=\frac{t_1-t_0}{\lfloor N^\beta/r_N\rfloor}$, we have
\begin{align*}\numberthis\label{23}
 &(r_N)^{-d}N^{-\frac{1}{2}+\beta}\mathbb{P}(B_i^N)\\
&\leq Vol(S_{d-1})(t_1-t_0)\int_{H^N}|u-v|f_1(\frac{u}{1+C^N})f_2(\frac{v}{1+C^N})dvdu\\
&\cdot \int_{\mathbb{R}^d}(1+CC^N)q_2(0,x_2;t^N_i(1+C^N),z)q_1(0,0;t^N_i(1+C^N),z\frac{(|z|-C^N)^+}{|z|})dz.
\end{align*}
Now, note that for $N$ sufficiently large, we have
\begin{displaymath}
 \mathbf{1}_{(u,v)\in H^N}|u-v|f_1(\frac{u}{1+C^N})f_2(\frac{v}{1+C^N})\leq |u-v|f_1(\frac{u}{2})f_2(\frac{v}{2}),
\end{displaymath}
which is integrable over $\mathbb{R}^d\times\mathbb{R}^d$ and also
\begin{displaymath}
 \mathbf{1}_{(u,v)\in H^N}|u-v|f_1(\frac{u}{1+C^N})f_2(\frac{v}{1+C^N})\rightarrow |u-v|f_1(u)f_2(v)
\end{displaymath}
pointwise. Thus, by dominated convergence theorem, we have
\begin{align*}\numberthis\label{24}
 &\int_{H^N}|u-v|f_1(\frac{u}{1+C^N})f_2(\frac{v}{1+C^N})dvdu\\
&\leq \int_{\mathbb{R}^d}\int_{\mathbb{R}^d}|u-v|f_1(u)f_2(v)dvdu+C^N\\
&\leq \int_{\mathbb{R}^d}|v|f(v)dv+C^N.
\end{align*}
Because $f$ is the probability density function of a normal random variable in $\mathbb{R}^d$ with variance $\frac{b_1^2}{2\tau_1}+\frac{b_1^2}{2\tau_2}$, $\int_{\mathbb{R}^d}|v|f(v)dv$ equals to $\sqrt{\frac{b_1^2}{2\tau_1}+\frac{b_1^2}{2\tau_2}}$ multiplied by the expected norm of an $\mathbb{R}^d$ standard normal random vector.
On the other hand, for $t_0\leq t\leq t_1$, we have
\begin{displaymath}
(1+CC^N)q_2(0,x_2;t(1+C^N),z)q_1(0,0;t(1+C^N),z\frac{(|z|-C^N)^+}{|z|})\rightarrow q_2(0,x_2;t,z)q_1(0,0;t,z)
\end{displaymath}
uniformly over $t$ and $z$. Moreover, there exists a constant $c$ such that whenever $|z|>c$, we have
\begin{align*}
&(1+CC^N)q_2(0,x_2;t(1+C^N),z)q_1(0,0;t(1+C^N),z\frac{(|z|-C^N)^+}{|z|})\\
&\leq (1+CC^N)q_2(0,x_2;t_1,z/2)q_1(0,0;t_1,z/2),
\end{align*}
which is integrable over $\mathbb{R}^d$.
Again by dominated convergence theorem, we have
\begin{align*}\numberthis\label{25}
&\int_{|z|>c}(1+CC^N)q_2(0,x_2;t^N_i(1+C^N),z)q_1(0,0;t^N_i(1+C^N),z\frac{(|z|-C^N)^+}{|z|})dz\\
&\leq\int_{|z|>c}q_2(0,x_2;t^N_i,z)q_1(0,0;t^N_i,z)dz+C^N.
\end{align*}
By uniform convergence, we have
\begin{align*}\numberthis\label{26}
 &\int_{|z|\leq c}(1+CC^N)q_2(0,x_2;t^N_i(1+C^N),z)q_1(0,0;t^N_i(1+C^N),z\frac{(|z|-C^N)^+}{|z|})dz\\
&\leq\int_{|z|\leq c}q_2(0,x_2;t^N_i,z)q_1(0,0;t^N_i,z)dz+C^N.
\end{align*}
So, with (\ref{23}), (\ref{24}),(\ref{25}) and (\ref{26}) we can deduce.
\begin{align*}
  &(r_N)^{-d}N^{-\frac{1}{2}+\beta}\mathbb{P}(B_i^N)\\
&\leq Vol(S_{d-1})(t_1-t_0)(\int_{\mathbb{R}^d}|v|f(v)dvdu+C^N)(\int_{\mathbb{R}^d}q_2(0,x_2;t^N_i,z)q_1(0,0;t^N_i,z)dz+C^N)\\
&\leq c_d(t_1-t_0)\sqrt{\frac{b_1^2}{\tau_1}+\frac{b_2^2}{\tau_2}}\int_{\mathbb{R}^d}q_2(0,x_2;t^N_i,z)q_1(0,0;t^N_i,z)dz+C^N.
\end{align*}
Now, for general $x_1$, we would have
\begin{displaymath}\numberthis\label{collisionprob}
 (r_N)^{-d}N^{-\frac{1}{2}+\beta}\mathbb{P}(B_i^N)\leq c_d(t_1-t_0)\sqrt{\frac{b_1^2}{\tau_1}+\frac{b_2^2}{\tau_2}}\int_{\mathbb{R}^d}q_2(0,x_2;t^N_i,z)q_1(0,x_1;t^N_i,z)dz+C^N.
\end{displaymath}
Similarly we can show that
\begin{align*}
(r_N)^{-d}N^{-\frac{1}{2}+\beta}\mathbb{P}(B_i^N)\geq c_d(t_1-t_0)\sqrt{\frac{b_1^2}{\tau_1}+\frac{b_2^2}{\tau_2}}\int_{\mathbb{R}^d}q_2(0,x_2;t^N_i,z)q_1(0,x_1;t^N_i,z)dz-C^N.
\end{align*}

This gives us an estimation on $\mathbb{P}(B_i^N)$. Next, we would like to show that the event $B^N_i$ is almost the same as $A_i^N$. More precisely, we want to show that the probability one of $A_i^N$ and $B_i^N$ happens but the other does not happen is bounded by $C^N(r_N)^{d}N^{-\beta+\frac{1}{2}}$. First, we show that the velocities of particles during $S^N_i$ are unlikely to change much. Note that, for $0<s<h_N$ and $j=1,2$ we have,
\begin{displaymath}
V_j^N(t^N_i+s)=V^N_j(t^N_i)e^{-N\tau_js}+\int_{t^N_i}^{t^N_i+s}e^{-N\tau_j(t^{N}_i+s-s')}Nb_jdB^j_{s'}.
\end{displaymath}
Let 
\begin{displaymath}
U_s=\int_{t^N_i}^{t^N_i+s}e^{-N\tau_j(t^{N}_i+s-s')}Nb_jdB^j_{s'}=e^{-N\tau_js}\int_0^se^{N\tau_js'}Nb_jdB^j_{t^{N}_i+s'},
\end{displaymath}
and $M(s)=\int_0^s[e^{N\tau_js'}Nb_j]^2ds'$, we have, by Dubins-Schwarz theorem, 
\begin{displaymath}
U_s=e^{-N\tau_js}W_{M(s)},
\end{displaymath}
for some Brownian motion $W$. By standard Doob's martingale inequality applied on the exponential of Brownian motion we obtain
\begin{displaymath}
\mathbb{P}(\sup_{0<s<h_N}|W_{M(s)}|\geq Nh_N^{k})\leq 2e^{-\frac{N^2h_N^{2k}}{2M(h_N)}}.
\end{displaymath}
Recall $h_N\leq \frac{Cr_N}{N^\beta}$ with $r_N<rN^{-\alpha}$ for some $\alpha>\frac{1}{2}$ and $\beta=\frac{1}{2}-2\epsilon$ for sufficiently small $\epsilon$. In particular, $h_N\leq \frac{C^N}{N}$ and
\begin{displaymath}
M(h_N)=\int_0^{h_N}[e^{N\tau_js'}Nb_j]^2ds'\leq Ch_NN^2b_j^2\leq Ch_NN^2.
\end{displaymath} 
Recall $k=\frac{1}{2}-\epsilon$, we have $\frac{N^2h_N^{2k}}{2M(h_N)}\geq \frac{h_N^{-2\epsilon}}{C}$ and thus
\begin{displaymath}
2e^{-\frac{N^2h_N^{2k}}{2M(h_N)}}\leq  C^N(r_N)^{d}N^{-\beta+\frac{1}{2}}.
\end{displaymath} 
Therefore, we conclude
\begin{displaymath}\numberthis\label{speed}
 \mathbb{P}(\sup_{0<s<h_N}|V_j^N(t^N_i+s)-V^N_j(t^N_i)e^{-N\tau_js}|\geq Nh_N^{k})\leq C^N(r_N)^{d}N^{-\beta+\frac{1}{2}}.
\end{displaymath}
 As a result,
\begin{displaymath}
\mathbb{P}(\sup_{0<s<h_N}|X_j^N(t^N_i+s)-X^N_j(t^N_i)-V^N_j(t^N_i)\int_0^se^{-N\tau_js'}ds'|\geq Nh_N^{1+k})\leq C^N(r_N)^{d}N^{-\beta+\frac{1}{2}}.
\end{displaymath}
This gives us an approximation of the particles' trajectories during $S^N_i$. So, we can further condition on the event 
\begin{displaymath}
\sup_{0<s<h_N}|X_j^N(t^N_i+s)-X^N_j(t^N_i)-V^N_j(t^N_i)\int_0^se^{-N\tau_js'}ds'|< Nh_N^{1+k}
\end{displaymath}
for $j=1,2$. Note that $\frac{Nh_N^{1+k}}{r_N}\rightarrow 0$ and $e^{-N\tau_jh_N}\rightarrow 1$ as $N\rightarrow\infty$. So, conditioning on the above event, we can estimate $X_j^N(t^N_i+s)$ by $X^N_j(t^N_i)+V^N_j(t^N_i)\int_0^se^{-N\tau_js'}ds'$ and the error will be small compared to $r_N$. We can now consider the following events
\begin{align*}
 F^N_i=&\{|X_1^N(t^N_i)-X^N_2(t^N_i)|> r_N\}\\
&\cap\{\exists 0\leq t\leq h_N:|X_1^N(t^N_i)-X^N_2(t^N_i)+V_1^N(t^N_i)\int_0^te^{-N\tau_1s'}ds'-V_2^N(t^N_i)\int_0^te^{-N\tau_2s'}ds')|\\
&\leq r_N-2Nh_N^{1+k}\}
\end{align*}
and
\begin{align*}
 G^N_i=&\{|X_1^N(t^N_i)-X^N_2(t^N_i)|> r_N\}\\
&\cap\{\exists 0\leq t\leq h_N:|X_1^N(t^N_i)-X^N_2(t^N_i)+V_1^N(t^N_i)\int_0^te^{-N\tau_1s'}ds'-V_2^N(t^N_i)\int_0^te^{-N\tau_2s'}ds')|\\
&\leq r_N+2Nh_N^{1+k}\}.
\end{align*}
Then under the conditions we had, we obtain $F^N_i\subseteq A^N_i\subseteq G^N_i$. Moreover, using the same approximation method we used before, we have that the probability that $B^N_i$ happens but $F^N_i$ does not is bounded above by $C^N(r_N)^{d}N^{-\beta+\frac{1}{2}}$ and also the probability that $G^N_i$ happens but $B^N_i$ does not is bounded by $C^N(r_N)^{d}N^{-\beta+\frac{1}{2}}$. Thus, we have the probability one of $A_i^N$ and $B_i^N$ happens but the other does not happen is bounded by $C^N(r_N)^{d}N^{-\beta+\frac{1}{2}}$. 

Now, we let $T'=\min\{t^N_i:B_i^N happens\}$, then we claim that the probability that $B_i^N$ happens but $T'\neq t^N_i$ is bounded by $C^N(r_N)^{d}N^{-\beta+\frac{1}{2}}$. Let $P^N_{ij}$ be the probability that $B_i^N$ and $B_j^N$ both happen and we want to show that
\begin{displaymath}
\sum_{j<i}P^N_{ij}\leq C^N(r_N)^{d}N^{-\beta+\frac{1}{2}}.
\end{displaymath}
For $j=i-1$, we can use similar argument as above to say that the probability that $B^N_j$ happens and $(V^N_1(t^N_i)-V^N_2(t^N_i))\cdot (X^N_1(t^N_i)-X^N_2(t^N_i))\leq 0$ is bounded above by $C^N(r_N)^{d}N^{-\beta+\frac{1}{2}}$. So, the probability that both $B^N_j$ and $B^N_i$ happens is bounded by $C^N(r_N)^{d}N^{-\beta+\frac{1}{2}}$. Now, we will show that for all $j<i-1$ and  $t^N_i-t^N_j\leq N^\lambda$, $P^N_{ij}\leq C^N(r_N)^{d+1}N^{-2\beta+\frac{1}{2}}$. We condition on $B^N_j$ happens and $\mathcal{F}_{t^N_j}$. It suffices to show that the probability $B^N_i$ happens is bounded by $C^Nh_N$. Our strategy is to show that if $B^N_j$ happened, then during $S^N_i$, the particles are probably moving away from each other, and thus they are unlikely to collide. Note that $\lambda<-1$. Let $s=t^N_i-t^N_j$, $X^N=X_1^N-X^N_2$ and $V^N=V_1^N-V^N_2$. Then $X^N(t^N_{i})$ is normally distributed with mean
\begin{displaymath}
X^N(t^N_{j})+V_1^N(t^N_j)\int_0^se^{-N\tau_1s'}ds'-V_2^N(t^N_j)\int_0^se^{-N\tau_2s'}ds'
\end{displaymath}
 and $\frac{1}{C}N^2s^3\leq Var(X^N(t^N_{i}))\leq CN^2s^3$ and $V^N(t^N_i)$ is normally distributed with mean
\begin{displaymath}
 V_1^N(t^N_j)e^{-N\tau_1s}-V_2^N(t^N_j)e^{-N\tau_2s}
\end{displaymath}
and $\frac{1}{C}N^2s\leq Var(V^N(t^N_i))\leq CN^2s$. Also, their correlation is between $\frac{1}{C}$ and $1-\frac{1}{C}$. Because $B^N_j$ happened, we know that
\begin{displaymath}
|V_1^N(t^N_j)-V^2_N(t^N_j)|>N^k.
\end{displaymath}
and
\begin{displaymath}
\exists 0\leq t\leq h_N: |X^N(t^N_j)+t(V_1^N(t^N_j)-V_2^N(t^N_j))|\leq r_N.
\end{displaymath}
Note that $r_N\leq CN^{-\epsilon}h_NN^k$, we have for sufficiently large $N$,
\begin{displaymath}
|X^N(t^N_j)+s(V_1^N(t^N_j)-V_2^N(t^N_j))|\geq \frac{1}{3}sN^k.
\end{displaymath}
Also, we know that
\begin{displaymath}
\max\{|V_1^N(t^N_j)|,|V^2_N(t^N_j)|\}\leq N^m.
\end{displaymath}
For $s'\leq s$, we have for $i'=1,2$,
\begin{displaymath}
1-CN^{\lambda+1}<e^{-N\tau_{i'}s}<1.
\end{displaymath}
Recall that $k=\frac{1}{2}-\epsilon$, $m=\frac{1}{2}+\epsilon$ for sufficiently small $\epsilon$ and $\lambda<2(k-1)=-1-2\epsilon$. Therefore, we have
\begin{displaymath}
|s(V_1^N(t^N_j)-V_2^N(t^N_j))-\big(V_1^N(t^N_j)\int_0^se^{-N\tau_1s'}ds'-V_2^N(t^N_j)\int_0^se^{-N\tau_2s'}ds'\big)|\leq C^NsN^k.
\end{displaymath}
So, if everything goes according to expectations, at time $t^N_i$, the two particles will be of distance at least $\frac{1}{3}sN^k$ away from each other and they will move further away from each ohter. So, in order to make $B^N_i$ happen, either $X^N(t^N_i)$ needs to deviate sufficiently from its mean or $V^N(t^N_i)$ needs to deviate sufficiently from its mean. More precisely, we need one of the following two events to happen
\begin{equation}\label{con1}
|X^N(t^N_i)-X^N(t^N_j)-V_1^N(t^N_j)\int_0^se^{-N\tau_1s'}ds'+V_2^N(t^N_j)\int_0^se^{-N\tau_2s'}ds'|>\frac{1}{4}sN^k,
\end{equation}
or
\begin{equation}\label{con2}
 |V^N(t^N_i)- V_1^N(t^N_j)e^{-N\tau_1s}+V_2^N(t^N_j)e^{-N\tau_2s}|>\frac{1}{2}N^k.
\end{equation}
First, we condition on the velocity $V^N(t^N_i)$ such that (\ref{con2}) is false. In order to make $B^N_i$ happen, $X^N(t^N_i)$ needs to lie in $D^N(V^N(t^N_i))$. Then the conditional probability density function of $X^N(t^N_i)$ inside $D^N(V^N(t^N_i))$ is bounded above by $N^{-2m}$. Then, by same calculation as earlier, we obtain that the probability $B^N_i$ happens is bounded above by $C^Nh_N$. Now, we condition on that (\ref{con2}) is true. Because the standard deviation of $V^N(t^N_i)$ is at most $CN\sqrt{s}$, which is smaller than $CN^{\frac{\lambda}{2}+1-k}N^{k}$, and $\frac{\lambda}{2}+1-k<0$, the probability (\ref{con2}) happens is at most $C^NN^{-m-d}s^{3d}$ and the conditional probability density function of $X^N(t^N_i)$ is at most $\frac{C}{(N^2s^3)^{\frac{d}{2}}}$. So, by same calculation as earlier again, we can deduce that the probability $B^N_i$ happens is bounded above by $C^Nh_N$. This concludes that $P^N_{ij}\leq C^N(r_N)^{d+1}N^{-2\beta+\frac{1}{2}}$. 

Now, suppose $N^{\lambda}\leq t^N_i-t^N_j\leq N^{-1}$. We again condition on $B^N_j$ happens and $\mathcal{F}_{t^N_j}$. Then, we know that conditional on any $V^{N}(t^N_i)$, $X^N(t^N_i)$ will be normally distributed with $Var(X^N(t^N_i))\geq\frac{N^{2+3\lambda}}{C}$. Therefore, we have that conditional on $B^N_j$ happening, $B^N_i$ happens with at most $N^{m-\beta}(r_N)^{d}N^{-\frac{d}{2}(2+3\lambda)}$ probability and
\begin{align*}
P^N_{ij}&\leq C(r_N)^{d}N^{-\beta+\frac{1}{2}}N^{m-\beta}(r_N)^{d}N^{-\frac{d}{2}(2+3\lambda)}\\
&\leq C[(r_N)^{d+1}N^{-2\beta+\frac{1}{2}}][(r_N)^{d-1}N^{m-\frac{d}{2}(2+3\lambda)}]\\
&\leq C[(r_N)^{d+1}N^{-2\beta+\frac{1}{2}}][N^{m-\alpha(d-1)-\frac{d}{2}(2+3\lambda)}].
\end{align*} 
Also
\begin{displaymath}
m-\alpha(d-1)-\frac{d}{2}(2+3\lambda)\leq m-2\alpha-3-\frac{9\lambda}{2}.
\end{displaymath}
Recall 
\begin{displaymath}
\lambda>\frac{2}{9}(m-2\alpha-4),
\end{displaymath}
we have 
\begin{displaymath}
 m-2\alpha-3-\frac{9\lambda}{2}\leq 1
\end{displaymath}
and $P^N_{ij}\leq C^N(r_N)^{d+1}N^{-2\beta+\frac{3}{2}}$. Finally, for $t^N_i-t^N_j>N^{-1}$ and condition on $B^N_j$ happens and $\mathcal{F}_{t^N_j}$, we know that conditional on any $V^{N}(t^N_i)$, $X^N(t^N_i)$ will be normally distributed with $Var(X^N(t^N_i))\geq\frac{(t^N_i-t^N_j)}{C}$. Therefore, we have
\begin{displaymath}
P^N_{ij}\leq C(r_N)^{d}N^{-\beta+\frac{1}{2}}N^{m-\beta}(r_N)^{d}(t^N_i-t^N_j)^{-\frac{d}{2}},
\end{displaymath}
and thus
\begin{align*}
\sum_{j:t^N_i-t^N_j>N^{-1}} P^N_{ij}&\leq C(r_N)^{d}N^{-\beta+\frac{1}{2}}N^{m-\beta}(r_N)^{d}\int_{N^{-1}}^{\infty}s^{-\frac{d}{2}}ds\\
&\leq C(r_N)^{d}N^{-\beta+\frac{1}{2}}N^{m}(r_N)^{d}N^{\frac{d}{2}-1}\\
&\leq C^N(r_N)^{d}N^{-\beta+\frac{1}{2}},
\end{align*}
where for the last inequality we use the fact that $r_N<N^{-\alpha}$ and $m-1<0$. We also have
\begin{displaymath}
\sum_{j:t^N_i-t^N_j\leq N^{-1}} P^N_{ij}\leq C^NN^{-1+\beta}r_N^{-1}(r_N)^{d+1}N^{-2\beta+\frac{3}{2}}\leq  C^N(r_N)^{d}N^{-\beta+\frac{1}{2}}.
\end{displaymath}
So, we can conclude that the probability that $B_i^N$ happens but $T'\neq t^N_i$ is bounded by $C^N(r_N)^{d}N^{-\beta+\frac{1}{2}}$.

So far, we analyzed the collision events during the time interval $[t_0,t_1]$, and the only place we used the lower bound $t_0$ is to make sure that for $t>t_0$, $(V^N_i,X^N_i)$ are bivariate normally distributed with 
\begin{align*}
 &|Var(V^N_i)-\frac{Nb^2_i}{2\tau_i}|\leq \frac{C}{N},\\
&|Var(X^N_i)-\frac{b_i^2t}{\tau_i^2}|\leq\frac{C}{N},\\
&|Cov(X^N_i,V^N_i)-\frac{b_i^2}{2\tau^2_i}|\leq\frac{C}{N}.
\end{align*}
So, our analysis would still work if we replace $t_0$ by $N^{-(1-\epsilon)}$. By same method as earlier, we could show that the probability that collision happens before $N^{-(1-\epsilon)}$ is at most $C^N(r_N)^{d-1}N^{-\beta+\frac{1}{2}}$. Also, we could show that the probability $T<N^{-(1-\epsilon)}$ and $B^N_i$ happens is bounded by $C^N(r_N)^{d}N^{-\beta+\frac{1}{2}}$. Thus, we can conclude that the probability $T<t_0$ and $B^N_i$ happens is bounded by $C^N(r_N)^{d}N^{-\beta+\frac{1}{2}}$. Therefore, we obtain
\begin{displaymath}
\mathbb{P}(T\in[t_0,t_1]\mbox{ or }T'\in[t_0,t_1]\mbox{ and }|T-T'|>h_N)\leq C^N(r_N)^{d}N^{-\beta+\frac{1}{2}}.
\end{displaymath}
By similar argument as before, we have
\begin{displaymath}
\mathbb{P}(T\in[t_0,t_1]\mbox{ or }T'\in[t_0,t_1]\mbox{ and }|X(T)-X^N_1(T')|>3N^mh_N)\leq C^N(r_N)^{d}N^{-\beta+\frac{1}{2}}.
\end{displaymath}
So, because $g$ is bounded and uniformly continuous, we would have
\begin{displaymath}
|(r_N)^{1-d}N^{-\frac{1}{2}}(\mathbb{E}(g(T,X(T))-g(T',X_1^N(T')))|\rightarrow 0.
\end{displaymath}
We also have
\begin{displaymath}
|(r_N)^{1-d}N^{-\frac{1}{2}}\mathbb{E}[g(T',X_1^N(T'))-\sum_i\mathbf{1}_{B^N_i}g(t^N_i,X_1^N(t^N_i))]|\rightarrow 0.
\end{displaymath}
By similar analyse as in deriving (\ref{collisionprob}), we would get
\begin{align*}
&|\mathbb{E}[\mathbf{1}_{B^N_i}g(t^N_i,X_1^N(t^N_i))]\\
&-c_d(r_N)^{d}N^{-\beta+\frac{1}{2}}\sqrt{\frac{b_1^2}{\tau_1}+\frac{b_2^2}{\tau_2}}r^{d-1}(t_1-t_0)\int_{\mathbb{R}^d}q_1(0,x_1;t^N_i,z)q_2(0,x_1;t^N_i,z)g(t^N_i,z)dz|\\
&\leq C^N(r_N^d)N^{-\beta+\frac{1}{2}}.
\end{align*}
Now, by continuity of $\int_{\mathbb{R}^d}q_1(0,x_1;t^N_i,z)q_2(0,x_1;t^N_i,z)g(t^N_i,z)dz$, we get
\begin{displaymath}
|(r_N)^{1-d}N^{-\frac{1}{2}}\mathbb{E}(g(T,X(T)))-c_d\sqrt{\frac{b_1^2}{\tau_1}+\frac{b_2^2}{\tau_2}}\int_{t_0}^{t_1}\int_{\mathbb{R}^d}q_1(0,x_1;t,z)q_2(0,x_2;t,z)g(t,z)dtdz|\rightarrow 0,
\end{displaymath}
as desired.
\end{proof}
\subsection{Proof of \ref{brown}}
Now, we look at what happens if the radii of the particles converge to zero slowly. More precisely, we will prove Theorem \ref{brown}. The idea is to approximate the free motions of the particles by Brownian motions and show that the deviations of the motions from Brownian are small enough for us to estimate the collision distributions. We start our proof with the following lemma, which allows us to bound the speed of the particles.
\begin{lem}
For all $m>\frac{1}{2}$, $k>0$ and $t_1>0$ we have that there exists a constant $C$ such that, for $j=1,2$,
\begin{displaymath}
\mathbb{P}(\sup_{t<t_1}|V^N_j(t)|>N^m)<CN^{-k},
\end{displaymath} 
for all $N$.
\end{lem}
\begin{proof}
Again, we let $C$ be a constant whose value can change from line to line. We fix $\beta>\frac{1}{2}$ and let $h_N=\frac{t_1}{\lfloor N^{\beta}\rfloor}$. Let $t^N_i=ih_N-h_N$ and we subdivide $(0,t_1]$ into $S^N_1,S^N_2,...,S^N_{\lfloor N^{\beta}\rfloor}$ where $S^N_i=(t^N_i,t^N_{i+1}]$. Choose any $m>m'>\frac{1}{2}$ and $k'>k+\beta>0$. For $j=1,2$, because $V^N_j(t)$ is Gaussian distributed with mean zero and
\begin{displaymath}
 Var(V^N_j(t))\leq \frac{Nb^2_j}{2\tau_j}
\end{displaymath}
for all $t$, we have
\begin{displaymath}
\mathbb{P}(|V^N_j(t^N_i)|>N^{m'})<CN^{-k'}
\end{displaymath}
for any $0<i\leq \lfloor N^{\beta}\rfloor$. Thus, the probability that there is any $t_i^N<t_1$ with $|V^N_j(t^N_i)|>N^{m'}$ is bounded above by $CN^{-k'+\beta}$ for sufficiently large $N$. Also, using the same method as in the derivation of (\ref{speed}), we can show that conditioning on the event  $|V^N_j(t^N_i)|<N^{m'}$,
\begin{displaymath}
 \mathbb{P}(\sup_{0<s<h_N}|V_j^N(t^N_i+s)-V^N_j(t^N_i)e^{-N\tau_is}|\geq N^m-N^{m'})\leq CN^{-k'},
\end{displaymath}
and thus
\begin{displaymath}
 \mathbb{P}(\sup_{0<s<h_N}|V_j^N(t^N_i+s)|\geq N^m)\leq CN^{-k'}.
\end{displaymath}
Summing over $0<i\leq \lfloor N^{\beta}\rfloor$, we can conclude that
\begin{displaymath}
\mathbb{P}(\sup_{t<t_1}|V^N_j(t)|>N^m)<CN^{-k},
\end{displaymath} 
as desired.
\end{proof}

Now, we can start proving Theorem \ref{brown}.
\begin{proof}
Again, let $C$ be a constant whose value can change from line to line and for $i=1,2$, let $W^N_i=X^N_i+\frac{1}{N\tau_i}V_i^N$, then $W^N_i$ is a $d$-dimensional Brownian motion with diffusivity $(\frac{b_i}{\tau_i})^2$. Let $<\frac{1}{2}<m<1-\alpha$ and define the stopping time $T'$ by 
\begin{displaymath}
 T'=\inf\{t\geq 0:|V^N_i(t)|>N^m\}\wedge T.
\end{displaymath}
Then, by the above lemma, we have $\mathbb{P}(T\neq T' \mbox{ and } T'<t_1)<CN^{-2\alpha d}$. We can now repeat the argument used in proving (\ref{eqn1}) to show that
\begin{align*}
 &\int_{t_0}^{t_1}\int_{\mathbb{R}^d}q_1(0,x_1;t,z)q_2(0,x_2;t,z)g(t,z)dtdz\\
&=\int_{t_0}^{t_1}\int_{\mathbb{R}^d}q_1(T',W^N_1(T');t,z)q_2(T',W^N_2(T');t,z)g(t,z)dtdz.
\end{align*}
Let $\epsilon\geq 2r_N$ and note that if $T<T'$, then 
\begin{displaymath}
 r_N-CN^{m-1} \leq\|W^N_1(T')-W^N_2(T')|\leq r_N+CN^{m-1}.
\end{displaymath}
By the same method as in deriving (\ref{stop}), we would have that on the event $\{T=T'\}$ ,
 \begin{align*}
 &||W^N_1(T')-W^N_2(T')|^{2-d}g(T,X(T))\\
&-c_d[(\frac{b_1}{\tau_1})^2+(\frac{b_2}{\tau_2})^2]\int_0^\infty\int_{\mathbb{R}^d} q_1(T',W^N_1(T');s,z)q_2(T',W^N_2(T');s,z)g(s.z)dzds|\\
&\leq C[\epsilon^{2-d}\|g\|+(r_N)^{2-d}(\phi_g(\epsilon)+\epsilon^2)].
\end{align*}
Note that $N^{m-1}<Cr_N$, thus
 \begin{align*}
&|(r_N)^{2-d}g(T,X(T))-c_d[(\frac{b_1}{\tau_1})^2+(\frac{b_2}{\tau_2})^2]\int_0^\infty\int_{\mathbb{R}^d} q_1(T',W^N_1(T');s,z)q_2(T',W^N_2(T');s,z)g(s.z)dzds|\\
&\leq C[(\epsilon^{2-d}+(r_N)^{1-d}N^{m-1})\|g\|+(r_N)^{2-d}(\phi_g(\epsilon)+\epsilon^2)].
\end{align*}
When $T\neq T'$, we simply have
 \begin{align*}
 &|(r_N)^{2-d}g(T,X(T))-c_d[(\frac{b_1}{\tau_1})^2+(\frac{b_2}{\tau_2})^2]\int_0^\infty\int_{\mathbb{R}^d} q_1(T',W^N_1(T');s,z)q_2(T',W^N_2(T');s,z)g(s.z)dzds|\\
&\leq C(r_N)^{2-d}.
\end{align*}
Also, using the same method as at the end of the proof of Theorem \ref{diff}, we could get $\mathbb{P}(T'<t_1)\leq C(r_N)^{d-2}$.
Thus,
\begin{align*}
&\big|(r_N)^{2-d}\mathbb{E}[g(T,X(T))]-c_d[(\frac{b_1}{\tau_1})^2+(\frac{b_2}{\tau_2})^2]\int_{t_0}^{t_1}\int_{\mathbb{R}^d}q_1(0,x_1;t,z)q_2(0,x_2;t,z)g(t,z)dtdz\big|\\
&\leq C[(\epsilon^{2-d}+(r_N)^{1-d}N^{m-1})\|g\|+(r_N)^{2-d}(\phi_g(\epsilon)+\epsilon^2)](r_N)^{d-2}+C\mathbb{P}(T\neq T' \mbox{ and } T'<t_1)(r_N)^{2-d}.
\end{align*}
By letting $\epsilon\rightarrow 0$ as $N\rightarrow\infty$, we have that the right hand side of the above inequality converges to zero as desired.
\end{proof}

\section{Uniqueness proof}
\subsection{Idea of the proof}\label{idea}
As explained in the introduction, we are interested in the well-posedness of Smoluchowski coagulation-diffusion equations. In the past, the most common way for obtaining well-posedness result for unbounded $K$ and $1/a$ was to approximate the PDEs by those with bounded $K$ and $1/a$. However, the allowed growths for $K$ and $1/a$ were not enough for us obtain the well-posedness results for the Ornstein-Uhlenbeck cases. Our approach, however, attempts to linearize the PDEs and exploit their properties to obtain results that work for the Ornstein-Uhlenbeck cases. We first explain a heuristic argument for Theorem \ref{ssthm1}. Suppose we have two solutions, $\mu^1$ and $\mu^2$, and let $\mu=\mu^2-\mu^1$. By Hahn decomposition theorem, for each $t$ and $x$, we can decompose $\mathbb{R}^d$ into a positive set $P(t,x)$ and a negative set $N(t,x)$ such that for all $A\subseteq P(t,x)$, $\mu(A)\geq 0$ and for all $A\subseteq N(t,x)$, $\mu(A)\leq 0$, and this decomposition is essentially unique. Define 
\begin{displaymath}|\mu_t|(x,A)=\int_{A}\mu_t(x,dy)\mathbf{1}_{y\in P}-\int_{A}\mu_t(x,dy)\mathbf{1}_{y\in N}\end{displaymath}
 and consider $\|\langle w,|\mu_t|\rangle\|_1$. Suppose at time $s$ and position $x$, there are more particles of mass $y$ in $\mu^2$ than in $\mu^1$. We look at what further difference would this cause. Since those extra particles can coagulate with particles of mass $y'$, this will decrease $|\mu|$ at position $x$ and mass $y$, and in the worst case increase $|\mu|$ for mass $y'$ and $y+y'$. So, the total rate of increase of $\langle w,|\mu_t|\rangle\ $ at time $s$ and position $x$ due to those extra particles will be at most
\begin{align*}
&|\mu(x,dy)|\int_0^\infty K(y,y')(\mu^1_s(x,dy')+\mu^2_s(x,dy'))[w(y+y')+w(y')-w(y)]\\
&\leq|\mu(x,dy)|\int_0^\infty K(y,y')(\mu_s^1(x,dy')+\mu_s^2(x,dy'))(2w(y'))\\
&\leq |\mu(x,dy)|w(y)\int_0^\infty (\mu^1_s(x,dy')+\mu^2_s(x,dy'))(2w(y')^2)\\
&\leq 2|\mu(x,dy)|w(y)\|\langle w^2,\mu^1_s+\mu^2_s\rangle\|_\infty
\end{align*}
Furthermore, Brownian motions of the particles won't increase $\|\langle w,|\mu_t|\rangle\|_1$. So, we can integrate the above inequality over $s,x$ and $y$ and obtain
\begin{displaymath}
\|\langle w,|\mu_t|\rangle\|_1\leq 2\sup_{s\leq t}(\|\langle w^2,\mu^1_s+\mu^2_s\rangle\|_\infty)\int_0^t\|\langle w,|\mu_s|\rangle\|_1ds.
\end{displaymath}
Then we can use Gronwall's inequality to show that $\mu^1_t=\mu^2_t$ provided $\sup_{s\leq t}(\|\langle w^2,\mu^1_s+\mu^2_s\rangle\|_\infty)<\infty$. So, this argument indicates that $\|\langle w,|\mu_t|\rangle\|_1$ is the right norm to look at.
\subsection{Space homogeneous case}
In \cite{th}, Norris has discussed about the well-posedness of Smoluchowski's coagulation equations in the space homogeneous case. More precisely, the space homogeneous analogue of equation (\ref{eqncd}) is 
\begin{displaymath}
\dot{\mu}(dy)=K^+(\mu_t)(dy)-K^-(\mu_t)(dy),
\end{displaymath}
where 
\begin{displaymath}
K^+(\mu)(A)=\frac{1}{2}\int_0^\infty\int_0^\infty \mathbf{1}_{y+y'\in A}K(y,y')\mu(dy)\mu(dy'),
\end{displaymath}
\begin{displaymath}
K^-(\mu)(A)=\int_{y\in A}\int_0^\infty K(y,y')\mu(dy)\mu(dy').
\end{displaymath}
We will now make an informal discussion about this problem. For the uniqueness part, the argument in $\ref{idea}$ still works if we just ignore the dependence of $\mu$ on $x$ and the  contribution of the Brownian motions. The corresponding result is when we let $\mu_1$ and $\mu_2$ be solutions, we will have $\mu^1_t=\mu^2_t$ provided $\sup_{s\leq t}(\langle w^2,\mu^1_s+\mu^2_s\rangle)<\infty$. 

Therefore, if we can show that whenever $\mu$ is a solution, $\sup_{s\leq t}(\langle w^2,\mu_s\rangle)<\infty$, then we know that there can be at most one solution. Now, we look at the evolution of $\langle w^2,\mu_s\rangle$. When a particle of mass $y$ collides with a particle of mass $y'$, it brings a change of $w^2(y+y')-w^2(y)-w^2(y')$ to $\langle w^2,\mu_s\rangle$. When $w$ is sublinear, we have 
\begin{displaymath}
w^2(y+y')-w^2(y)-w^2(y')\leq 2w(y)w(y').
\end{displaymath}
Integrating over all possible collisions and over time, we have
\begin{align*}
\langle w^2,\mu_t\rangle&\leq \int_0^t\int_0^\infty\int_0^\infty w(y)w(y')K(y,y')\mu_s(y)\mu_s(y')dydy'ds+\langle w^2,\mu_0\rangle\\
&\leq \int_0^t\int_0^\infty\int_0^\infty w^2(y)w^2(y')\mu_s(y)\mu_s(y')dydy'ds+\langle w^2,\mu_0\rangle\\
&\leq \int_0^t\langle w^2,\mu_s\rangle^2ds+\langle w^2,\mu_0\rangle\\
&\leq \frac{1}{\frac{1}{\langle w^2,\mu_0\rangle}-t}.
\end{align*}
So, we conclude that there can be at most one local solution in the time interval $[0,\frac{1}{\langle w^2,\mu_0\rangle}]$. Note that, in the case when $w(y)=y$ and $K(y,y')=yy'$, all the above inequalities become equalities, and we know thus $\langle w^2,\mu_t\rangle$ will blow up after $\frac{1}{\langle w^2,\mu_0\rangle}$. 

Now, if we assume $K(y,y')\leq w(y)v(y')+w(y')v(y)$ with $w(y)v(y)<y$, then we have
\begin{align*}
\langle w^2,\mu_t\rangle&\leq \int_0^t\int_0^\infty\int_0^\infty w(y)w(y')K(y,y')\mu_s(y)\mu_s(y')dydy'ds+\langle w^2,\mu_0\rangle\\
&\leq \int_0^t\int_0^\infty\int_0^\infty w^2(y)w(y')v(y')\mu_s(y)\mu_s(y')dydy'ds\\+&\int_0^t\int_0^\infty\int_0^\infty w^2(y')w(y)v(y)\mu_s(y)\mu_s(y')dydy'ds+\langle w^2,\mu_0\rangle\\
&\leq 2 \int_0^t\langle w^2,\mu_s\rangle\langle wv,\mu_s\rangle ds+\langle w^2,\mu_0\rangle\\
&\leq  2\int_0^t\langle w^2,\mu_s\rangle\langle y,\mu_s\rangle ds+\langle w^2,\mu_0\rangle.
\end{align*}
Recall that we defined solutions to satisfy
\begin{displaymath}
\sup_{s<T}\langle y,\mu_s\rangle<\infty,
\end{displaymath}
we can thus use Gronwall's inequality to show $\langle w^2,\mu_t\rangle<\infty$ provided $\langle w^2,\mu_0\rangle<\infty$. Since this works for all $T$, we  conclude that there can be at most one global solution.
\subsection{Proof of Theorem \ref{ssthm1}}
Now, we will prove Theorem \ref{ssthm1} rigorously. In order to make sense of $(\ref{eqn})$, we need to first show that both $K^+$ and $K^-$ are kernels. Let $\nu(x,.)$ denote the product measure of $\mu(x,.)$ and $\mu(x,.)$. Then as $K:(0,\infty)\times (0,\infty)\rightarrow (0,\infty)$ is a measurable function, $K\nu(x,.)$ is also a measure. As $f:(y,y')\rightarrow y+y'$ is a measurable function from $(0,\infty)\times (0,\infty)$ to $(0,\infty)$, $K^+(\mu)(x,.)$ is the image measure of $K\nu(x,.)$ induced by $f$. Also, $g:(y,y')\rightarrow y'$ is a measurable function, and thus $K^-(\mu)(x,.)$ is the image measure of $K\nu(x,.)$ induced by $g$.

Now, we assume that $\mu^1$ and $\mu^2$ are solutions and for $i=1,2$, $\sup_{t\leq T}\|\langle w^2,\mu^i_t\rangle\|_\infty<\infty$. We formulate a differential equation describing the behaviour of individual particles in the solutions. For kernels $\nu$ and $\mu$ and any measurable set $A$, let
\begin{displaymath}
K^{\nu+}(\mu)(x,A)=\int_0^\infty\int_0^\infty \mathbf{1}_{y+y'\in A}\frac{y'}{y+y'}K(y,y')\nu(x,dy)\mu(x,dy'),
\end{displaymath}
\begin{displaymath}
K^{\nu-}(\mu)(x,A)=\int_{A}\mu(x,dy')\int_0^\infty K(y,y')\nu(x,dy),
\end{displaymath}
and $K^{\nu}(\mu)=K^{\nu+}(\mu)-K^{\nu-}(\mu)$. By similar analysis as in Section \ref{s2.2}, we have that $K^{\nu\pm}(\mu)$ are also kernels. Denote $K^{i\pm}=K^{\mu^i\pm}$ and $K^{i}=K^{\mu^i}$. Consider the linear evolution equation
\begin{equation}\label{sseqn1}
q^i_t+\int_0^t P_{t-s}K^{i-}_s(q^i_s)ds=P_tq_0+\int_0^t P_{t-s}K^{i+}_s(q^i_s)ds.
\end{equation}
Let $\mathcal{M}'$ be the set of $q$ which can be written as $q^+-q^-$ with $q^+,q^-\in\mathcal{M}$. We say $q\in\mathcal{M}'$ is a solution to (\ref{sseqn1}) up to time $T$ if $q$ satisfies (\ref{sseqn1}) for $t\leq T$ and
\begin{displaymath}
\sup_{t\leq T}\|\langle y,|q_t|\rangle\|_{1}<\infty.
\end{displaymath}
Also, let $S$ be the set of $\nu:\mathbb{R}^d\times\mathcal{B}(0,\infty)\rightarrow[-\infty,\infty]$ such that $\|\langle y,|\nu|\rangle\|_{1}<\infty$.
\begin{prop} 
If we start at $q_0=\mu_0$, then $q^i_t=\mu^i_t$ is a solution of (\ref{sseqn1}).
\end{prop}
\begin{proof}
Note that
\begin{displaymath}
K^{i-}_s(\mu^i_s)=K^{-}_s(\mu^i_s),
\end{displaymath}
and
\begin{align*}
K^{i+}_s(\mu^i_s)(x,A)&=\int_0^\infty\int_0^\infty \mathbf{1}_{y+y'\in A}\frac{y'}{y+y'}K(y,y')\mu^i(x,dy)\mu^i(x,dy')\\
&=\frac{1}{2}[\int_0^\infty\int_0^\infty \mathbf{1}_{y+y'\in A}\frac{y'}{y+y'}K(y,y')\mu^i(x,dy)\mu^i(x,dy')\\
&+\int_0^\infty\int_0^\infty \mathbf{1}_{y+y'\in A}\frac{y}{y+y'}K(y,y')\mu^i(x,dy)\mu^i(x,dy')]\\
&=K^{+}_s(\mu^i_s)(x,A).
\end{align*}
Plugging these into (\ref{eqn}), we have 
\begin{displaymath}
\mu^i_t+\int_0^t P_{t-s}K^{i-}_s(\mu^i_s)ds=P_t\mu_0+\int_0^t P_{t-s}K^{i+}_s(\mu^i_s)ds,
\end{displaymath}
and thus $\mu^i_t$ is a solution of (\ref{sseqn1}).
\end{proof}
We now look at the heuristic meaning of the above equation. Suppose a particle with initial distribution $yq_0$ and makes Brownian motion and coagulates with other particles distributed according to $\mu^i$, then at time $t$, its distribution is $yq^i_t$.

\begin{prop}\label{ssprop32}
Assume $\|\langle y,|q_0|\rangle\|_{1}<\infty$, then equation (\ref{sseqn1}) has at a unique solution in $\mathcal{M}'$. Moreover, if $q_0$ is non-negative, then $q_t$ is also non-negative.
\end{prop} 
\begin{proof}
Let $c^i_s(x,y)=\int_0^\infty K(y,y')\mu^i_s(x,dy')$ and (\ref{sseqn1}) becomes
\begin{displaymath}
q^i_t+\int_0^tP_{t-s}c_sq^i_s ds=P_tq_0+\int_0^t P_{t-s}K^{i+}_s(q^i_s)ds.
\end{displaymath}
Suppose $q_0\geq 0$ and consider first the equation
\begin{equation}\label{sseqn3}
\lambda^i_t+\int_0^tP_{t-s}c_s\lambda^i_s ds=P_tq_0.
\end{equation}
Let $B^{x,x',a,t}$ be the conditional Brownian motion with diffusivity $a$ being at $x$ at time $0$ and at $x'$ at $t$. Then by Feynman-Kac formula, we have
\begin{displaymath}
\lambda^i_t(x,dy)=\int_{\mathbb{R}^d}q_0(z,dy)\mathbb{E}[exp(-\int_0^tc^i_s(B^{z,x,a(y),t}_s,y)ds)]p^{t,z,x}(y)dz
\end{displaymath}
is a solution of this equation and in particular, $\lambda^i_t$ is non-negative. Then, we want to show that this is the unique solution of (\ref{sseqn3}). By linearity, it suffices to show that zero solution is the unique solution of
\begin{displaymath}
\lambda^i_t+\int_0^tP_{t-s}c_s\lambda^i_s ds=0.
\end{displaymath}
For $z>0$, we have
\begin{align*}
\|\langle \mathbf{1}_{y\leq z},|\lambda^i_t|\rangle\|_1&\leq\|\langle \mathbf{1}_{y\leq z},\int_0^tP_{t-s}c_s|\lambda^i_s| ds\rangle\|_1\\
&\leq\|\langle 1,\mathbf{1}_{y\leq z}\int_0^tc_s|\lambda^i_s| ds\rangle\|_1.\\
\end{align*}
Recall that $\sup_{t\leq T}\|\langle w^2,\mu^i_t \rangle\|_\infty<\infty$, we have $c_s\mathbf{1}_{y\leq z}$ is bounded. Thus, we can apply Gronwall's inequality to show $\|\langle \mathbf{1}_{y\leq z},|\lambda^i_t|\rangle\|_1=0$ and as this works for all $z>0$, we have $\lambda^i_t=0$.

Now, we want to show that
\begin{displaymath}
q^i_t+\int_0^tP_{t-s}c_sq^i_s ds=P_tq_0+\int_0^t P_{t-s}K^{i+}_s(q^i_s)ds
\end{displaymath}
has a unique solution. Again, we assume without loss of generality that $q_0=0$. Let $\lambda^i$ be the unique solution of (\ref{sseqn3}). Note that, $\mathbf{1}_{y<2\delta}K^{i+}_s(q^i_s)=0$, we have $\mathbf{1}_{y<2\delta}q^i_s=\mathbf{1}_{y<2\delta}\lambda^i_s=0$. Then, we have $\mathbf{1}_{2\delta\leq y<3\delta}K^{i+}_s(q^i_s)=0$ and thus $\mathbf{1}_{2\delta\leq y<3\delta}q^i_s=0$, and we can keep going. This proves the uniqueness of $q^i_t$. 

Now, as we showed equation (\ref{sseqn3}) has a unique solution, similarly, we can show that conditional on $\lambda_s$,
\begin{equation}\label{sseqn31}
\lambda^i_{t}+\int_s^{t}P_{t-t'}c_{t'}\lambda^i_{t'}dt'=P_{t-s}\lambda_s
\end{equation}
has a unique solution too. Thus for $t\geq s$, we define $f_{s,t}:S\rightarrow S$ such that if $(\lambda_{t'})_{s\leq t'\leq t}$ solves (\ref{sseqn31}) then $f_{s,t}(\lambda_s)=\lambda_t$. Now, we will show a version of variation of constant formula,
\begin{equation}\label{sseqn32}
q^i_t=\lambda^i_t+\int_0^tf_{s,t}(K^{i+}_s(q^i_s))ds.
\end{equation}
Suppose $q^i_t$ is a solution of (\ref{sseqn32}), we want to show that then $q^i_t$ is indeed a solution of (\ref{sseqn1}). By subtracting (\ref{sseqn3}) from (\ref{sseqn1}), it suffices to show 
\begin{displaymath}
q^{i}_t-\lambda^i_t+\int_0^tP_{t-s}K^{i-}_s(q^i_s-\lambda^i_s)ds=\int_0^t P_{t-s}K^{i+}_s(q^i_s)ds.
\end{displaymath}
Plugging (\ref{sseqn32}) into the left hand side of this equation, we have
\begin{align*}
&q^{i}_t-\lambda^i_t+\int_0^tP_{t-s}K^{i-}_s(q^i_s-\lambda^i_s)ds\\&=\int_0^tf_{s,t}(K^{i+}_s(q^i_s))ds
+\int_0^tP_{t-s}K^{i-}_s(\int_0^sf_{s',s}(K^{i+}_{s'}(q^i_{s'}))ds')ds\\
&=\int_0^tf_{s,t}(K^{i+}_s(q^i_s))ds
+\int_0^t\int_s^tP_{t-s'}K^{i-}_{s'}(f_{s,s'}(K^{i+}_{s}(q^i_{s}))ds')ds\\
&=\int_0^t[f_{s,t}(K^{i+}_s(q^i_s))+\int_s^tP_{t-s'}K^{i-}_{s'}(f_{s,s'}(K^{i+}_{s}(q^i_{s}))ds')]ds\\
&=\int_0^t P_{t-s}K^{i+}_s(q^i_s)ds,
\end{align*}
where we used (\ref{sseqn31}) for the last step. This shows that any solution of (\ref{sseqn32}) is also a solution of (\ref{sseqn1}). Now, we construct a solution to (\ref{sseqn32}). For natural number $n$, We note that $\mathbf{1}_{(n)\delta\leq y<(n+1)\delta}K^{i+}_s(q^i_s)$ depends only on $\mathbf{1}_{\delta\leq y<(n)\delta}(q^i_s)$. So, we can inductively give a solution of (\ref{sseqn32}) by letting 
\begin{displaymath}
\mathbf{1}_{y<2\delta}q^i_s=\mathbf{1}_{y<2\delta}\lambda^i_s
\end{displaymath}
and for $n>1$
\begin{displaymath}
\mathbf{1}_{n\delta\leq y<(n+1)\delta}q^i_s=\mathbf{1}_{n\delta\leq y<(n+1)\delta}[\lambda^i_s+\int_0^tf_{s,t}(K^{i+}_s(\mathbf{1}_{ y<n\delta}q^i_s))ds].
\end{displaymath}
Moreover, if $q_0$ is non-negative, then $\lambda^i$ is non-negative too, and using the above construction, we can see by induction that $q_t$ is non-negative too.
\end{proof}
\begin{prop}
For a solution $q^i$ of (\ref{sseqn1}), we have, for $s\leq t$,
\begin{displaymath}
\|\langle w,|q_t^i|\rangle\|_1\leq\|\langle w,|q_s^i|\rangle\|_1
\end{displaymath}
\end{prop}
\begin{proof}
Assume first $q_0\geq 0$, then for any $z>0$, we have
\begin{displaymath}
\sup_{s\leq t}\|\langle \mathbf{1}_{y< z}w,K^{i-}_s(q^i_s)\rangle\|_1\leq \sup_{s\leq t}[\|\langle \mathbf{1}_{y< z}w^2,q^i_s\rangle\|_1\|\langle w,\mu^i_s\rangle\|_\infty]<\infty.
\end{displaymath}
So, we have
\begin{displaymath}
\|\langle \mathbf{1}_{y\leq z} w,q^i_t\rangle\|_1=\|\langle \mathbf{1}_{y\leq z}w,q^i_0\rangle\|_1+\int_0^t \|\langle \mathbf{1}_{y\leq z}w,K^{i+}_s(q^i_s)- K^{i-}_s(q^i_s)\rangle\|_1 ds,
\end{displaymath}
where
\begin{align*}
&\langle \mathbf{1}_{y\leq z}w,K^{i+}_s(q^i_s)- K^{i-}_s(q^i_s)\rangle&\\
&\leq\int_0^{z} w(y)\int_0^y\frac{y'}{y}K(y',y-y')q^i_s(dy')\mu^i_s(d(y-y'))- w(y)q^i_s(dy)\int_0^\infty K(y,y')\mu^i_s(dy')\\
&\leq\int_0^{z}\int_0^\infty K(y,y')\mu^i_s(dy')q^i_s(dy)[w(y+y')\frac{y}{y+y'}-w(y)]\\
&\leq 0.
\end{align*}
Therefore,
\begin{displaymath}
\|\langle \mathbf{1}_{y\leq z} w,q^i_t\rangle\|_1\leq\|\langle \mathbf{1}_{y\leq z}w,q^i_0\rangle\|_1,
\end{displaymath}
for all $z\geq 1$. Let $z\rightarrow\infty$, we conclude that
\begin{displaymath}
\|\langle w,q^i_t\rangle\|_1\leq\|\langle w,q^i_0\rangle\|_1
\end{displaymath}
when $q^i_0\geq 0$. By linearity, we can extend this to
\begin{displaymath}
\|\langle w,|q^i_t|\rangle \|_1\leq \|\langle w,|q^i_0|\rangle \|_1,
\end{displaymath}
without the condition $q^i_0\geq 0$. Similarly, we also have
\begin{displaymath}
\|\langle w,|q^i_t|\rangle \|_1\leq \|\langle w,|q^i_s|\rangle \|_1,
\end{displaymath}
whenever $s\leq t$. 
\end{proof}
Now, for $t\geq s$ let $\Phi^i_{s,t}:S\rightarrow S$ be the map such that if $(q^i_{t'})_{s\leq t'\leq t}$ solves 
\begin{displaymath}
q^i_{t'}+\int_s^{t'} P_{t'-t''}K^{i-}_{t''}(q^i_{t''})dt''=P_{t'-s}q^i_s+\int_s^{t'} P_{t'-t''}K^{i+}_{t''}(q^i_{t''})dt'',
\end{displaymath}
then $\Phi^i_{s,t}(q^i_s)=q_t^i$. We want to verify that we can apply the variation of constant formula in the following way. \begin{prop}
\begin{displaymath}
\mu^2_t-\mu^1_t=\int_0^t\Phi^1_{s,t}[(K^2_s-K^1_s)(\mu^2_s)]ds.
\end{displaymath}
\end{prop}
\begin{proof}
We will start with showing that there exists $\nu\in\mathcal{M}'$ such that
\begin{displaymath}
\nu_t-\mu^1_t=\int_0^t\Phi^1_{s,t}[(K^2_s-K^1_s)(\nu_s)]ds.
\end{displaymath}
For any $z>0$, we have from earlier result that
\begin{displaymath}
\|\langle \mathbf{1}_{y\leq z}w,|\Phi^1_{s,t}[(K^2_s-K^1_s)(\nu_s)]|\rangle \|_1\leq \|\langle \mathbf{1}_{y\leq z}w,|(K^2_s-K^1_s)(\nu_s)|\rangle \|_1.
\end{displaymath}
For $i=1,2$, we have
\begin{align*}
\|\langle \mathbf{1}_{y\leq z}w,|K^i_s(\nu_s)|\rangle \|_1&\leq\|\int_0^z\int_0^\infty K(y,y')\mu^i_s(dy')(|\nu_s(dy)w(y)|+|\nu_s(dy)w(y+y')|)\|_1\\
&\leq \|\int_0^z\int_0^\infty w(y')w(y)\mu^i_s(dy')(|\nu_s(dy)w(y)|+|\nu_s(dy)w(y+y')|)\|_1\\
&\leq 2\|\langle w,\mu^i_s\rangle\|_\infty w(2z)\|\langle \mathbf{1}_{y\leq z}w,|(\nu_s)|\rangle \|_1.
\end{align*}
So, we conclude that
\begin{displaymath}
\|\langle \mathbf{1}_{y\leq z}w,|\Phi^1_{s,t}[(K^2_s-K^1_s)(\nu_s)]|\rangle \|_1\leq C\|\langle \mathbf{1}_{y\leq z}w,|(\nu_s)|\rangle \|_1
\end{displaymath}
for some constant $C$. As this works for all $z>0$, we can use iteration scheme to show the existence and uniqueness of $\nu$. Therefore, we conclude $\nu=\mu^2$. 
\end{proof}
Now, we have enough tools to prove Theorem \ref{ssthm1}.
\begin{proof}
We have
\begin{align*}
\|\langle w,|\mu^2_t-\mu^1_t|\rangle\|_1&=\|\langle w,|\int_0^t\Phi^1_{s,t}[(K^2_s-K^1_s)(\mu^2_s)]ds|\rangle\|_1\\&
\leq \|\langle w,\int_0^t|\Phi^1_{s,t}[(K^2_s-K^1_s)(\mu^2_s)]|ds\rangle\|_1\\
&\leq\|\langle w,\int_0^t|(K^2_s-K^1_s)(\mu^2_s)|ds\rangle\|_1.
\end{align*}
Now, we also have
\begin{align*}
&\|\langle w,|(K^2_s-K^1_s)(\mu^2_s)|ds\rangle\|_1\\&\leq \|\int_0^\infty\int_0^\infty|(\mu^2_s-\mu^1_s)(dy)|\mu^2_s(dy')K(y,y')(w(y')+\frac{y'}{y+y'}w(y+y'))\|_1\\
&\leq 2\|\int_0^\infty\int_0^\infty|(\mu^2_s-\mu^1_s)(dy)|\mu^2_s(dy')K(y,y')w(y')\|_1\\
&\leq 2\|\int_0^\infty\int_0^\infty|(\mu^2_s-\mu^1_s)(dy)|\mu^2_s(dy')w(y)w(y')^2\|_1\\
&\leq 2\|\langle w,|\mu^2_s-\mu^1_s|\rangle\|_1\|\langle w^2,\mu^2_s\rangle\|_\infty.
\end{align*}
Because we assumed $\|\langle w^2,\mu^2_s\rangle\|_\infty<C$ for some constant C, we have
\begin{displaymath}
\|\langle w,|\mu^2_t-\mu^1_t|\rangle\|_1\leq 2C\int_0^t\|\langle w,|\mu^2_s-\mu^1_s|\rangle\|_1ds.
\end{displaymath}
Also by definition of solutions, we know that $\|\langle w,|\mu^2_t-\mu^1_t|\rangle\|_1<\infty$. So, we can apply Gronwall's inequality to obtain $\mu^1_t=\mu^2_t$ almost surely. This concludes the proof of Theorem \ref{ssthm1}.
\end{proof}
\subsection{Uniqueness part of Theorem \ref{ssthm2}}
Now, we will show the uniqueness part of Theorem \ref{ssthm2}. 
\begin{proof}
By Theorem \ref{ssthm1}, it suffices to show that for some $T>0$, if $\mu$ is a solution to (\ref{eqn}), then $\sup_{t\leq T}\|\langle w^2,\mu_t\rangle\|_\infty<\infty$. For this, we will use a similar approach as in Section 5 of \cite{BM}. For any $z>0$, apply $P_s$ to equation (\ref{eqn}), multiply by $\mathbf{1}_{y\leq z}w^2$ and integrate over $\mathbb(0,\infty)$ to obtain, for all $s,t\geq 0$,
\begin{displaymath}
\langle \mathbf{1}_{y\leq z}w^2,P_s\mu_t\rangle\leq \langle \mathbf{1}_{y\leq z}w^2,P_{s+t}\mu_0\rangle+\int_0^t \langle \mathbf{1}_{y\leq z}w^2,P_{s+t-r}K(\mu_r)\rangle dr.
\end{displaymath}

Summing up the inequalities
\begin{displaymath}
\frac{y}{y+y'}w^2(y+y')p(y+y')-w^2(y)p(y)\leq C[w(y)w(y')p(y)+w(y)w(y')p(y')]
\end{displaymath}
and
\begin{displaymath}
\frac{y'}{y+y'}w^2(y+y')p(y+y')-w^2(y')p(y')\leq C[w(y)w(y')p(y)+w(y)w(y')p(y')],
\end{displaymath}
we know that
\begin{displaymath}
w^2(y+y')p(y+y')-w^2(y)p(y)-w^2(y')p(y')\leq 2C[w(y)w(y')p(y)+w(y)w(y')p(y')].
\end{displaymath}
We obtain
\begin{align*}
&\langle \mathbf{1}_{y\leq z}w^2,P_{s+t-r}K(\mu_r)\rangle\\
&\leq 2C\int_{\mathbb{R}^d}\int_0^z\int_0^z w(y)w(y')p^{s+t-r,x,x'}(y')K(y,y')\mu_r(x',dy)\mu_r(x',dy')dx'\\
&\leq 2C\int_{\mathbb{R}^d}\int_0^z\int_0^z w(y)^2w(y')^2p^{s+t-r,x,x'}(y')\mu_r(x',dy)\mu_r(x',dy')dx'\\
&\leq 2C\|\langle \mathbf{1}_{y\leq z}w^2,\mu_r\rangle\|_\infty\langle \mathbf{1}_{y\leq z}w^2,P_{s+t-r}\mu_r\rangle(x).
\end{align*}
Now, set $h(t)=\sup_{s\geq 0}\|\langle\mathbf{1}_{y\leq z}w^2,P_s\mu_t\rangle\|_\infty$. We then obtain
\begin{displaymath}
h(t)\leq \sup_{s\geq0}\|\langle\mathbf{1}_{y\leq z}w^2,P_{s+t}\mu_0\rangle\|_\infty+2C\int_{0}^t h(s)^2ds,
\end{displaymath}
and this implies
\begin{displaymath}
h(t)\leq [\sup_{s\geq0}\|\langle \mathbf{1}_{y\leq z}w^2,P_{s+t}\mu_0\rangle\|_\infty-2Ct]^{-1}.
\end{displaymath}
As this is true for all $z$, we can set $T=\frac{1}{2C}\sup_{s\geq 0}\|\langle w^2,P_{s+t}\mu_0\rangle\|_\infty$ and conclude $\sup_{t\leq T}\|\langle w^2,\mu_t\rangle\|_\infty<\infty$ as desired.
\end{proof}
\subsection{Uniqueness part of Theorem \ref{ssthm3}}
Now, we will use the same strategy to prove the uniqueness part of Theorem \ref{ssthm3}. 
\begin{proof}
In case (a), similar to earlier, we would have
\begin{align*}
&\langle \mathbf{1}_{y\leq z}w^2,P_{s+t-r}K(\mu_r)\rangle\\
&\leq 2C\int_{\mathbb{R}^d}\int_0^z\int_0^z w(y)w(y')p^{s+t-r,x,x'}(y')K(y,y')\mu_r(x',dy)\mu_r(x',dy')dx'\\
&\leq 2C\int_{\mathbb{R}^d}\int_0^z\int_0^z (w(y)^2w(y')v(y')+w(y)v(v)w^2(y'))p^{s+t-r,x,x'}(y')\mu_r(x',dy)\mu_r(x',dy')dx'\\
&\leq C[\|\langle \mathbf{1}_{y\leq z}w^2,\mu_r\rangle\|_\infty\langle \mathbf{1}_{y\leq z}wv,P_{s+t-r}\mu_r\rangle(x)+\|\langle\mathbf{1}_{y\leq z}wv,\mu_r\rangle\|_\infty\langle \mathbf{1}_{y\leq z}w^2,P_{s+t-r}\mu_r\rangle(x)].
\end{align*} 
Note that,
\begin{displaymath}
\langle wv,P_s\mu_t\rangle\leq \langle wv,P_{s+t}\mu_0\rangle+\int_0^t \langle wv,P_{s+t-r}K(\mu_r)\rangle dr,
\end{displaymath}
and since $wvp$ is sublinear, we will have
\begin{displaymath}
\langle wv,P_{s+t-rK(\mu_r)}\rangle\leq 0,
\end{displaymath}
and thus
\begin{displaymath}
\langle wv,P_s\mu_t\rangle\leq \langle wv,P_{s+t}\mu_0\rangle<c,
\end{displaymath}
for some constant c. So, we obtain
\begin{displaymath}
h(t)\leq \sup_{s\geq0}\|\langle\mathbf{1}_{y\leq z}w^2,P_{s+t}\mu_0\rangle\|_\infty+2cC\int_0^th(s)ds,
\end{displaymath}
and we can apply Gronwall to conclude that $\sup_{t\leq T}\|\langle w^2,\mu_t\rangle\|_\infty<\infty$ for any $T>0$.

For case $(b)$, we set $h(t)=\sup_{s\geq 0}(1+s+t)^{1+\epsilon}\|\langle\mathbf{1}_{y\leq z}w^2,P_s\mu_t\rangle\|_\infty$.
Then, by similar computations as earlier, we will have
\begin{align*}
\langle \mathbf{1}_{y\leq z}w^2,P_s\mu_t\rangle &\leq \langle \mathbf{1}_{y\leq z}w^2,P_{s+t}\mu_0\rangle+\int_0^t \langle \mathbf{1}_{y\leq z}w^2,P_{s+t-r}K(\mu_r)\rangle dr\\
&\leq \langle \mathbf{1}_{y\leq z}w^2,P_{s+t}\mu_0\rangle+2C\int_0^t\|\langle \mathbf{1}_{y\leq z}w^2,\mu_r\rangle\|_\infty\langle \mathbf{1}_{y\leq z}w^2,P_{s+t-r}\mu_r\rangle.
\end{align*}
So, we obtain
\begin{displaymath}
h(t)\leq c+C\int_{0}^t \frac{h(s)^2}{(1+s)^{1+\epsilon}}ds.
\end{displaymath}
If, for example, $c$ is small enough such that
\begin{displaymath}
4c^2C\int_{0}^\infty\frac{1}{(1+s)^{1+\epsilon}}ds<c,
\end{displaymath}
then we have $h(t)<2c$ for all $t$ and thus we have uniqueness of the global solution.
\end{proof}
\section{Existence}
In this section we will prove Theorem \ref{ssthm2} and Theorem \ref{ssthm3}.
We consider the following linear PDE
\begin{equation}\label{sseqn5}
q_t=P_tq_0+\int_0^t P_{t-s}K^\nu_s(q_s)ds
\end{equation}
for $t\leq T$ with $\nu_0=\mu_0=q_0$ and $\nu_s$ non-negative satisfying $\sup_{s\leq T}\|\langle w^2,\nu_s\rangle\|_\infty\leq c$. Proposition \ref{ssprop32} tells us the existence, uniqueness and non-negativity of $q$.

Now, let $G$ be the set of $\tau\in\mathcal{M}$ such that $\sup_{s\leq T}\|\langle y,\tau_s\rangle\|_1<\infty$ and $H$ be the set of $\tau\in G$ such that $\sup_{s\leq T}\|\langle w^2,\tau_s\rangle\|_\infty<\infty$. We can then define function $f:H\rightarrow G$ so that for any $\nu\in H$, $f(\nu)=q$, where $q$ is the solution of ($\ref{sseqn5}$). We aim to construct solutions using iteration scheme with $f$. We will now give a bound on $\|\langle w^2,q_t\rangle\|_\infty$. 
\begin{prop}\label{propbrownian}
Assume
\begin{displaymath}
\frac{y}{y+y'}w^2(y+y')p(y+y')-w^2(y)p(y)\leq C[w(y)w(y')p(y)+w(y)w(y')p(y')].
\end{displaymath}
then we can find $c>0$ and $T>0$ so that 
\begin{displaymath}\sup_{t\leq T}\sup_{s\geq 0}\|\langle w^2,P_sf(\nu)_t\rangle\|_\infty\leq c
\end{displaymath} 
if \begin{displaymath}\sup_{t\leq T}\sup_{s\geq 0}\|\langle w^2,P_s\nu_t\rangle\|_\infty\leq c.\end{displaymath}
\end{prop}
\begin{proof}
Again, we can copy the argument in \cite{BM}. For any $z>0$, Let
\begin{displaymath}
h(t)=\sup_{s\geq 0}\|\langle \mathbf{1}_{y<z}w^2,P_sq_t\rangle\|_\infty.
\end{displaymath}
Then we have
\begin{displaymath}
\langle \mathbf{1}_{y<z}w^2,P_sq_t\rangle\leq
\langle \mathbf{1}_{y<z}w^2,P_{s+t}(q_0)\rangle+\int_0^t\langle w^2,P_{s+t-r}K^\nu_r(q_r)\rangle dr
\end{displaymath}
and
\begin{align*}
&\langle \mathbf{1}_{y<z}w^2,P_{s+t-r}K^\nu_r(q_r)\rangle\\
&\leq\int_{\mathbb{R}^d} \int_0^\infty\int_0^z K(y,y')q_r(dy)\nu_r(dy')(\frac{yw^2(y+y')}{y+y'}p^{s+t-r,x',x}(y+y')-p^{s+t-r,x',x}(y)w^2(y))dx\\
&\leq C\int_{\mathbb{R}^d} \int_0^\infty\int_0^z K(y,y')q_r(dy)\nu_r(dy')(w(y)w(y')p^{s+t-r,x',x}(y)+w(y)w(y')p^{s+t-r,x',x}(y'))\\
&\leq w(y)w(y')q_r(dy)\nu_r(dy')(w(y)w(y')p^{s+t-r,x',x}(y)+w(y)w(y')p^{s+t-r,x',x}(y'))\\
&\leq C[\langle w^2,P_{s+t-r}\nu_r\rangle\|\langle \mathbf{1}_{y<z}w^2,q_r\rangle\|_\infty+\langle \mathbf{1}_{y<z}w^2,P_{s+t-r}q_r\rangle\|\langle w^2,\nu_r\rangle\|_\infty].
\end{align*}
So we have
\begin{displaymath}
h(t)\leq \|\langle w^2,q_0\rangle\|_{\infty}+2\int_0^tcCh(s)ds,
\end{displaymath}
for all $t\leq T$. Then we can use Gronwall's inequality to obtain
\begin{displaymath}
h(t)\leq \|\langle w^2,q_0\rangle\|_{\infty}e^{2cCt}.
\end{displaymath}
So, we can pick $c$ large and $T$ small such that
\begin{displaymath}
c\geq\|\langle w^2,q_0\rangle\|_{\infty}e^{2cCT}.
\end{displaymath}
As this works for all $z>0$, we conclude that if we have
\begin{displaymath}
\|\langle w^2,P_s\mu_t\rangle\|_{\infty}\leq c
\end{displaymath}
for all $s$ and $t\leq T$, then we also have
\begin{displaymath}
\|\langle w^2,P_sf(\mu)_t\rangle\|_{\infty}\leq c
\end{displaymath}
for all $s$ and $t\leq T$.
\end{proof}

Now, we will modify our argument for proving Theorem \ref{ssthm1} to prove Theorem \ref{ssthm2}. 
\begin{proof}
Suppose
\begin{displaymath}
\|\langle w^2,P_s\mu^i_t\rangle\|_{\infty}\leq c
\end{displaymath}
for all $s\geq 0$ and $t\leq T$.

Let $q^i=f(\mu^i)$ and let $\Phi^i_{s,t}$ be the map mapping $q^i_s$ to $q^i_t$ as defined earlier. Also let $K^i=K^{\mu^i}$. Then by variation of constants formula, we have
\begin{displaymath}
q^2_t-q^1_t=\int_0^t\Phi^1_{s,t}(K^2_s-K^1_s)(q^2_s)ds.
\end{displaymath}
Further, we would have
\begin{align*}
&\|\langle w,|(K^2_s-K^1_s)(q^2_s)|ds\rangle\|_1\\&\leq \int_\mathbb{R^d}\int_0^\infty\int_0^\infty|(\mu^2_s-\mu^1_s)(dy)|q^2_s(dy')K(y,y')(w(y')+\frac{y'}{y+y'}w(y+y'))\\
&\leq 2\int_\mathbb{R^d}\int_0^\infty\int_0^\infty|(\mu^2_s-\mu^1_s)(dy)|q^2_s(dy')K(y,y')w(y')\\
&\leq 2\int_\mathbb{R^d}\int_0^\infty\int_0^\infty|(\mu^2_s-\mu^1_s)(dy)|q^2_s(dy')w(y)w(y')^2\\
&\leq 2\|\langle w,|\mu^2_s-\mu^1_s|\rangle\|_1\|\langle w^2,q^2_s\rangle\|_\infty\\
&\leq 2c\|\langle w,|\mu^2_s-\mu^1_s|\rangle\|_1.
\end{align*}
Therefore, we obtain
\begin{displaymath}
\|\langle w,|q^2_t-q^1_t|\rangle\|_1\leq 2c\int_0^t\|\langle w,|\mu^2_s-\mu^1_s|\rangle\|_1ds.
\end{displaymath}
So, for $T$ sufficiently small, we would have $f$ is a contraction with respect to the metric $d_T(\mu^1,\mu^2)=\sup_{s\leq T}\|\langle w,|\mu^1-\mu^2|_s\rangle\|_1$ in the space of kernels $\mu$ with $
\sup_{t\leq T,s\geq 0}\|\langle w^2,P_s\mu_t\rangle\|_{\infty}\leq c
$.
By contraction mapping theorem, $f$ must have a fixed point and that fixed point is the solution we want. This ends the proof of Theorem \ref{ssthm2}.
\end{proof}
Now, we can modify the argument to prove Theorem \ref{ssthm3}
\begin{proof}
First, we assume (a) in Theorem \ref{ssthm3}, then we have
\begin{displaymath}
\langle wv,P_sq_t\rangle\leq \int_{\mathbb{R}^d}\int_0^R\int_0^R
\langle wv,P_{s+t}(q_0)\rangle+\int_0^t\langle wv,P_{s+t-r}K^\nu_r(q_r)\rangle dr.
\end{displaymath}
Note that, when $wvp$ is sublinear, $\langle wv,P_{s+t-r}K^\nu_r(q_r)\rangle$ is non-positive. So, we have
\begin{displaymath}
\langle wv,P_sq_t\rangle\leq \int_{\mathbb{R}^d}\int_0^R\int_0^R
\langle wv,P_{s+t}(q_0)\rangle\leq\sup_{t>0}\|\langle w^2,P_t(q_0)\rangle\|_\infty\leq c,
\end{displaymath}
for some constant $c$. Moreover, the inequality \begin{align*}
&\langle w^2,P_{s+t-r}K^\nu_r(q_r)\rangle\\
&\leq C[\langle w^2,P_{s+t-r}\nu_r\rangle\|\langle w^2,\mu_r\rangle\|_\infty+\langle w^2,P_{s+t-r}q_r\rangle\|\langle w^2,\nu_r\rangle\|_\infty]
\end{align*}
becomes now
\begin{align*}
&\langle w^2,P_{s+t-r}K^\nu_r(q_r)\rangle\\
&\leq C[\langle wv,P_{s+t-r}\nu_r\rangle\|\langle w^2,q_r\rangle\|_\infty+\langle wv,P_{s+t-r}q_r\rangle\|\langle w^2,\nu_r\rangle\|_\infty\\
&+\langle w^2,P_{s+t-r}\nu_r\rangle\|\langle wv,q_r\rangle\|_\infty+\langle w^2,P_{s+t-r}q_r\rangle\|\langle wv,\nu_r\rangle\|_\infty].
\end{align*}
Now, for constants $a,b>0$, if we have
\begin{displaymath}
\|\langle w^2,P_s\nu_t\rangle\|_{\infty}\leq ae^{bt}
\end{displaymath}
for all $s$ and $t$ then we also have
\begin{displaymath}
h(t)\leq c+2Cc\int_0^t(h_s+ae^{bs})ds\leq c+2\frac{aCc}{b}e^{bt}+2Cc\int_0^th_sds.
\end{displaymath}
By Grownwall's inequality, we have
\begin{displaymath}
h(t)\leq [c+2\frac{aCc}{b}e^{bt}]e^{2cCt}.
\end{displaymath}
So, for any $T>0$, we can choose $b$ to be sufficiently large such that if
\begin{displaymath}
\|\langle w^2,P_s\nu_t\rangle\|_{\infty}\leq ae^{bt}
\end{displaymath}
for all $t\leq T$, then
\begin{displaymath}
h(t)\leq ae^{bt}
\end{displaymath}
for all $t\leq T$.
Then we use the same argument as earlier to show that there is some $T'$ such that $f$ is a contraction with respect to $d_{T'}$ whenever $\sup_{s>0}\|\langle w^2,P_s\mu_t\rangle\|_{\infty}\leq ae^{bt}$ for all $0\leq t\leq T'$. So, we have existence of the solution up to time $T'$ and then by the same argument, we can extend the solution to $2T'$ and so on up to time $T$. As this works for any $T$, we have a global solution.

Now, we assume instead $(b)$ in Theorem \ref{ssthm3}. Recall
\begin{displaymath}
\langle w^2,P_sq_t\rangle\leq
\langle w^2,P_{s+t}(q_0)\rangle+\int_0^t\langle w^2,P_{s+t-r}K^\nu_r(q_r)\rangle dr,
\end{displaymath}
and
\begin{displaymath}
\langle w^2,P_{s+t-r}K^\nu_r(q_r)\rangle\leq C[\langle w^2,P_{s+t-r}\nu_r\rangle\|\langle w^2,q_r\rangle\|_\infty+\langle w^2,P_{s+t-r}q_r\rangle\|\langle w^2,\nu_r\rangle\|_\infty].
\end{displaymath}
Suppose $\sup_{s\geq 0}(1+s+t)^{1+\epsilon}\|\langle\mathbf{1}_{y\leq z}w^2,P_s\nu_t\rangle\|_\infty<2c$ and set 
\begin{displaymath}
h(t)=\sup_{s\geq 0}(1+s+t)^{1+\epsilon}\|\langle\mathbf{1}_{y\leq z}w^2,P_sq_t\rangle\|_\infty,
\end{displaymath} 
we obtain
\begin{displaymath}
h(t)\leq c+C\int_0^t\frac{4ch(r)}{(1+r)^{1+\epsilon}}dr.
\end{displaymath}
So, if $c$ is small enough such that $8c^2Ch\int_0^\infty\frac{1}{(1+r)^{1\epsilon}}dr<c$, then we have $h(t)<2c$ for all $t\geq 0$. By similar argument as earlier, we obtain global existence of the solution.
\end{proof}
\section{Well-posedness of the Smoluchowski coagulation equations with a drift term}
\subsection{Comparison with the case without the drift term}
In the previous two sections, we have seen how to establish well-posedness of Smoluchowski coagulation equations
\begin{displaymath}
\mu_t+\int_0^t P_{t-s}K^-(\mu_s)ds=P_t\mu_0+\int_0^t P_{t-s}K^+(\mu_s)ds,
\end{displaymath}
where 
\begin{displaymath}
P_t\mu(x,dy)=\int_{\mathbb{R}^d}\mu(x',dy)p^{t,x',x}(y)dx'.
\end{displaymath}
Here, $p^{t,x',x}(y)$ was defined to be the transition density of a Brownian particle with diffusivity $a(y)$. A natural question to ask would be what if $p$ is instead the transition density of a Brownian particle with a space dependent drift. More precisely, consider a particle whose free motion satisfies $X_0=x'$ and
\begin{displaymath}
dX_t=\sqrt{a(y)}dB_t+b(x,y)dt,
\end{displaymath}
with $b$ bounded and measurable in $x$, then we let $p^{t,x',x}(y)$ denote the probability density function of $X_t$ evaluated at $x$. In this section, we will investigate the well-posedness of (\ref{eqn}) in this case.

The key difference between the case without drift and the case with drift is that in the case without drift we know $p$ explicitly. We note that in the proof of Theorem \ref{ssthm1}, Theorem \ref{ssthm2} and Theorem \ref{ssthm3}, we did not use the explicit form of $p$, which means that the proofs also work in the case there is a drift. While Theorem \ref{ssthm1} is still a useful result about the uniqueness of the Smoluchowski PDEs, Theorem \ref{ssthm2} and Theorem \ref{ssthm3} can hardly be used because the conditions in these two theorems are usually not satisfied or hard to verify. For the case without drift, we have seen that there is a strong link between the well-posedness of the PDEs and the a priori estimates of the norm $\|\langle w^2,\mu\rangle\|_\infty$, and this link still exists for the case with a drift term. This link will be the key starting point of this section.
\subsection{Proof of Theorem \ref{driftcase2}}
The uniqueness part of the Theorem follows directly from Theorem \ref{driftcase}. For the existence part, we will try to modify the strategies we used in the last section. We will continue to use the same notations as in the last section. We first show an analogy of Proposition \ref{propbrownian}.
\begin{prop}
Assume
\begin{displaymath}
\frac{y}{y+y'}w^2(y+y')q(y+y')-w^2(y)q(y)\leq C[w(y)w(y')q(y)+w(y)w(y')q(y')],
\end{displaymath}
then we can find $c>0$ and $T>0$ so that \begin{displaymath}
\sup_{t\leq T}\sup_{s\geq 0}\|\langle w^2,Q_sf(\nu)_t\rangle\|_\infty\leq c\end{displaymath} whenever \begin{displaymath}
\sup_{t\leq T}\sup_{s\geq 0}\|\langle w^2,Q_s\nu_t\rangle\|_\infty\leq c.\end{displaymath}
\end{prop}
\begin{proof}
In the Brownian case, we had
\begin{displaymath}
P_sq_t=P_{s+t}(q_0)+\int_0^tP_{s+t-r}K^\nu_r(q_r)dr.
\end{displaymath}
As an analogy, we will show
\begin{displaymath}
Q_sq_t\leq Q_{s+t}(q_0)+\int_0^tQ_{s+t-r}K^\nu_r(q_r)dr.
\end{displaymath}
We will look at the evolution of $Q_{s}q_t$ keeping $s+t$ fixed. Consider $Q_{s-h}q_{t+h}-Q_sq_t$ for sufficiently small $h>0$. We know that
\begin{displaymath}
q_{t+h}=P_hq_t+\int_{0}^{h}P_{h-r}K^\nu_{t+r}(q_{t+r})dr
\end{displaymath} 
and thus
\begin{displaymath}
Q_{s-h}q_{t+h}=Q_{s-h}P_hq_t+\int_{0}^{h}Q_{s-h}P_{h-r}K^\nu_{t+r}(q_{t+r})dr.
\end{displaymath} 
By optimality result in Theorem \ref{drift}, we know that $Q_{s-h}P_{h}q_t\leq Q_{s}q_t$ and therefore,
\begin{displaymath}
Q_{s-h}q_{t+h}-Q_sq_t\leq \int_{0}^{h}Q_{s-h}P_{h-r}K^\nu_{t+r}(q_{t+r})dr.
\end{displaymath} 
Now, we want to approximate $Q_{s-h}P_{h-r}K^\nu_{t+r}(q_{t+r})$ with $Q_{s-r}K^\nu_{t+r}(q_{t+r})$.
When we showed H\"older continuity of $\rho$ in Theorem \ref{drift}, we actually showed that for $t_1>0$, we can find some constant $C>0$, such that for all $t_1<t_2\leq t_1+h$,
\begin{displaymath}
|\rho(t_2,y)-\rho(t_1,y)|\leq Ch^{1/8}\rho(t_2,y)+h.
\end{displaymath}
This essentially implies that 
\begin{displaymath}
|Q_{s-h}P_{h-r}K^{\nu_r}(q_r)-Q_{s-r}K^\nu_{t+r}(q_{t+r})|\leq Ch^{1/8}K^\nu_{t+r}(q_{t+r})+h\|K^\nu_{t+r}(q_{t+r})\|_1.
\end{displaymath}
Because we have assumed 
\begin{displaymath}
\sup_{t\leq T}\sup_{s\geq 0}\|\langle w^2,Q_s\nu_t\rangle\|_\infty\leq c,
\end{displaymath}
we know that $K^\nu_{t+r}(q_{t+r})$ and $\|K^\nu_{t+r}(q_{t+r})\|_1$ are both bounded. Therefore, we have
\begin{displaymath}
Q_{s-h}q_{t+h}-Q_sq_t\leq \int_{0}^{h}Q_{s-r}K^\nu_{t+r}(q_{t+r})dr+Ch^{9/8}.
\end{displaymath} 
Now, let $h=t/n$ for some sufficiently large integer $n$ and for $0\leq m<n$, we have
\begin{displaymath}
Q_{s+mh}q_{t-mh}-Q_{s+(m+1)h}q_{t-(m+1)h}\leq \int_{0}^{h}Q_{s+(m+1)h-r}K^\nu_{t+r}(q_{t-(m+1)h+r})dr+Ch^{9/8}.
\end{displaymath} 
Summing over all $m$ and let $h\rightarrow 0$, we conclude 
\begin{displaymath}
Q_sq_t\leq Q_{s+t}(q_0)+\int_0^tQ_{s+t-r}K^\nu_r(q_r)dr.
\end{displaymath}
To complete the rest of the proof, we literally only need to change $P$ into $Q$ in the proof of Proposition \ref{propbrownian}. For any $z>0$, Let
\begin{displaymath}
h(t)=\sup_{s\geq 0}\|\langle \mathbf{1}_{y<z}w^2,Q_sq_t\rangle\|_\infty.
\end{displaymath}
Then we have
\begin{displaymath}
\langle \mathbf{1}_{y<z}w^2,Q_sq_t\rangle\leq
\langle \mathbf{1}_{y<z}w^2,Q_{s+t}(q_0)\rangle+\int_0^t\langle w^2,Q_{s+t-r}K^\nu_r(q_r)\rangle dr
\end{displaymath}
and
\begin{align*}
&\langle \mathbf{1}_{y<z}w^2,Q_{s+t-r}K^\nu_r(q_r)\rangle\\
&\leq\int_{\mathbb{R}^d} \int_0^\infty\int_0^z K(y,y')q_r(dy)\nu_r(dy')(\frac{yw^2(y+y')}{y+y'}q^{s+t-r,x',x}(y+y')-q^{s+t-r,x',x}(y)w^2(y))dx\\
&\leq C\int_{\mathbb{R}^d} \int_0^\infty\int_0^z K(y,y')q_r(dy)\nu_r(dy')(w(y)w(y')q^{s+t-r,x',x}(y)+w(y)w(y')q^{s+t-r,x',x}(y'))\\
&\leq w(y)w(y')q_r(dy)\nu_r(dy')(w(y)w(y')q^{s+t-r,x',x}(y)+w(y)w(y')q^{s+t-r,x',x}(y'))\\
&\leq C[\langle w^2,Q_{s+t-r}\nu_r\rangle\|\langle \mathbf{1}_{y<z}w^2,q_r\rangle\|_\infty+\langle \mathbf{1}_{y<z}w^2,Q_{s+t-r}q_r\rangle\|\langle w^2,\nu_r\rangle\|_\infty].
\end{align*}
So, we have
\begin{displaymath}
h(t)\leq \|\langle w^2,q_0\rangle\|_{\infty}+2\int_0^tcCh(s)ds,
\end{displaymath}
for all $t\leq T$. Then we can use Gronwall's inequality to obtain
\begin{displaymath}
h(t)\leq \|\langle w^2,q_0\rangle\|_{\infty}e^{2cCt}.
\end{displaymath}
So, we can pick $c$ large and $T$ small such that
\begin{displaymath}
c\geq\|\langle w^2,q_0\rangle\|_{\infty}e^{2cCT}.
\end{displaymath}
As this works for all $z>0$, we conclude that if we have
\begin{displaymath}
\|\langle w^2,Q_s\mu_t\rangle\|_{\infty}\leq c
\end{displaymath}
for all $s$ and $t\leq T$, then we also have
\begin{displaymath}
\|\langle w^2,Q_sf(\mu)_t\rangle\|_{\infty}\leq c
\end{displaymath}
for all $s$ and $t\leq T$.
\end{proof}
The rest of the proof of Theorem \ref{driftcase2} will be similar as in the Brownian case too.
Suppose
\begin{displaymath}
\|\langle w^2,Q_s\mu^i_t\rangle\|_{\infty}\leq c
\end{displaymath}
for all $s\geq 0$ and $t\leq T$.

For $i=1,2$, let $q^i=f(\mu^i)$ and let $\Phi^i_{s,t}$ be the map mapping $q^i_s$ to $q^i_t$ as defined earlier. Also let $K^i=K^{\mu^i}$. Then by variation of constants formula we have
\begin{displaymath}
q^2_t-q^1_t=\int_0^t\Phi^1_{s,t}(K^2_s-K^1_s)(q^2_s)ds.
\end{displaymath}
Further we would have
\begin{align*}
&\|\langle w,|(K^2_s-K^1_s)(q^2_s)|ds\rangle\|_1\\
&\leq 2\|\langle w,|\mu^2_s-\mu^1_s|\rangle\|_1\|\langle w^2,q^2_s\rangle\|_\infty\\
&\leq 2c\|\langle w,|\mu^2_s-\mu^1_s|\rangle\|_1.
\end{align*}
Therefore, we have
\begin{displaymath}
\|\langle w,|q^2_t-q^1_t|\rangle\|_1\leq 2c\int_0^t\|\langle w,|\mu^2_s-\mu^1_s|\rangle\|_1ds.
\end{displaymath}
So, for $T$ sufficiently small, we would have $f$ is a contraction with respect to the metric $d_T(\mu^1,\mu^2)=\sup_{s\leq T}\|\langle w,|\mu^1-\mu^2|_s\rangle\|_1$ in the space of kernels $\mu$ with $
\sup_{t\leq T,s\geq 0}\|\langle w^2,Q_s\mu_t\rangle\|_{\infty}\leq c
$.
By contraction mapping theorem, $f$ must have a fixed point and that fixed point is the solution we want. 
In the case when $K(y,y')\leq w(y)v(y')+w(y')v(y)$ for some $v$ such that $wvq$ is sublinear, the proof is again similar as in the Brownian case. 
\subsection{Proof of lemma \ref{driftlem}}
In this part, we will look at some properties of $q$. We know $q$ explicitly from Theorem \ref{drift}:
\begin{displaymath}
q^{t,x,x'}(y)= \frac{1}{(2\pi a(y)t)^{d/2}}\prod^d_{i=1}\int_{|x^i-x'^i|/\sqrt{a(y)t}}^\infty ze^{-(z-B(y)\sqrt{t/a(y)})^2/2}dz,
\end{displaymath}
and we can rewrite it as 
\begin{displaymath}
q^{t,x,x'}(y)= \frac{1}{(2\pi a(y)t)^{d/2}}\prod^d_{i=1}\int_{|x^i-x'^i|/\sqrt{a(y)t}-B(y)\sqrt{t/a(y)}}^\infty \big(z+B(y)\sqrt{t/a(y)}\big)e^{-z^2/2}dz.
\end{displaymath}
If $B/\sqrt{a}$ is non-increasing and $a$ is non-increasing, then the integrand will also be non-increasing. So, we conclude that for $y>y'>0$,
\begin{displaymath}
q(y)/q(y')\leq [a(y)/a(y')]^{-\frac{d}{2}}.
\end{displaymath}
When $a$ is non increasing, $B$ is non-increasing and $B/\sqrt{a}$ is non-decreasing, we have for $y>y'$, if 
\begin{displaymath}
|x^i-x'^i|/\sqrt{a(y)t}-B(y)\sqrt{t/a(y)}> 0,
\end{displaymath}
then
\begin{align*}
|x^i-x'^i|/\sqrt{a(y')t}-B(y')\sqrt{t/a(y')}&=\frac{|x^i-x'^i|/\sqrt{t}-B(y')\sqrt{t}}{\sqrt{a(y')}}\\
&\leq \frac{|x^i-x'^i|/\sqrt{t}-B(y)\sqrt{t}}{\sqrt{a(y)}}\\&=|x^i-x'^i|/\sqrt{a(y)t}-B(y)\sqrt{t/a(y)}.
\end{align*}
Note that $\frac{z+B(y)\sqrt{t/a(y)}}{z+B(y')\sqrt{t/a(y')}}$ is non-increasing in $z\geq 0$, we have
\begin{align*}
&\frac{\int_{|x^i-x'^i|/\sqrt{a(y)t}-B(y)\sqrt{t/a(y)}}^\infty (z+B(y)\sqrt{t/a(y)})e^{-z^2/2}dz}{\int_{|x^i-x'^i|/\sqrt{a(y')t}-B(y')\sqrt{t/a(y')}}^\infty (z+B(y')\sqrt{t/a(y')})e^{-z^2/2}dz}\\
&\leq \frac{\int_{|x^i-x'^i|/\sqrt{a(y)t}-B(y)\sqrt{t/a(y)}}^\infty (z+B(y)\sqrt{t/a(y)})e^{-z^2/2}dz}{\int_{|x^i-x'^i|/\sqrt{a(y)t}-B(y')\sqrt{t/a(y)}}^\infty (z+B(y')\sqrt{t/a(y')})e^{-z^2/2}dz}\\
&\leq\frac{\int_0^\infty (z+B(y)\sqrt{t/a(y)})e^{-z^2/2}dz}{\int_0^\infty (z+B(y')\sqrt{t/a(y')})e^{-z^2/2}dz}\\
&\leq \frac{B(y)\sqrt{1/a(y)}}{B(y')\sqrt{1/a(y')}}.
\end{align*}
If
\begin{displaymath}
|x^i-x'^i|/\sqrt{a(y)t}-B(y)\sqrt{t/a(y)}\leq 0,
\end{displaymath}
we would have
\begin{align*}
&\frac{\int_{|x^i-x'^i|/\sqrt{a(y)t}-B(y)\sqrt{t/a(y)}}^\infty (z+B(y)\sqrt{t/a(y)})e^{-z^2/2}dz}{\int_{|x^i-x'^i|/\sqrt{a(y')t}-B(y')\sqrt{t/a(y')}}^\infty (z+B(y')\sqrt{t/a(y')})e^{-z^2/2}dz}\\
&\leq \frac{\int_{|x^i-x'^i|/\sqrt{a(y)t}-B(y)\sqrt{t/a(y)}}^\infty (z+B(y)\sqrt{t/a(y)})e^{-z^2/2}dz}{\int_{|x^i-x'^i|/\sqrt{a(y)t}-B(y')\sqrt{t/a(y')}}^\infty (z+B(y')\sqrt{t/a(y')})e^{-z^2/2}dz}\\
&\leq \frac{\int_{-B(y)\sqrt{t/a(y)}}^\infty (z+B(y)\sqrt{t/a(y)})e^{-z^2/2}dz}{\int_{-B(y')\sqrt{t/a(y')}}^\infty (z+B(y')\sqrt{t/a(y')})e^{-z^2/2}dz}.
\end{align*}
Substitute $u=B(y')\sqrt{t/a(y')}$ and let $f(u)=\int_{-u}^\infty(z+u)e^{-z^2/2}dz$, we have $\frac{d(f/u)}{du}<0$, so we can again conclude
\begin{displaymath}
\frac{\int_{|x^i-x'^i|/\sqrt{a(y)t}-B(y)\sqrt{t/a(y)}}^\infty (z+B(y)\sqrt{t/a(y)})e^{-z^2/2}dz}{\int_{|x^i-x'^i|/\sqrt{a(y')t}-B(y')\sqrt{t/a(y')}}^\infty (z+B(y')\sqrt{t/a(y')})e^{-z^2/2}dz}\leq \frac{B(y)\sqrt{1/a(y)}}{B(y')\sqrt{1/a(y')}}.
\end{displaymath}.
So, we have  $q_{x'}^{t,x,x'}(y)/q_{x'}^{t,x,x'}(y')\leq (\frac{B(y)/a(y)}{B(y')/a(y')})^d$ as required.
\section{Acknowledgement}
I would like to express my gratitude to my supervisor, Prof. J. Norris, for suggesting the problem and for the insightful discussions and comments on the topic. My thanks also go to my fellow students for providing an atmosphere in which it is a pleasure to work. I am supported jointly by UK Engineering and Physical Sciences Research Council (EPSRC) grant.

\newpage
\nocite{*} 
\bibliographystyle{plain}

\begin{thebibliography}{1}

\bibitem{coafra}
Herbert Amann.
\newblock Coagulation-fragmentation processes.
\newblock {\em Arch. Ration. Mech. Anal}, 151(2000), no.4, 339-366.

\bibitem{locglob}
Herbert Amann and Christoph Walker.
\newblock Local and global strong solutions to continuous coagulation-fragmentation equations with diffusion.
\newblock{\em J. Differential Equations}, 218(2005), no.1, 159-186.\bibitem{Davies}
E.B. Davies.
\newblock {\em Heat Kernels and Spectral Theory}.
\newblock Cambridge Tracts in Mathematics. Cambridge University Press, 1990.
\bibitem{bai}
I. F. Bailleul.
\newblock{\em Spatial coagulation with bounded coagulation rate.}
\newblock J. Evol. Equ., 11(2011), no. 3, 675-686.\bibitem{Carr}
J. M. Ball and J. Carr.
\newblock{\em The discrete coagulation-fragmentation equations: existence, uniqueness, and density conservation}.
\newblock J. Statist. Phys., 61(1990), no. 1-2, 203-234.\bibitem{reg}
J.A. Ca\~{n}izo, L. Desvillettes, K. Fellner.
\newblock Regularity and mass conservation for discrete coagulation-fragmentation equations with diffusion.
\newblock{\em Ann. Inst. H. Poincar\'{e} Anal. Non Lin\'{e}aire}, 27(2010), no. 2, 639-654.
\bibitem{mechanical}
D. D\"{u}rr, S. Goldstein and J.L. Lebowitz
\newblock A mechanical model of Brownian motion
\newblock {\em Communications in Mathematical Physics}, 78(1980/81),no.4, 507-530.\bibitem{fabes}
E.B. Fabes and D.W. Stroock.
\newblock {\em A New Proof of Moser's Parabolic Harnack Inequality Via the Old
  Ideas of Nash}.
\newblock LIDS-P. Laboratory for Information and Decision Systems,
  Massachusetts Institute of Technology, 1986.
  \bibitem{convection}
A. Fannjiang and G. Papanicolaou.
\newblock {\em Convection enhanced diffusion for periodic flows}.
\newblock SIAM J. Appl. Math. 54 (1994), no. 2, 333-408. 
\bibitem{moment}
{Hammond, Alan and Rezakhanlou, Fraydoun}.
\newblock{\em Moment bounds for the {S}moluchowski equation and their 
consequences}.
\newblock Comm. Math. Phys. 276 (2007), no. 3, 645-670.
\bibitem{HR}
{Hammond, Alan and Rezakhanlou, Fraydoun}.
\newblock{\em The kinetic limit of a system of coagulating Brownian particles}.
\newblock Arch. Ration. Mech. Anal. 185 (2007), no. 1, 1-67.
\bibitem{homogenization}
P.H. Haynes, V. H. Hoang, J. R. Norris and K. C. Zygalakis.
\newblock {\em Homogenization for advection-diffusion in a perforated domain}.
\newblock  Probability and mathematical genetics, 397-415, London Math. Soc. Lecture Note Ser., 378, Cambridge Univ. Press, Cambridge, 2010. 
\bibitem{JKO}
V. V. Jikov, S. M. Kozlov and O. A. Oleinik
\newblock{\em Homogenization of Differential Operators and Integral Functionals}.
\newblock Springer, Berlin, 1994.\bibitem{disc}
Lauren{\c{c}}ot, P. and Mischler, S.
\newblock {\em Global existence for the discrete diffusive coagulation-fragmentation equations in $\mathbf{L}^1$}.
\newblock Rev. Mat. Iberoamericana, 18 (2002), no. 3, 731-745.\bibitem{continuous}
Lauren{\c{c}}ot, P. and Mischler, S.
\newblock {\em The continuous coagulation-fragmentation
equations with diffusion}.
\newblock Archive for rational mechanics and analysis 162 (2002), no. 1, 45-99.
\bibitem{conditional}
T.J. Lyons and W.A. Zheng.
\newblock On conditional diffusion processes.
\newblock {\em Proceedings of the Royal Society of Edinburgh: Section A
  Mathematics}, 115:243--255, 1990.
  
\bibitem{mono}
Mischler, S and Rodriguez Richard, M.
\newblock Existence globale pour l'\'{e}quation de Smoluchowski continue non homog\`{e}ne et comportement asymptotique des solutions.
\newblock{\em C. R. Math. Acad. Sci. Paris}, 365(2003), no. 5, 407-412. \bibitem{Norris}
J.R. Norris.
\newblock Long-time behaviour of heat flow: Global estimates and exact
  asymptotics.
\newblock {\em Archive for Rational Mechanics and Analysis}, 140:161--195,
  1997.
\bibitem{cluster}
J. R. Norris.
\newblock{\em Cluster coagulation}.
\newblock Comm. Math. Phys., 209(2000), no. 2, 407-435.\bibitem{bmk}
J.R.Norris.
\newblock Brownian Coagulation.
\newblock {\em Comm. Math. Sci. Supplemental Issue}, No. 1, 93-101, 2004
\bibitem{BM}
J.R. Norris.
\newblock Measure solutions for the Smoluchowski coagulation-diffusion equation.
\newblock 2014.
\bibitem{th}
J.R. Norris
\newblock {\em Smoluchowski's coagulation equation: uniqueness, non-uniqueness and a hydrodynamic limit for the stochastic coalescent}
\newblock {\em Ann. Appl. Prob.} No. 9, 78-109, 1999
\bibitem{N13}
G. C. Papanicolaou and S. R. S. Varadhan
\newblock{\em Boundary value problems with rapidly oscillating random coefficients}.
\newblock Colloq. Math. Soc. J\'anos Bolyai, 27(1981)

\bibitem{ed}
G.A. Pavliotis, A.M. Stuart, K.C. Zygalakis
\newblock{\em Calculating effective diffusivities in the limit of vanishing molecular diffusion}.
\newblock Journal of Computational Physics, 228(2009), no. 4, 1030-1055.
\bibitem{sharp}
Zhongmin Qian and Weian Zheng.
\newblock Sharp bounds for transition probability densities of a class of
  diffusions.
\newblock {\em Comptes Rendus Mathematique}, 335(11):953 -- 957, 2002.
\bibitem{fra}
Fraydoun Rezakhanlou.
\newblock{\em Moment bounds for the solutions of the Smoluchowski equation with coagulation and fragmentation}.
\newblock Proc. Roy. Soc. Edinburgh Sec. A, 140(2010), no. 5, 1041-059.
\bibitem{rez}
Fraydoun Rezakhanlou.
\newblock{\em Pointwise bounds for the solutions of the Smoluchowski equation with diffusion}.
\newblock Arch. Ration. Mech. Anal., 212(2014), no. 3, 1011-1035.
\bibitem{wro1}
Dariusz Wrzosek.
\newblock{\em Existence of solutions for the discrete coagulation-fragmentation model with diffusion}. 
\newblock Topol. Methods Nonlinear Anal., 9(1997), no. 2, 279-296.
\bibitem{wro3}
Dariusz Wrzosek.
\newblock{\em Mass-conserving solutions to the discrete coagulation-fragmentation model with diffusion}.
\newblock Nonlinear Anal. 49(2002), no. 3, Ser. A: Theory Methods, 297-314.\bibitem{wro2}
Dariusz Wrzosek.
\newblock{\em Weak solutions to the Cauchy problem for the diffusive discrete coagulation-fragmentation system}. 
\newblock J. Math. Anal. Appl., 289(2004), no. 2, 405-418.
\bibitem{HR1}
{Mohammad Yaghouti, Alan Hammond and Fraydoun Rezakhanlou}.
\newblock{\em Coagulation, diffusion and the continuous Smoluchowski equation}.
\newblock Stochastic Process. Appl., 119 (2009), no. 9, 3042-3080.
\
\end{thebibliography}

\end{document}